\documentclass[a4,twoside, 10pt]{article}

\usepackage[a4paper, textwidth=15.75 cm, textheight=23.4cm, marginratio={5:5,5:6},marginparwidth=2.1cm]{geometry}
 \usepackage{amsmath,amsthm,amssymb,mathrsfs,mathtools,dsfont}
\usepackage[shortlabels]{enumitem}
\usepackage{subdepth}
\usepackage{tikz-cd}
\usepackage{hyperref}
\usepackage{stmaryrd}
\usepackage[colorinlistoftodos]{todonotes}
\usepackage{titlesec}
\usepackage{relsize}
\usepackage{floatrow}
\usepackage{placeins}

\newcommand{\EG}{\mathbf{E}}
\newcommand{\BG}{\mathbf{B}}

\newcommand{\MMt}{\mathds{M}^{C_2}_2} 
\newcommand{\MMtn}{\mathds{M}^{C_2}_n}

\newcommand{\mytitle}{$RO(C_2\times C_2)$-graded cohomology ring of a point and applications}

\title{\mytitle}
\author{  
Bill Deng,
Mircea Voineagu}
\date{}

\allowdisplaybreaks

\tikzset{close/.style={near start,outer sep=2pt}}

\newcommand{\R}{\mathbb{R}}

\newcommand{\F}{\mathbb{F}}

\newcommand{\ul}{\underline}
\newcommand{\rp}{\mathbb{RP}}
\newcommand{\ot}{\otimes}
\newcommand{\al}{\alpha}
\newcommand{\ga}{\gamma}
\newcommand{\op}{\oplus}
\newcommand{\Z}{\mathbb{Z}}
\newcommand{\hr}{\tilde{H}}
\newcommand{\susp}{\Sigma}
\newcommand{\si}{\sigma}
\newcommand{\Si}{\Sigma}
\newcommand{\sd}{_{\Si_2}}

\newcommand{\ep}{\varepsilon}
\newcommand{\tl}{\tilde}
\newcommand{\st}{\star}
\newcommand{\dl}{\delta}
\newcommand{\so}{{\si\ot\ep}}

\newcommand{\xr}{\xrightarrow}
\newcommand{\xl}{\xleftarrow}

\newcommand{\spec}{\mathrm{Spec}}

\newcommand{\ti}{\times}
\newcommand{\te}{\tl{E}}
\newcommand{\norm}[1]{\left\lVert#1\right\rVert}

\newcommand{\xx}{\text}
\newcommand{\ct}{{C_2}}
\newcommand{\hc}{H^}
\newcommand{\tc}{\tl{H}^}
\newcommand{\hs}{\hc\st}
\newcommand{\fc}{\frac}
\newcommand{\ta}{\theta}

\newcommand{\tx}{\quad\text}

\newcommand{\txa}{\quad\text{and}\quad}
\theoremstyle{definition}
\newtheorem{theorem}{Theorem}
\numberwithin{theorem}{subsection}
\newtheorem{corollary}[theorem]{Corollary}
\newtheorem{lemma}[theorem]{Lemma}
\newtheorem{proposition}[theorem]{Proposition}

\newtheorem{remark}[theorem]{Remark}

\newcommand{\ec}{{{E_{\Si_2}C_2}_+}}

\newcommand{\ek}{{E\sd\ct}}
\newcommand{\et}{{\te\sd\ct}}

\newcommand{\bc}{{B\sd\ct}}

\newcommand{\hra}{\hookrightarrow}

\newcommand{\cd}{\cdots}
\newcommand{\ld}{\ldots}

\newcommand{\sss}{\subsubsection}

\newcommand{\sh}{\wedge}

\newcommand{\se}{\simeq}
\newcommand{\ka}{\kappa}
\newcommand{\be}{\beta}
\newcommand{\io}{\iota}
\newcommand{\ic}{\mathbf{1}}
\newcommand{\ub}{\underbrace}
\newcommand{\ob}{\overbrace}

\newcommand{\cp}{\cdot}
\newcommand{\mt}{\mapsto}

\newcommand{\rH}{\widetilde{H}}
\newcommand{\wt}[1]{\widetilde{#1}}

\newcommand{\Br}{{Br}}
\DeclareMathOperator*{\colim}{\mathrm{colim}}

	\theoremstyle{plain}

	\newcommand\im{\mathrm{Im}}

	\tikzset{
		labl/.style={anchor=south, rotate=90, inner sep=.5mm}
	}

	\titleformat*{\section}{\large\bfseries}
	\titleformat*{\subsection}{\normalsize\bfseries}
	\titleformat*{\subsubsection}{\small\bfseries}

\begin{document} 
  \maketitle
  \begin{abstract}
We describe the main properties of the $RO(C_2\times \Sigma_2)$-graded cohomology ring of a point and apply the results to compute the subring of motivic classes given by the Bredon motivic cohomology of real numbers and to compute $RO(C_2\times \Sigma_2)$-graded cohomology ring of $E_{\Sigma_2}C_2$. 

This generalizes Voevodsky's identification of motivic cohomology of real numbers with the positive cone of $RO(C_2)$ graded cohomology of a point.
 
   
\paragraph{Mathematics Subject Classification 2010.}
Primary:
    \href{https://mathscinet.ams.org/msc/msc2010.html?t=14Fxx&btn=Current}{14F42},
    \href{https://mathscinet.ams.org/msc/msc2010.html?t=55Pxx&btn=Current}{55P91}.
Secondary:
    \href{https://mathscinet.ams.org/msc/msc2010.html?t=55Pxx&btn=Current}{55P42},
    \href{https://mathscinet.ams.org/msc/msc2010.html?t=55Pxx&btn=Current}{55P92}.
    
\end{abstract} 

\tableofcontents
  
\section{Introduction}

         Recently, there is a flourished interest in the study of $RO(G)$-graded cohomology theory properties and $RO(G)$-graded cohomology group computations for certain topological $G$-spaces, mainly due to the recent paper of Hill, Hopkins and Ravenel on the non-existence of elements of Kervaire invariant one \cite{HHR}. 
  
  One of the directions that is particularly studied is when the group $G=\mathbb{Z}/2$; i.e. $RO(C_2)$-graded cohomology theories  and $C_2$-equivariant homotopy theory (for example  \cite{dug:kr},\cite{CM},\cite{CHT}). One of the most fundamental results in $C_2$-equivariant stable homotopy is the computation of the $RO(C_2)$-graded cohomology ring of a point by Stong (unpublished); the computation appears published  in \cite{dug:kr} or \cite{HKO} for example. In this paper, we describe the $RO(C_2\times \Sigma_2)$-graded cohomology ring of a point, an essential input in the $C_2\times \Sigma_2$-equivariant stable homotopy, with a view towards the motivic world and towards computations of $RO(C_2\times \Sigma_2)$-graded cohomologies. Both  $C_2$ and $\Sigma_2$ here mean $C_2$ but the notation distinguishes between copies of $C_2$.
  
  Another direction in equivariant homotopy is given by the relation between real/complex motivic homotopy theory and $C_2$-equivariant homotopy theory. The complexity of the $C_2$-equivariant situation is typically understood by first determining the complex motivic and real motivic situations \cite{BGI}, \cite{BS}. $C_2$-cohomological invariants are computed using real motivic cohomological invariants, and the understanding is that the latter are somehow easier to compute (and with better properties) than the former \cite{BGI},\cite{BS}. For example, Voevodsky showed that the motivic cohomology of real numbers is a polynomial subring in $RO(C_2)$-graded cohomology of a point given by its positive cone \cite{Voc}.  In particular, although the $RO(C_2)$-graded cohomology of a point contains lots of nilpotents, none of the nilpotents live in the motivic subring.
  
  It is natural to ask whether the interplay between the motivic cohomology of real/complex numbers and the $RO(C_2)$-graded cohomology of a point can be extended to other relevant finite groups like the next simple case of $C_2\times \Sigma_2$. 
  
  In \cite{HKO} the authors construct a $C_2$-equivariant stable motivic category and define the motivic real cobordism spectrum in it. In \cite{HOV1} the authors construct a $C_2$-motivic cohomology spectrum (called Bredon motivic cohomology)  and proved a Beilinson-Lichtenbaum type of theorem for its realization map over real/complex numbers. The Bredon motivic cohomology, denoted $H^{*,*}_{C_2}(-,\Z/2)$ below, is a $C_2$-equivariant motivic cohomology bi-indexed by $C_2$-virtual representations (in both dimension and weight) that extends Voevodsky's motivic cohomology to $C_2$-equivariant smooth schemes and $C_2$-representation bi-indexes. We have the following commutative diagram
  \[\begin{tikzcd}
	{H^{*,*}(\R,\Z/2)} & {H^{\st} _{Br,C_2}(pt,\Z/2)} \\
	{H^{\star,\star} _{C_2}(\R,\Z/2)} & {H^{\st} _{Br,K}(pt,\Z/2)}
	\arrow[tail, from=1-1, to=1-2]
	\arrow[tail, from=1-2, to=2-2]
	\arrow[tail, from=1-1, to=2-1]
	\arrow[tail, from=2-1, to=2-2]
\end{tikzcd}\]
  where the vertical maps are inclusions and the horizontal maps are realization maps. The upper horizontal map (between motivic cohomology of real numbers and $RO(C_2)-$graded cohomology of a point) is an inclusion according to Voevodsky's computation and the lower horizontal map (between Bredon motivic cohomology and $RO(C_2\times \Sigma_2)-$graded cohomology of a point) is proved to be a monomorphism in \cite{DV1} that extends the top horizontal monomorphism.  We denote the $RO(C_2)-$graded cohomology by $H^{\st} _{Br,C_2}(-,\Z/2)$ and the $RO(C_2\times\Sigma_2)-$graded cohomology by $H^{\st} _{Br,K}(-,\Z/2)$, where $K$ is the Klein group. We only consider $\Z/2$ coefficients for the above cohomologies.
  
    The main objective of this paper is to study in detail the lower horizontal ring map that appears in the above diagram. In \cite{DV1} the Bredon motivic cohomology groups $H^{\star,\star} _{C_2}(\R,\Z/2)$ are computed; in \cite{KHo} the $RO(C_2\times \Sigma_2)$-graded cohomology groups of a point are also computed- therefore the additive groups in the above diagram are known. In \cite{BH}, the ring structure of the positive cone of the $RO(C_2\times \Sigma_2)$-graded cohomology of a point is determined (see also \cite{HS} for the same computation using a different method) and applications of this computation are given in \cite{YL}. The multiplicative structure of the positive cone is also implied in \cite{KHo}. In \cite{BH} all the cones of $RO(C_2\times \Sigma_2)$-graded cohomology ring of a point are described as the homology of two stage chain complexes of chains of the form 
    $$0\rightarrow T_1\stackrel{f}{\rightarrow} T_2\rightarrow 0$$
  with $T_1$ and $T_2$ two nonzero graded $\Z/2-$vector spaces of chains.
  Moreover in \cite{BH}, there is a description of some important cohomology classes (and their relations) from the positive cone and outside the positive cone.  This recollection is reviewed in Section 3.1. 
  
   While the $Coker(f)$, for $f$ as in the above display, has a representation in terms of cohomology classes in \cite{BH}, the $Ker(f)$ doesn't have one in $\cite{BH}$; its representation is in terms of chains that form a $\Z/2$ vector subspace of $T_1$. In consequence, there is no defined multiplication among the chains of $Ker(f)$ in $\cite{BH}$. Moreover, a complete discussion of the $RO(C_2\times \Sigma_2)-$graded cohomology ring of a point in terms of generating cohomology classes is missing from the literature and this is the aim of Section 3 of this paper. 
    
    One of the main results of this paper is a complete $\mathbb{Z}/2$-vector space description of the Mixed Cones of Types I and II and of the negative cone of the $RO(C_2\times \Sigma_2)$-graded cohomology of a point in terms of generators and relations (see Theorem \ref{kconj}, Theorem \ref{negcone} and Propositions \ref{propcase1v2} and \ref{propcase2v2}) and a description of the main multiplicative properties between the existing cones (see Corollary \ref{ncm1}, Corollary \ref{ncm2} and Corollary \ref{ncm3}). Our methods are essentially counting methods based on the computations from \cite{KHo} and \cite{BH}. We give a cohomological representation of the elements in $Ker(f)$ in the negative cone (and in some different particular cases like Corollary \ref{excl} (2)), but in general we don't attempt to describe $Ker(f)$; rather, we use the finite dimensionality of a graded piece of $Ker(f)$ to construct enough non-zero cohomology classes which together with cohomology classes from $Coker(f)$ give an additive basis for the finite dimensional $\Z/2-$vector space $H^{a+p\sigma+b\epsilon+q\sigma\otimes\epsilon}_{Br,K}(pt,\Z/2)$ when $a,p,b,q$ vary. We notice and describe the non-nilpotent cohomology classes outside of the positive cone of $RO(C_2\times\Sigma_2)-$graded ring of a point and we also describe the nilpotent part of the $RO(C_2\times\Sigma_2)-$graded ring of a point.
    
    Denote $\MMt$ the $RO(C_2)-$graded cohomology ring of a point with $\Z/2$ coefficients. One of the main theorems of this paper is:
     \begin{theorem} \label{M1}
 We have a $\MMt$-algebra isomorphism
 $$H^{\star}_{Br,K}(pt,\Z/2)\simeq \frac{\Z/2[x_i,y_i,\ka_i]}{(\ka_i\ka_j=y_k^2,\ka_ix_i=x_jy_k+x_ky_j,\ka_iy_i=y_ky_j)}+ nilpotents_2.$$ The first component is non-nilpotent and the component $nilpotents_2$ is completely described in Section 3 and its products are not necessary zero.  This is Theorem \ref{Bmt}. The cohomological classes are described in Section 3.1. 
 
\end{theorem}
      In \cite{DV1} the Bredon motivic cohomology groups of real numbers are computed together with their realization maps into $RO(C_2\times \Sigma_2)-$graded cohomology of a point. We described there the Bredon motivic cohomology ring as a direct sum between an abstract subring and a  $\MMt-$submodule called $NC$ that has zero products, but the direct sum's components are only given as abstract objects without a description of their cohomology classes as cohomology classes of $RO(C_2\times \Sigma_2)$-graded cohomology of a point. In \cite{DV1} we defined $NC=\oplus _{b\geq 0,b+q<0}H^{\star,b+q\sigma}_{C_2}(\R,\Z/2)$.
      
   As a first application, we describe the cohomological ring of the Bredon motivic cohomology of real numbers as a subring in $H^{\st} _{Br,K}(pt,\Z/2)$ in terms of cohomology classes and their relations. As we can also notice from \cite{DV1} it contains nilpotent elements which is a new property that is not true in the case of Voevodsky's motivic cohomology. In this paper, we describe the entire subset of motivic nilpotents and prove that it has zero multiplication and that contains strictly the previously determined $NC$ (which we also describe in terms of our cohomology generators, see Theorem \ref{ncg}).  We prove that, as in the case of Voevodsky's motivic cohomology of real numbers,  Bredon motivic cohomology of real numbers and $RO(C_2\times\Sigma_2)-$graded ring of a point have the same non-nilpotent part.  Therefore we conclude that  the subring of $C_2$-motivic cohomology classes is better behaved than the ring of $C_2\times \Sigma_2$-topological cohomology classes because the motivic nilpotents have zero multiplication (while the topological nilpotents can have non-zero multiplication) and the two rings have the same non-nilpotent part.  
    
    One of the main theorems of this paper is the following:
  \begin{theorem}\label{1st}  We have a $\MMt$-algebra isomorphism
 $$H^{\star,\star}_{C_2}(\R,\Z/2)\simeq \frac{\Z/2[x_i,y_i,\ka_i]}{(\ka_i\ka_j=y_k^2,\ka_ix_i=x_jy_k+x_ky_j,\ka_iy_i=y_ky_j)}+ nilpotents_1.$$ 
 The component $nilpotents_1$ has only zero products among its elements, contains only nilpotents and it is completely described in Section 5. It is strictly contained in the component $nilpotents_2$ from Theorem \ref{M1}. The non-nilpotent component coincides with the non-nilpotent component of Theorem \ref{M1}. This is Theorem \ref{d1}.
   \end{theorem}
 
  Theorem \ref{1st} is a vast generalization of Voevodsky's description of the motivic cohomology of real numbers as the positive cone of the $RO(C_2)$-graded cohomology of a point. Indeed, if we consider only integer indexes and the corresponding realization monomorphism in the theorem above we obtain Voevodsky's identification of motivic cohomology of real numbers with the positive cone of $RO(C_2)-$graded cohomology of a point. The above result is a conclusion of Sections 3 and 5.
  
  In Section 4 we apply the results of Section 3 to determine the $RO(C_2\times \Sigma_2)$-graded cohomology ring of the $\Sigma_2$-equivariant universal space free $C_2$ space $E_{\Sigma_2}C_2$ which is the complex realization of the motivic space $\EG C_2=colim_n \mathbb{A}(n\sigma)\setminus\{0\}$ over real numbers and an homotopically important $C_2\times \Sigma_2$ topological space. This is a generalization (and an independent proof) of the computation of the integer graded cohomology subring of $E_{\Sigma_2}C_2$ that was determined in \cite{Kr} and \cite{HK1}. The proof of the latest in \cite{Kr} was based on Voevodsky's computation of the motivic cohomology ring of $\mathbf{B}C_2$, yet more evidence of the interplay between motivic homotopy ideas and $C_2$-equivariant homotopy ideas. The idea of computation of $RO(C_2\times\Sigma_2)-$graded cohomology of $E_{\Sigma_2}C_2$ relies on studying the cohomological properties of the  $C_2\times \Sigma_2$-topological isotropy cofiber sequence $$E_{\Sigma_2}C_{2+}\rightarrow pt_+\rightarrow \tilde{E}_{\Sigma_2}C_2.$$

  Moreover, for both $RO(C_2\times \Sigma_2)$ graded cohomologies of a point and $E_{\Sigma_2}C_2$ we determine all the levels of their corresponding cohomology Mackey functor in Subsections 3.2 and 4.1.
  
  One of the main theorems of Section 4 is the following (this result was already used in \cite{DV1}):
  \begin{theorem} \label{2nd} As rings,
$$H^{\st}_{Br, K}(\ek,\Z/2)=\frac{H^{\st}_{Br,\Sigma_2}(pt,\Z/2)[x_1,y_1,x_3,y_3,\ka_2^{\pm 1}]}{(\ka_2x_2=x_1y_3+x_3y_1,\ka_2y_2=y_1y_3)}.$$
The cohomology classes in this theorem are the same as those in Theorem \ref{M1} via Theorem \ref{vanishing}. This is Theorem \ref{cEC}.
  \end{theorem}
  Besides the results of Section 3, the proof of the Theorem \ref{2nd} and the appendix of this paper contain  a study of some multiplicative properties of the $C_2\times \Sigma_2$-topological isotropy cofiber sequence.
 Also, in \cite{DV1} we proved that the ring in Theorem \ref{2nd} contains as a subring the Borel motivic cohomology of real numbers (indexed by virtual representations). In Remark \ref{Breal} we discuss some particularities of this embedding.
 
  Notice that in Theorem \ref{1st} the multiplications between non-nilpotent elements and the multiplications between nilpotent elements are either zero or, if non-zero, completely expressed in terms of given motivic generators. In the Appendix, we also give some of the non-trivial multiplications in the Bredon motivic cohomology ring  (these multiplications are used in the previous sections) in terms of the given generators, more precisely multiplications of some non-nilpotent motivic cohomology classes ($y_1, y_3,x_3$) with nilpotent motivic cohomology classes that we describe in terms of nilpotent motivic generators (more multiplications of this type can be found in Appendix of \cite{BD}). We also give proofs here to some counting arguments that we used in the previous sections of the paper.
   
\sss*{Notations and Conventions:}
  \begin{itemize}[label = {$\cdot$}]
  	\item $K:=C _2\times \Sigma _2$ is the Klein four-group. We denote by $\Sigma_2$ the second copy of $C_2$. The one dimensional real representations of $K$ are denoted by $\sigma$ (the action of $C_2$ is the sign action and the action of $\Sigma_2$ is trivial), $\epsilon$ (the action of $C_2$ is trivial and the action of $\Sigma_2$ is nontrivial is the sign action) and $\sigma\otimes\epsilon$ ( the action of the diagonal subgroup $\Delta$ is trivial and the elements from $K\setminus\Delta$ act by $-1$ on $\R$). As above $C_2=Ker(\epsilon)$, $\Sigma_2=Ker(\sigma)$ and $\Delta=Ker(\sigma\otimes\epsilon)$ denote the maximal subgroups of $K$ that act trivially on the respective representation. We write $$\rho_K=1+\sigma+\epsilon+\sigma\otimes \epsilon$$ for the regular $K-$representation. We denote $\rho_{\Sigma_2}=1+\epsilon$,  $\rho_{C_2}=1+\sigma$ and $\rho_{\Delta}=1+\sigma\otimes\epsilon$.
	\item $H^{a+V}_{\Br,G}(X,A)$ is the $RO(G)-$graded cohomology of a $G$-topological space $X$ with coefficients in the constant Mackey functor $A$ with $V$ a virtual $G-$representation. We only consider the case where $A=\Z/2$ therefore when we omit the coefficients we understand $\Z/2$ coefficients. We write $H^{a+V}_{\Br,K}(X,A)$ for the same cohomology, where $K$ is the Klein four-group.
  	\item $H^{a+p\sigma,b+q\sigma} _{C _2}(X,A)$ is the Bredon motivic cohomology of a $C _2$-smooth scheme, with coefficients $A$.  All cohomology that appears in this paper is assumed to be with $\Z/2$. The bi-index is given by two arbitrary $C_2$ representations, with $\sigma$ being the sign representation. We call motivic all cohomology classes from Bredon motivic cohomology of real numbers (thess include the cohomology classes of motivic cohomology of real numbers). 
	\item We denote $H^n_{sing}(X)$ the $n^{th}$ singular cohomology of a topological space $X$.
	\item For a cohomology class $\alpha$ in $RO(K)-$graded cohomology $H^{a+p\sigma+b\epsilon+q\sigma\otimes\epsilon}_{Br,K}(X)$ we denote by $|\alpha|$ its 4-tuple degree $(a,p,b,q)$. We denote by $|\alpha|_\Sigma$ the 2-tuple degree associated to the representation $\epsilon$ and trivial representation i.e. $(a,b)$. We use $(a,p,b,q)$ as notations for coefficients of $K$ representation $a+p\sigma+b\epsilon+q\sigma\otimes\epsilon$. We use everywhere in the  paper the integer $a'$ to mean the integer $-a$. It is used as a notation for the integer index of the stable equivariant homotopy group $\pi^K_{a'}(S^V\wedge H\Z/2)$. Therefore we have  $$\pi^K_{a'}(S^V\wedge H\Z/2)=H^{a+V}_{Br,K}(pt,\Z/2)=\pi_{a'}((S^V\wedge H\Z/2)^K).$$
  	\item $H^{a,b}(X,A)$ is the motivic cohomology of a smooth scheme $X$. We only consider the case where $A=\Z/2$ therefore when we omit the coefficients we understand $\Z/2$ coefficients.
	\item We have the convention that $\st$ denotes a $RO(G)$-grading, while $*$ denotes an integer grading. For example, we have the following rings
	$$H^{\star,\star} _{C _2}(X)=\oplus_{a,b,p,q}  H^{a+p\sigma,b+q\sigma} _{C _2}(X), H^{*,*}(X)=\oplus_{a,b} H^{a,b}(X)$$ or 
  	$$H^{\star} _{Br,K}(X)=\oplus_{a,b,p,q}  H^{a+p\sigma+b\epsilon+q\sigma\otimes \epsilon} _{Br,K}(X).$$  
  	\item $H^{\st}_{\Br,C_2}(pt,\Z/2)\simeq \Z/2[x_1,y_1]\oplus \Z/2\{\frac{\theta_1}{x^n_1y_1^m}\}$ is the Bredon cohomology of a point with coefficients in the constant Mackey functor $\Z/2$ (generated by the classes $x_1,y_1,\theta_1$ of degrees $deg(x_1)=\sigma$, $deg(y_1)=\sigma-1$ and $deg(\theta_1)=2-2\sigma$). If instead of $\sigma$ we write $\epsilon,$ we mean the same cohomology generated by $x_2,y_2, \theta_2$ (both rings are viewed as being subrings in the $RO(C_2\times \Sigma_2)$-graded cohomology of a point). 
	\item $\MMtn:=H^{*+*\sigma}_{\Br,C_2}(pt, \Z/n)$.
	\item We call the direct sum $\oplus_{a\in\Z, p\geq 0} H^{a+p\sigma}_{\Br,C_2}(pt,\Z/2)$ the positive cone of $\MMt$. We call the direct sum  $\oplus_{a\in \Z,p<0} H^{a+p\sigma}_{\Br,C_2}(pt,\Z/2)$ the negative cone of $\MMt$.
  	\item We write $\ul{H}^{Br,K}_{V}(X)$, $\ul{H}^V_{Br,K}(X)$ for the usual Mackey functors associated to $H^{Br,K}_{V}$ or $H^V_{Br,K}$ where $V$ is a real $K$-representation. For $H\leq K$ and a based $K-$ space $X$ we write $$\ul{H}^{Br,K}_{V}(X)(K/H):=\rH^{Br,K}_{V}(K/H_+ \wedge X)$$ and $$\ul{H}^V_{Br,K}(X)(K/H):=\rH^V_{Br,K}(K/H_+\wedge X).$$ We also write $\ul{\pi}^K_V$ for the Mackey functor associated to the stable equivariant homotopy group $\pi^K_V$. For a Mackey functor $M$ and $H_1\leq H_2\leq K$ we denote $$Res^{H_2}_{H_1}:M(K/H_2)\rightarrow M(K/H_1)$$ the restriction map.
  	\item  $S^V$ is the one-point compactification of a $C_2\times \Sigma_2$ representation $V$. $S(V)$ is the unit sphere of a $C_2\times \Sigma_2$ representation $V$ included in the unit disk $D(V)$. 
	\item For $\mathcal{F}$ a family of subgroups of $K$ closed under taking subgroups we denote $E\mathcal{F}$ to be the $K$-CW complex uniquely determined up to $K$-homotopy equivalence by the fact that $E\mathcal{F}^H\simeq *$ (homotopy equivalence) for $H\in \mathcal{F}$ and $E\mathcal{F}^H=\emptyset$ if $H\notin \mathcal {F}$ \cite{L:book}. We denote $\mathcal{F}[H]=\{L\leq K|H\notin L\}$. We have a $K-$equivariant cofiber sequence
	$$E\mathcal{F}[H]_+\rightarrow S^0\rightarrow \tilde{E}\mathcal{F}[H]$$
induced by the $K-$equivariant collapsing map of $E\mathcal{F}[H]$. In this paper we denote $E_{\Sigma_2}C_2:=E\mathcal{F}[C_2]$ and $\tilde{E}_{\Sigma_2}C_2:=\tilde{E}\mathcal{F}[C_2]$. According to the definition, this is the $\Sigma_2$-equivariant universal free $C_2$ space (see \cite[VII.1]{May:equi}). The models we use for these spaces are 
$$E_{\Sigma_2}C_2=\colim_n S(n\sigma+n\sigma\otimes\ep)$$
and respectively
$$\tilde{E}_{\Sigma_2}C_2=\colim_n S^{n\sigma+n\sigma\otimes\ep}.$$
We call the $C_2\times \Sigma_2$ cofiber sequence 
$$E_{\Sigma_2}C_{2+}\rightarrow S^0\rightarrow \tilde{E}_{\Sigma_2}C_2$$
the $C_2\times \Sigma_2$ topological isotropy cofiber sequence.

  	\item  All $C_2$-varieties are over $\R$, and we view $C_2$ as the group scheme $C_2=\spec(\R)\sqcup \spec(\R)$.
  	
  	\item The $C_2$-equivariant Tate spheres are denoted by $S^1_t=\mathbb{G}_m$ with trivial $C_2$ action and $S^\sigma_t=\mathbb{G}_m$ with the $C_2$ action given by $x\rightarrow x^{-1}$.
	
	We define $\EG C_2=colim_n \mathbb{A}(n\sigma)\setminus\{0\}.$ We denote by $\kappa_2$ the invertible element in the Bredon motivic cohomology of $\EG C_2$ as well as its Betti realization in the $RO(C_2\times \Sigma_2)$-graded cohomology of $E_{\Sigma_2}C_2$ (which also gives $E_{\Sigma_2}C_2$ its unique periodicity in cohomology). We have a $C_2$-motivic cofiber sequence 
	$$ {\EG C_2}_+\rightarrow S^0\rightarrow \wt\EG C_2$$ with the motivic cofiber  given by $\wt\EG C_2=colim_n S^{n\sigma}\wedge S^{n\sigma}_t$.
	\item We refer to the direct sum
  	\begin{align*}
  		\bigoplus_{a\in \Z,p,b,q\geq 0} H^{a+p\sigma+b\epsilon+q\sigma\otimes \epsilon} _{\Br,K}(pt,\Z/2)
  	\end{align*}
  	 as the positive cone. This is a subring in $H^{\st}_{\Br,K}(pt, \Z/2)$.
  	\item We refer to the direct sum
  	\begin{align*}
  		\bigoplus_{a\in \Z,p\leq -1,b,q\geq 0} H^{a+p\sigma+b\epsilon+q\sigma\otimes \epsilon} _{\Br,K}(pt,\Z/2)
  	\end{align*}
  	where $p\leq -1,$ and $b,q \geq 0$ as the Mixed Cone of Type I corresponding to $\ka_1$; $\ka_1$ is a cohomology class of $deg(\ka_1)=-1-\sigma+\epsilon+\sigma\otimes\epsilon)$. By symmetry, we have analogous definitions for cohomology classes $\ka_2$ (with $deg(\ka_2)=-1+\sigma-\epsilon+\sigma\otimes\epsilon$) and $\ka_3$ (with $deg(\ka_3)=-1+\sigma+\epsilon-\sigma\otimes\epsilon$). 
	
	The Mixed Cone of Type I is the direct sum of the Mixed Cones of Type I associated to $\ka_i$ for $i=1,2,3.$

  	  \item We refer to the direct sum
  	\begin{align*}
  		\bigoplus_{a\in \Z,p\geq 0,b,q\leq -1} H^{a+p\sigma+b\epsilon+q\sigma\otimes \epsilon} _{\Br,K}(pt,\Z/2)
  	\end{align*}
  	 as the Mixed Cone of Type II corresponding to the cohomology class $\io_1$ (of $deg(\io_1)= 1+\sigma-\epsilon-\sigma\otimes\epsilon$). By symmetry, we have analogous definitions for $\io_2$ (of $deg(\io_2)=1-\sigma+\epsilon-\sigma\otimes\epsilon$) and $\io_3$ (of $deg(\io_3)=1-\sigma-\epsilon+\sigma\otimes\epsilon$). 
	 
	 The Mixed Cone of Type II is the direct sum of the Mixed Cones of Type II associated to $\io_i$ for $i=1,2,3.$
  	\item We refer to the direct sum
  	\begin{align*}
  		\bigoplus_{a\in \Z,p,b,q\leq -1} H^{a+p\sigma+b\epsilon+q\sigma\otimes \epsilon} _{\Br,K}(pt,\Z/2)
  	\end{align*}
  	as the negative cone. 
  \item We denote $$IP:=\oplus_{p,q\geq 0}\hc{a+p\si+b\ep+q\so}_{Br,K}(pt)$$ the direct sum of the Mixed Cone of Type I corresponding to $\ka_2$ with the positive cone. 
\end{itemize}
\sss*{Acknowledgements} The first author would like to thank Benjamin Ellis-Bloor for helpful discussions; the authors thank Benjamin Ellis-Bloor and Vigleit Angeltveit for the impetus the work \cite{BH} gave to the research in this paper. The authors thank the referee for many useful comments that improved the initial preprint.
  
\section{$C_2\times \Sigma_2$ topological isotropy sequence}
In this section, we review the motivic isotropy sequence for Bredon motivic cohomology and its realization the $C_2\times\Sigma_2$-topological isotropy sequence.
\subsection{The $C_2\times\Sigma_2$ topological isotropy sequence}
Consider the cofiber sequences
\begin{equation}\label{eqn:cof3}
	{K/\Sigma_2}_+\to pt_+\to S^{\sigma}.
\end{equation}
and 
\begin{equation}\label{eqn:cof4}
	E_{\Sigma_2}C_{2+}\rightarrow pt_+\rightarrow \tilde{E}_{\Sigma_2} C_2.
\end{equation}
The cofiber sequence \ref{eqn:cof3} is associated to the $C_2\times \Sigma_2$ representation $\sigma$; consequently we have by symmetry two other cofibrations associated to the $C_2\times \Sigma_2$ representations $\ep$ and $\sigma\otimes \ep$.  The same happens for the cofiber sequence \ref{eqn:cof4}; there are two other symmetric cofiber sequences depending on which two one-dimensional $C_2\times \Sigma_2$-representations are chosen. 


In \cite[Theorem 5.4, Proposition 5.7]{HOV1} it is proved that the Bredon motivic cohomology of $\EG C_2$ is $(2\sigma-2,\sigma-1)$-periodic and the reduced Bredon motivic cohomology of $\wt\EG C_2$ is $(0,\sigma)$ and $(\sigma,0)$ periodic. According to $\cite{HOV1},$ the Betti realisation of $\EG C_2$ over $\R$ is $E_{\Sigma_2}C_2$. This implies that the $RO(C_2 \ti \Si_2)$-graded cohomology  of the $C_2\times\Sigma_2$-space $E_{\Sigma_2}C_2$ has a $-1+\sigma-\ep+\sigma\otimes\ep$ periodicity. By construction, we also have that $\tilde{E}_{\Sigma_2}C_2$ has, for its reduced $RO(C_2 \ti \Si_2)$-graded cohomology, periodicities $\sigma$ and $\sigma\otimes\ep$. The above isotropy cofiber sequence is given by the realisation of the $C_2$-motivic isotropy sequence 
$$\EG {C_2}_+\rightarrow pt_+\rightarrow \wt\EG C_2,$$
obtained from taking the colimits over the cofiber sequence
$$(\mathbb{A}(n\sigma)\setminus \{0\})_+\rightarrow pt_+\rightarrow T^{n\sigma}:=S^\sigma\wedge S^\sigma_t.$$

\subsection{$\Sigma_2$-equivariant classifying spaces}
Let $B_{\Sigma_2}C_2$ be the $\Sigma_2$-equivariant classifying space that can be described as $\mathbb{RP}^\infty$ with a non-trivial $\Sigma_2$-action. It is constructed as 
$$B_{\Sigma_2}C_2=E_{\Sigma_2}C_2/C_2=\colim_n S(n\sigma+n\sigma\otimes\ep)/C_2.$$
It is the realisation of $\BG C_2=\EG C_2/C_2=\colim_n (\mathbb{A}(n\sigma)\setminus{0})/C_2,$ the $C_2$-classifying space over the field of real numbers.
The $RO(\Sigma_2)$-graded cohomology of $B_{\Sigma_2}C_2$ is given by the following theorem.
\begin{theorem} ([\cite{Kr}, Theorem 4.10] [\cite{HK1}, Lemma 6.27]) \label{comp1} We have $$H^{*+*\ep}_{Br,\Si_2}(B_{\Sigma_2}C_2,\Z/2)=\frac{H^{*+*\ep}_{Br,\Si_2}(pt)[c,b]}{(c^2=x_2 c+y_2 b)},$$ where $deg(c)=\ep$, $deg(b)=1+\ep$ and $x_2,y_2\in H^{*+*\ep}_{Br,\Si_2}(pt,\Z/2)$ are the generators of the positive cone in degrees $\ep$ and $\ep-1$ respectively. 
\end{theorem}
\begin{theorem} \label{isomt}
			As $H^\st_{Br,K}(pt)$-algebras without units, we have that
			\begin{align*}
				\hr_{Br,K}^{\st+1}(\et) \simeq \hr_{Br,\Sigma_2}^{\st}(\bc)[x_1^{\pm1},x_3^{\pm1}].
			\end{align*}
			\label{et}
	where $deg(x_1)=\sigma$ and $deg(x_3)=\sigma\otimes\epsilon$. 		
		\end{theorem}
		\begin{proof}
		From  Lemma \ref{quot}, we have that
		\begin{align*}
			\hr_{Br,K}^{a+p\sigma+b\epsilon+q\sigma\otimes\epsilon+1}(\et) \simeq \hr_{Br,K}^{a+b\epsilon+1}(\et)\simeq \hr^{a+b\epsilon+1}_{Br,\Sigma_2}(\Si\bc) \simeq \hr^{a+b\epsilon}_{Br,\Sigma_2}(\bc).
		\end{align*}
		
		Multiplication by $x_1$ and $x_3$ on $\hr_{Br,K}^{\st}(\et)$ give the first isomorphism from the corresponding periodicities.
		\end{proof}

From  Theorem \ref{comp1} and from Theorem \ref{isomt} we have that 
		\begin{theorem} As a $H^\st_{Br,K}(pt)$-algebra without units we have that $$\rH^{\st,\st}_{Br,K}(\tilde{E}_{\Sigma_2}C_2)\simeq \frac{\rH^{*-1+*\ep}_{Br,\Sigma_2}(pt)[b,c]}{(c^2=cx_2+by_2)}[x^{\pm 1}_1,x^{\pm 1} _3],$$
with $deg(b)=1+\ep$ and $deg(c)=\ep$. 

We denote $\Sigma b^k$, $\Sigma c$ and $\Sigma b^kc$ the cohomology classes in  $$\rH^{*+*\ep}_{Br,\Sigma_2}(\tilde{E}_{\Sigma_2}C_2)\simeq \rH^{*+*\ep}_{Br,\Sigma_2}(\Sigma B_{\Sigma_2}C_2)$$
of degrees $deg(\Sigma b^k)=k+1+k\ep$, $deg(\Sigma c)=1+\ep$ and $deg(\Sigma b^kc)= k+1+(k+1)\ep$.
\end{theorem}
The following theorem was proved in \cite{DV1} using a different method.
\begin{theorem}[Vanishing]
		\label{vanishing}
			The map
			\begin{align*}
				\tc{a+p\si+b\ep+q\so}_{Br,K}(\et)\to\hc{a+p\si+b\ep+q\so}_{Br,K}(pt)
			\end{align*}
			induced by the inclusion $f: pt_+ \hra \et$ vanishes for all $a,b\in\Z$ and $p,q \geq 0.$
			
			The map
			\begin{align*}
				\tc{a+p\si+q\so}_{Br,K}(\et)\to\hc{a+p\si+q\so}_{Br,K}(pt)
			\end{align*}
			vanishes for all $a\in \Z$ and $\{p,q\}\nsubseteq \Z_{<0}.$
		\end{theorem}
		\begin{proof}
		All cohomology elements $\Sigma b^k$, $\Sigma c$ and $\Sigma b^kc$ belong to the range $a,b\in\Z$ and $p,q \geq 0$ and generate the domain as a module over $H^*_{Br,K}(pt)$. In the range of the first vanishing we have cohomology classes $x_1^px_3^q\frac{\theta_2}{x^n_2y^m_2}\Sigma b^kc^t$ and  $x_1^px_3^qx_2^ny_2^m\Sigma b^kc^t$ with $t\in\{0,1\}$, $p,q\geq 0$, $k\in\Z$. 

But $f^*(\Sigma b^k)=f^*(\Sigma c)=f^*(\Sigma b^kc)=0$ for degree reasons. It implies that $f^*=0$ because $f^*$ is a $H^*_{Br,K}(pt)-$module map. This implies also the second case in the case $p,q\geq 0$ from Theorem \ref{isomt}. 

For the second vanishing, in degree $a-n\sigma$ in the domain, there are only the elements of the form  $\frac{\theta_2}{x^i_2y^j_2}\frac{\Sigma b^{i+j+1}c}{x^n_1}$ (for $a\geq 4$) and $\frac{\theta_2}{x^i_2y^j_2}\frac{\Sigma b^{i+j+2}}{x^n_1}$ (for $a\geq 5$). The domain is zero if $a\leq 3$. The image of $f^*$ on these elements belongs to $H^{a-n\sigma}_{Br,C_2}(pt)$, where the elements are nilpotents of the form $\frac{\theta_1}{x_1^ky_1^l}$ or zero. Suppose the case the image is non-zero. The multiplication with $x_2$ on these nilpotents is always non-zero and because of the module structure it implies that $f^*$ sends zero to nonzero which is contradiction. The case $a-n\sigma\otimes\ep$ is symmetric, so it follows too.
				\end{proof}
				Given the above theorem, the Mixed Cone of Type I corresponding to $\ka_2$ along with the positive cone will be of particular interest. We have that the direct sum of the groups in these two cones is $$IP:=\oplus_{p,q\geq 0}\hc{a+p\si+b\ep+q\so}_{Br,K}(pt)\subset \hc{a+p\si+b\ep+q\so}_{Br,K}(E_{\Sigma_2}C_2).$$

\section{$RO(C_2\times \Sigma_2)$-graded cohomology of a point}

\subsection{About $RO(C_2\times \Sigma_2)-$graded cohomology ring of a point}
In this section we describe the already known results about the $RO(C_2\times \Sigma _2)$- graded cohomology ring of a point following \cite{BH} and \cite{KHo}. The $RO(C_2\times \Sigma _2)$-graded cohomology groups of a point with $\Z/2$-coefficients were described in \cite{KHo} and the ring structure of the positive cone of $RO(C_2\times \Sigma _2)$- graded cohomology ring of a point was described in \cite{BH} together with specific cohomology classes (reviewed below) in $RO(C_2\times \Sigma _2)$- graded cohomology ring of a point and their relations.

The $RO(C _2\times \Sigma _2)$- graded cohomology of a point with $\Z/2$-coefficients was computed in \cite{KHo} and the results were given in terms of Poincar\'e series and reproduced in the below form in \cite{GY}. We consider $\alpha,\beta,\gamma\in \{\sigma,\epsilon,\sigma\otimes\epsilon\}$ be any three arbitrary distinct choices of one dimensional irreducible representations of $C_2\times\Sigma_2$. We have the following:
\begin{proposition} (\cite{KHo})\label{PC} Let $l,n\geq 0$ and $i,j\geq 0$. Let $\alpha$ and $\beta$ be two distinct nontrivial irreducible $C_2\times \Sigma_2$-representations. The Poincar\'e series for $\pi_*((S^V\wedge H\Z/2)^K)=H^{-*+V}_{Br,K}(pt,\Z/2)$ is

a) If $V=0$ then $1$.

b) If $V=n\alpha$ then $1+x+x^2+...+x^n$.

c) If $V=-n\alpha$ then $x^{-n}+...+x^{-3}+x^{-2}$.

d) If $V=n\alpha+l\beta$ then $(1+x+...+x^n)(1+x+...+x^l)$.

e) If $V=n\alpha-j\beta$ then $(1+x+...+x^n)(x^{-j}+...+x^{-2})$.

f) If $V=-n\alpha-j\beta$ then $(x^{-n}+...+x^{-2})(x^{-j}+...+x^{-2})$.
\end{proposition}

The following is the description of the $RO(C_2\times\Sigma_2)-$graded cohomology groups of a point in the positive cone.
\begin{proposition} (\cite{KHo})\label{pc} Let $p,b,q\geq 0$. The Poincar\'e series for $H^{-*+p\sigma+b\epsilon+q\sigma\otimes\epsilon}_{Br,K}(pt,\Z/2)$ is 
$$(1+x+...+x^p)(1+x+...+x^b)+x(1+x+...+x^{p+b})(1+...+x^{q-1}).$$
\end{proposition}
The following is the description of the $RO(C_2\times\Sigma_2)-$graded cohomology groups of a point in the Mixed Cone of Type I.

\begin{proposition}(\cite{KHo}) \label{caseMC1} Let $k\geq 1, l,m\geq 0$. If $k\leq l,m$ then the Poincar\'e series for $H^{-*+l\alpha+m\beta-k\gamma}_{Br,K}(pt,\Z/2)$ is
$$(\frac{1}{x^k}+...+\frac{1}{x})(1+x+...+x^{k-2})+x^k(1+...+x^{l-k})(1+...+x^{m-k}). $$ 
In the case $k>l$ the Poincar\'e series for $H^{-*+l\alpha+m\beta-k\gamma}_{Br,K}(pt,\Z/2)$ is
$$\frac{1}{x^{l+1}}(1+...+x^l)(1+...+x^{l-1})+\frac{1}{x^k}(1+...+x^{k-l-2})(1+...+x^{l+m}).$$
Swapping the role of $l$ and $m$ gives the case $k>m$.
\end{proposition}

The following is the description of the $RO(C_2\times\Sigma_2)-$graded cohomology groups of a point in the Mixed Cone of Type II.

\begin{proposition} (\cite{KHo}) \label{caseMC2} Let $j,k\geq 1,l\geq 0$. Then the Poincar\'e series for $H^{-*+l\alpha-j\beta-k\gamma}_{Br,K}(pt,\Z/2)$ is 
$$\frac{1}{x^{j+k-l}}(1+...+x^{j-l-2})(1+...+x^{k-l-2})+\frac{1}{x^{l+1}}(1+...+x^l)(1+...+x^{l-1}),$$
if $j,k\geq l+1$ or 
$$\frac{1}{x^j}(1+...+x^{j-2})(1+...+x^{l-k})+\frac{1}{x^k}(1+...+x^{l-1})(1+...+x^{k-1})$$
if $l\geq k$. Swapping the role of $j$ and $k$ gives the case $l\geq j$.
\end{proposition}
The following is the description of the $RO(C_2\times\Sigma_2)-$graded cohomology groups of a point in the negative cone:

\begin{proposition}(\cite{KHo})\label{NC} Let $p,b,q\leq -1$. Then the Poincar\'e series for $H^{-*+p\sigma+b\epsilon+q\sigma\otimes\epsilon}_{Br,K}(pt,\Z/2)$ is 
$$\frac{1}{x^{-p-b-q}}[(1+x+...+x^{-b-q-2})(1+...+x^{-p-2})+x^{-p-1}(1+...+x^{-q-1})(1+...+x^{-b-1})].$$
\end{proposition}
\newpage
The next part of the section is a description of the positive cone of $RO(C_2\times \Sigma_2)$ cohomology ring of a point and of the list of important cohomology classes and their relations from the negative and mixed cones of $RO(C_2\times \Sigma_2)$ cohomology ring of a point. This part reviews the main results from \cite{BH}.

{\bf{Positive Cone}}

We have for $V$ an actual  $C_2\times \Sigma_2$-representation the following equality of Mackey functors (we call them positive cones):
\begin{align*}
	\ul{\pi}^K_p(S^V \wedge H\ul{\mathbb{Z}/2}) \cong \ul{H}^{Br,K}_p(S^V;\ul{\mathbb{Z}/2}) \cong \ul{H}^{Br,K}_{p-V}(pt;\ul{\mathbb{Z}/2}) \cong \ul{H}^{-p+V}_{Br,K}(pt;\ul{\mathbb{Z}/2}).
\end{align*}
Then we have the generators of the positive cones $\mathbb{Z}/2[x_i,y_i]$ corresponding to the three $C_2$ cases. We denote $\pi^K_*$ the top level of the Mackey functor given by the equivariant stable homotopy group.
\begin{align*}
	x_1 &\in \pi^K_0(S^{0,1,0,0}\wedge H\ul{\mathbb{Z}/2}) \cong {H}_{Br,K}^{\sigma}(pt;\ul{\mathbb{Z}/2}) \cong \mathbb{Z}/2,\\
	y_1 &\in \pi^K_1(S^{0,1,0,0}\wedge H\ul{\mathbb{Z}/2}) \cong {H}_{Br,K}^{-1+\sigma}(pt;\ul{\mathbb{Z}/2}) \cong  \mathbb{Z}/2,\\
	x_2 &\in \pi^K_0(S^{0,0,1,0}\wedge H\ul{\mathbb{Z}/2}) \cong {H}_{Br,K}^{\epsilon}(pt;\ul{\mathbb{Z}/2}) \cong  \mathbb{Z}/2,\\
	y_2 &\in \pi^K_1(S^{0,0,1,0}\wedge H\ul{\mathbb{Z}/2}) \cong {H}_{Br,K}^{-1+\epsilon}(pt;\ul{\mathbb{Z}/2}) \cong  \mathbb{Z}/2,\\
	x_3 &\in \pi^K_0(S^{0,0,0,1}\wedge H\ul{\mathbb{Z}/2}) \cong {H}_{Br,K}^{\sigma\otimes\epsilon}(pt;\ul{\mathbb{Z}/2}) \cong  \mathbb{Z}/2,\\
	y_3 &\in \pi^K_1(S^{0,0,0,1}\wedge H\ul{\mathbb{Z}/2}) \cong {H}_{Br,K}^{-1+\sigma\otimes\epsilon}(pt;\ul{\mathbb{Z}/2}) \cong  \mathbb{Z}/2.
\end{align*}
\begin{theorem} (\cite{BH})\label{pcon}
	The Mackey functor structure of the positive cone in $\ul{\pi}^K_\star H\ul{\mathbb{Z}/2}$ is given by the Mackey functor of $RO(C_2\times \Sigma_2)$-graded rings
	\begin{equation*}
		\begin{tikzcd}
			& \frac{\F_2[x_1,y_1,x_2,y_2,x_3,y_3]}{(x_1y_2y_3+y_1x_2y_3+y_1y_2x_3)}\ar[dl,"Res^K_{\Si_2}"'] \ar[d,"Res^K_{C_2}"] \ar[dr,"Res^K_{\Delta}"] & \\
			\frac{\F_2[y_1,x_2,y_2,x_3,y_3]}{(x_2y_3+y_2x_3)}\ar[dr,"Res^{\Si_2}_{e}"'] & \frac{\F_2[x_1,y_1,y_2,x_3,y_3]}{(x_1y_3+y_1x_3)}\ar[d,"Res^{C_2}_{e}"] & \frac{\F_2[x_1,y_1,x_2,y_2,y_3]}{(x_1y_2+y_1x_2)}\ar[dl,,"Res^{\Delta}_{e}"]\\
			& \F_2[y_1,y_2,y_3] & 
		\end{tikzcd}
	\end{equation*}
	where each restriction map is the identity on a generator of the domain that is also a generator of the codomain and is zero on a generator otherwise. For example, the restriction of the $x_1$ in the top level is zero in $\frac{\mathbb{Z}/2[y_1,x_2,y_2,x_3,y_3]}{(x_2y_3+y_2x_3)}$ and is $x_1$ in $\frac{\mathbb{Z}/2[x_1,y_1,y_2,x_3,y_3]}{(x_1y_3+y_1x_3)}$ and $\frac{\mathbb{Z}/2[x_1,y_1,x_2,y_2,y_3]}{(x_1y_2+y_1x_2)}.$ The transfer maps are always zero.
\end{theorem}
{\bf{The negative and mixed cones}}

We first consider the top level $\pi^K_*(S^V \wedge H\ul{\mathbb{Z}/2}).$ The generators and relations from these section were given in \cite{BH}. 

We again have the generators of the three negative cones from the $C_2$ cases:
\begin{align*}
	\theta_1 & \in \pi^K_{-2}(S^{0,-2,0,0}\wedge H \ul{\mathbb{Z}/2}) \cong {H}_{Br,K}^{2-2\sigma}(pt;\ul{\mathbb{Z}/2}) \cong \mathbb{Z}/2,\\
	\theta_2 & \in \pi^K_{-2}(S^{0,0,-2,0}\wedge H \ul{\mathbb{Z}/2}) \cong {H}_{Br,K}^{2-2\epsilon}(pt;\ul{\mathbb{Z}/2}) \cong \mathbb{Z}/2,\\
	\theta_3 & \in \pi^K_{-2}(S^{0,0,0,-2}\wedge H \ul{\mathbb{Z}/2}) \cong {H}_{Br,K}^{2-2\sigma\otimes\epsilon}(pt;\ul{\mathbb{Z}/2}) \cong \mathbb{Z}/2.
\end{align*}
However, we also have seven new classes:
\begin{align*}
	\Theta &\in \pi^K_{-3} (S^{0,-1,-1,-1}\wedge H \ul{\mathbb{Z}/2}) \cong {H}_{Br,K}^{3-\sigma-\epsilon-\sigma\otimes\epsilon}(pt;\ul{\mathbb{Z}/2}) \cong \mathbb{Z}/2,\\
	\kappa_1 &\in \pi^K_{1} (S^{0,-1,1,1}\wedge H \ul{\mathbb{Z}/2}) \cong {H}_{Br,K}^{-1-\sigma+\epsilon+\sigma\otimes\epsilon}(pt;\ul{\mathbb{Z}/2}) \cong \mathbb{Z}/2,\\
	\kappa_2 &\in \pi^K_{1} (S^{0,1,-1,1}\wedge H \ul{\mathbb{Z}/2}) \cong {H}_{Br,K}^{-1+\sigma-\epsilon+\sigma\otimes\epsilon}(pt;\ul{\mathbb{Z}/2}) \cong \mathbb{Z}/2,\\
	\kappa_3 &\in \pi^K_{1} (S^{0,1,1,-1}\wedge H \ul{\mathbb{Z}/2}) \cong {H}_{Br,K}^{-1+\sigma+\epsilon-\sigma\otimes\epsilon}(pt;\ul{\mathbb{Z}/2}) \cong \mathbb{Z}/2,\\
	\iota_1 &\in \pi^K_{-1} (S^{0,1,-1,-1}\wedge H \ul{\mathbb{Z}/2}) \cong {H}_{Br,K}^{1+\sigma-\epsilon-\sigma\otimes\epsilon}(pt;\ul{\mathbb{Z}/2}) \cong \mathbb{Z}/2,\\
	\iota_2 &\in \pi^K_{-1} (S^{0,-1,1,-1}\wedge H \ul{\mathbb{Z}/2}) \cong {H}_{Br,K}^{1-\sigma+\epsilon-\sigma\otimes\epsilon}(pt;\ul{\mathbb{Z}/2}) \cong \mathbb{Z}/2,\\
	\iota_3 &\in \pi^K_{-1} (S^{0,-1,-1,1}\wedge H \ul{\mathbb{Z}/2}) \cong {H}_{Br,K}^{1-\sigma-\epsilon+\sigma\otimes\epsilon}(pt;\ul{\mathbb{Z}/2}) \cong \mathbb{Z}/2.\\
\end{align*}
These classes satisfy the following relationships. For each $\{i,j,k\} = \{1,2,3\}$ (i.e. they are all distinct), we have
\begin{align*}
	\iota_i \theta_i  &= \Theta \text{ and } \kappa_i\theta_j = \iota_k,\\
	\theta_j \theta_k &\neq 0,\\
	\iota_i \theta_j &= 0,\\
	\iota_i \kappa_i &= 0,\\
	\Theta^2 &= \theta_i \Theta = \kappa_i \Theta = \iota_i \Theta = 0,\\
	\iota_i \iota_j &= 0,\\
\end{align*}
and we can think of $\Theta$ as being divisible by $\theta_1,\theta_2$ and $\theta_3,$ where $\iota_i \theta_i= \Theta$  and $\kappa_i\theta_j\theta_k  = \Theta$. 

We also have that 
\begin{equation}\label{eq4}
\begin{aligned}
\theta_ix_i=0,\\
\theta_iy_i=0,\\
\theta_i^2=0,\\
\end{aligned}
\end{equation}
from the description of $RO(C_2)$-graded cohomology of a point.

Notice that the relations $\kappa_i\theta_j = \iota_k$ imply relations
\begin{equation}\label{eq1}
\begin{aligned}
            \ka_1\theta_3=\ka_3\theta_1,\\
            \ka_2\theta_1=\ka_1\theta_2,\\
            \ka_3\theta_2=\ka_2\theta_3.\\
\end{aligned}
\end{equation}
By degree reasons, we have
\begin{align*}
	\Theta x_i = \Theta y_i = 0
\end{align*}
for all $i\in \{1,2,3\}$ and similarly
\begin{align*}
	\iota_i x_j = \iota_i y_j = 0
\end{align*}
for all $i,j \in \{1,2,3\}$ with $i\neq j.$ However, it is not true that $\kappa_i x_i = \kappa_i y_i = 0$ for all $i \in \{1,2,3\}.$\\
For each $\{i,j,k\}=\{1,2,3\}$ we have $\kappa_i^2\neq 0$ as well as the following relations
\begin{equation}\label{eq2}
\begin{aligned}
	\kappa_i x_i &= x_j y_k + y_j x_k,\\
	\kappa_i y_i &= y_j y_k,\\
	\kappa_i \kappa_j &= y_k^2.\\
\end{aligned}
\end{equation}
We cannot express $\kappa_i^2$ in terms of $x_i,y_i$ and $\theta_i.$\\

If $x_1^{i_1}y_1^{j_1}x_2^{i_2}y_2^{j_2}x_3^{i_3}y_3^{j_3}$ is a monomial then from Proposition 4.27 \cite{BH} we have that 
$$\frac{\theta_1}{x_1^{i_1}y_1^{j_1}}\frac{\theta_2}{x_2^{i_2}y_2^{j_2}}\frac{\theta_3}{x_3^{i_3}y_3^{j_3}}=0.$$
We use below the notation
\begin{align*}
	\frac{\mathbb{Z}/2[x_1,y_1,x_2,y_2,x_3,y_3]}{(x_1^\infty,y_1^\infty)}\{\kappa_1\}
\end{align*}
to denote the $\mathbb{Z}/2$-linear span of elements of the form
\begin{align*}
	\frac{\kappa_1}{x_1^{i_1}y_1^{j_1}}x_2^{i_2}y_2^{j_2}x_3^{i_3}y_3^{j_3}
\end{align*}
where $x_1^{i_1}y_1^{j_1}x_2^{i_2}y_2^{j_2}x_3^{i_3}y_3^{j_3}$ is a monomial in $\mathbb{Z}/2[x_1,y_1,x_2,y_2,x_3,y_3].$ We call elements in this span $\ka_1$-chains, and similarly for $\ka_i,\theta_i, \Theta$ and $\io_i.$ Notice that this is an abstract notation for chains and not necessary the notation for a cohomology class (we can see that the cohomology class $k_1$ is not divisible by the cohomology classes $x_1$ and $y_1$). The notation does not a priori determine any multiplicative structure.

We have the following descriptions of the cohomology in terms of homology of a complex given in \cite{BH}. 
\begin{theorem} ([\cite{BH}, Theorem 4.31])\label{m1cone}
	The homology in the Mixed Cone of Type I corresponding to $\ka_1$ is given by the homology of the chain complex
	\begin{align*}
		0\to \frac{\mathbb{Z}/2[x_1,y_1,x_2,y_2,x_3,y_3]}{(x_1^\infty,y_1^\infty)}\{\kappa_1\} \xrightarrow{f} \frac{\mathbb{Z}/2[x_1,y_1,x_2,y_2,x_3,y_3]}{(x_1^\infty,y_1^\infty)}\{\theta_1\} \to 0
	\end{align*}
	where the map is given by multiplication with the polynomial $f = x_1y_2y_3+y_1x_2y_3+y_1y_2x_3$, and similarly for the other two mixed cones of type I.
\end{theorem}
The elements of the codomain of the above map are all $\Z/2$ sums of cohomology classes generated by cohomology classes of the form  $\frac{\theta_1}{x_1^{n_1}y_1^{m_1}}x_2^{n_2}y_2^{m_2}x_3^{n_3}y_3^{m_3}$ with $n_i,m_j\geq 0$, but not all these sums are non-zero cohomology classes. Not every element of the domain of the above map represents a cohomology class.
\begin{theorem} ([\cite{BH}, Theorem 4.30])\label{ncon}
	The homology in the negative cone is given by the homology of the chain complex 
	\begin{align*}
		0 \to \frac{\mathbb{Z}/2[x_1,y_1,x_2,y_2,x_3,y_3]}{x_1^{i_1}y_1^{j_1}x_2^{i_2}y_2^{j_2}x_3^{i_3}y_3^{j_3}}\{\Theta\} \twoheadrightarrow  \frac{\mathbb{Z}/2[x_1,y_1,x_2,y_2,x_3,y_3]}{x_1^{i_1}y_1^{j_1}x_2^{i_2}y_2^{j_2}x_3^{i_3}y_3^{j_3}}\{\theta_1\theta_2\theta_3\}\to 0,
	\end{align*}
	where the map is given by multiplication with the polynomial $f = x_1y_2y_3+y_1x_2y_3+y_1y_2x_3$. Moreover  the homology in the negative cone is only given by the above map's kernel as the map in display is surjective.
\end{theorem}

\begin{theorem} ([\cite{BH}, Theorem 4.32]) \label{m2cone}
	The homology in the Mixed Cone of Type II corresponding to $\io_1$ is given by the homology of the chain complex
	\begin{align*}
		0\to \frac{\mathbb{Z}/2[x_1,y_1,x_2,y_2,x_3,y_3]}{(x_2^\infty,y_2^\infty,x_3^\infty,y_3^\infty)}\{\iota_1\} \xrightarrow{f} \frac{\mathbb{Z}/2[x_1,y_1,x_2,y_2,x_3,y_3]}{(x_2^\infty,y_2^\infty,x_3^\infty,y_3^\infty)}\{\theta_2\theta_3\} \to 0
	\end{align*}
	where the map is given by multiplication with the polynomial $f = x_1y_2y_3+y_1x_2y_3+y_1y_2x_3$, i.e. the map
	\begin{align*}
		x_1^{i_1}y_1^{j_1}\frac{\iota_1}{x_2^{i_2}y_2^{j_2}x_3^{i_3}y_3^{j_3}}\mapsto f \cdot x_1^{i_1}y_1^{j_1}\frac{\theta_2\theta_3}{x_2^{i_2}y_2^{j_2}x_3^{i_3}y_3^{j_3}},
	\end{align*}
	and similarly for the other two Mixed Cones of Type II. 
\end{theorem}
The elements of the codomain of the above map are all $\Z/2$ sums of cohomology classes generated by cohomology classes of the form  $\frac{\theta_2}{x_2^{n_2}y_2^{m_2}}\frac{\theta_3}{x_3^{n_3}y_3^{m_3}}x_1^{n_1}y_1^{m_1}$ with $n_i,m_j\geq 0$, but not all these sums are non-zero cohomology classes. Not every element of the domain of the above map represents a cohomology class.

We briefly sketch the proof of Theorem \ref{m1cone} (Theorems \ref{ncon}, \ref{m2cone} are similar) given in \cite{BH}. Recall that $S^{n\si}$ has a $G$-CW decomposition with a single $C_2$-cell in each dimension $1 \leq i \leq n$ and two $C_2/C_2$-cells in dimension $0.$ Then the reduced chain complex computing $\pi_*((\susp^{n\si}H\underline{\Z/2})^{C_2})$ is given by
\begin{align*}
	C_*(S^{n\si})^{C_2},
\end{align*}
where
\begin{align*}
	C_*(S^{n\si}):= \Z/2[C_2/C_2] \xl{\nabla} \Z/2[C_2/e] \xl{1+\si} \Z/2[C_2/e]\xl{1+\si} \Z/2[C_2/e]\cd \xl{1+\si} \Z/2[C_2/e]
\end{align*}
and concentrated in degrees inside the interval $[0,n].$ Here $1+\si$ is the map taking both basis elements in the domain to the sum of the basis elements in the codomain. By Spanier-Whitehead duality, we can extend this definition to include all $n< 0.$ Turning to the $C_2\ti C_2$ case, we have three such chain complexes - one for each copy of $C_2.$  Then the chain complex computing $\pi^K_\st((\susp^{i\si+j\ep+k\so}H\ul{\Z/2}))$ is given by
\begin{align*}
		C:=\bigoplus_{i,j,k \in \Z} (C_*(S^{i\si}) \ot  C_*(S^{j\ep}) \ot  C_*(S^{k\so}))^K.
\end{align*}
Then there exists a subquotient (see \cite{BH})
\begin{align*}
	 T_1 \op T_2 \subseteq C
\end{align*}
given by removing the (domain and image of) differentials whose domain is a single copy of $\Z/2.$ The generators of $T_1$ can be labeled with the appropriate fractions of classes, and the generators of $T_2$ can be labeled with the appropriate products of classes, for each cone. For example, when we consider the Mixed Cone of Type I corresponding to $\ka_1$ (see Theorem \ref{m1cone}), we have that
\begin{align*}
	T_1 = \frac{\mathbb{Z}/2[x_1,y_1,x_2,y_2,x_3,y_3]}{(x_1^\infty,y_1^\infty)}\{\kappa_1\} \txa T_2 = \frac{\mathbb{Z}/2[x_1,y_1,x_2,y_2,x_3,y_3]}{(x_1^\infty,y_1^\infty)}\{\theta_1\}.
\end{align*}

In particular, each element of $T_2$ is a product of chains representing cohomology classes from the copies of $C_2,$ and this identification is bijective modulo the differentials we removed from $C$; i.e. we have that $T_2$ restricted to Mixed Cone of Type I associated to $k_1$ is
\begin{align*}
		 T_2 = \bigoplus_{i<0,j\geq 0,k\geq 0} C_*(S^{i\si})^K \ot  C_*(S^{j\ep})^K \ot  C_*(S^{k\so})^K -\text{differentials}\subseteq C.
\end{align*}
Notice that in $T_2$ we take fixed points before passing to the tensor product, so the chains in $T_2$ come from products of $C_2$ cohomology classes. But in $T_1$ we take fixed points after passing to the tensor product, so we have chains that do not come from products of $C_2$ cohomology classes. Hence the fractions in $T_1$ are a purely notational labeling, whereas the fractions in $T_2$ are actual products, as stated in the following paragraph.

 Similarly, each generator of $T_1=\frac{\mathbb{Z}/2[x_1,y_1,x_2,y_2,x_3,y_3]}{(x_1^\infty,y_1^\infty)}\{\kappa_1\} $ represents a generator of a copy of $\Z/2$ in $C$ which comes from the triple tensor product of chain complexes. For example, we have that
 \begin{align*}
 	\Z/2 \fc{\ka_1}{x_1y_1} \op \Z/2 \fc{\ta_1}{x_1}y_2y_3 &= (C_{-2}(S^{-3\si})\ot C_1(S^\ep)\ot C_1(S^\so))^K\\
 	&\supseteq C_{-2}(S^{-3\si})^K\ot C_1(S^\ep)^K\ot C_1(S^\so)^K=\Z/2 \fc{\ta_1}{x_1}y_2y_3.
 \end{align*}

 However, each element of $T_1$ is a $K$-fixed point chain which does not decompose into a product of $C_2$-fixed point chains, so the labeling with fractions is purely notational, and does not a priori determine any multiplicative structure. In particular, we will later define division at the cohomology level in a different way. 

 An elementary reduction shows that $H_*(C) = H_*(T_1\op T_2),$ so it suffices to work with $T_1\op T_2.$ 
The differential of $C$ is zero on $T_2,$ since it consists of products, and unwinding the definitions shows that it is given by $f$ on $T_1,$ where $f$ is multiplication with the element $x_1y_2y_3+y_1x_2y_3+y_1y_2x_3$ after a substitution of the fractions of classes with the appropriate products of classes, for each cone. For example, when we restrict to the case of  the Mixed Cone of Type I corresponding to $\ka_1,$ we have that $f$ is given by
\begin{align*}
f: \ub{\frac{\mathbb{Z}/2[x_1,y_1,x_2,y_2,x_3,y_3]}{(x_1^\infty,y_1^\infty)}\{\kappa_1\}}_{T_1} &\to \ub{\frac{\mathbb{Z}/2[x_1,y_1,x_2,y_2,x_3,y_3]}{(x_1^\infty,y_1^\infty)}\{\theta_1\} }_{T_2}\\
\fc{\ka_1}{x_1^{n_1}y_1^{m_1}}x_2^{n_2}y_2^{m_2}x_3^{n_3}y_3^{m_3}&\mt (x_1y_2y_3+y_1x_2y_3+y_1y_2x_3)\cp\fc{\ta_1}{x_1^{n_1}y_1^{m_1}}x_2^{n_2}y_2^{m_2}x_3^{n_3}y_3^{m_3}.
\end{align*} Then it suffices to consider the chain complex
\begin{align*}
0 \to T_1 \xr{f} T_2\to 0,
\end{align*}
from which the result below follows by passing to homology. Thus, we can understand the cohomology in terms of $\ka_1$-chains and $\ta_1$-chains, and similarly for the other cones.

\subsection{The ring structure of the middle level of the Mackey functor $\ul{H}_{Br,K}^\st(pt,\Z/2)$}
		We first provide a simplification of the description of the middle level of $\ul{H}_{Br,K}^\st(pt,\mathbb{Z}/2)$ given in \cite{BH}.
		\begin{lemma}[ \cite{HOV1}, Theorem 5.12]
			\label{st}
			Let $A$ be an abelian group. There is a $H^{*,*}(k,A)$-algebra isomorphism
			\begin{align*}
				H^{\star,\star}_{C_2}(C_2,A)\cong H^{*,*}(k,A)[s^{\pm 1}, t^{\pm 1}],
			\end{align*}
			where $s \in H^{\sigma-1,0}_{C_2}(C_2,A)$ and $t \in H^{\sigma-1,\sigma-1}_{C_2}(C_2,A).$
		\end{lemma}
		
		The cycle map (\cite{HOV1}) 
		\begin{align*}
			H_{C_2}^{a+p\sigma,b+q\sigma}(C_2,\Z/2) \to H_{Br,K}^{a-b+(p-q)\sigma+b\ep+q\sigma\otimes\ep}(C_2,\Z/2)
		\end{align*}
		 is a ring map, so the images of $s$ and $t$ from Lemma \ref{st} are units $s_1\in H^{-1+\sigma}_{Br,K}(C_2,\Z/2)$ and $t_1\in H^{-\epsilon+\sigma\otimes\epsilon}_{Br,K}(C_2,\Z/2).$ By symmetry, we have units for the other subgroups as well, so we have the units
		 \begin{equation} \label{eq4}
		\begin{aligned}
			s_1\in & H_{Br,K}^{-1+\sigma}(C_2),\\
			t_1\in & H_{Br,K}^{-\epsilon+\sigma\otimes\epsilon}(C_2),\\
			s_2\in & H_{Br,K}^{-1+\epsilon}(\Si_2),\\
			t_2\in & H_{Br,K}^{\sigma-\sigma\otimes\epsilon}(\Si_2),\\
			s_3\in & H_{Br,K}^{-1+\sigma\otimes\epsilon}(K/\Delta),\\
			t_3\in & H_{Br,K}^{-\sigma+\epsilon}(K/\Delta).\\
		\end{aligned}
		\end{equation}
				Because we have an isomorphism $Res^K_{\Sigma_2}: H_{Br,K}^{-1+\sigma}(pt)\simeq H_{Br,K}^{-1+\sigma}(C_2)$ (Theorem \ref{pcon}) we use the same notation for $s_1$ and $y_1$, the only nontrivial elements in these two groups (and similarly for $s_2$ and $y_2$ and for  $s_3$ and $y_3$).
				
				We have proven the following theorem:
		\begin{theorem}  \label{def}
			\label{middle}
			The middle level of $\ul{H}_{Br,K}^\st(pt)$ is given by
			\begin{align*}
				\ul{H}_{Br,K}^\st(pt)(C_2) &= {H}_{Br,K}^\st(C_2) \simeq {H}_{Br,\Sigma_2}^{*,*}(pt) [y_1^{\pm1},t_1^{\pm1}],\\
				\ul{H}_{Br,K}^\st(pt)(\Si_2) &= {H}_{Br,K}^\st(\Si_2) \simeq {H}_{Br, C_2}^{*,*}(pt) [y_2^{\pm1},t_2^{\pm1}],\\
								\ul{H}_{Br,K}^\st(pt)(K/\Delta) &= {H}_{Br,K}^\st(K/\Delta) \simeq {H}_{Br,\Delta}^{*,*}(pt) [y_3^{\pm1},t_3^{\pm1}].
			\end{align*}

		\end{theorem}
The ring structure at the lowest level of $\ul{H}_{Br,K}^\st(pt)$ is given in the Subsection 4.2 as it coincides with the ring structure at the lowest level of $\ul{H}_{Br,K}^\st(E_{\Sigma_2}C_2)$.

\subsection{The Mixed Cone of Type I}
We give a better description of Mixed Cone of Type I corresponding to $\ka_2$ from the analogue of Theorem \ref{m1cone} by giving a full description of it in terms of its cohomology classes. By symmetry the description is valid for Mixed Cone of Type I corresponding to $\ka_1$ and $\ka_3$.

We have the following result:
		\begin{theorem}
		\label{kconj}
		Let
		\begin{align*}
			T := \frac{\Z/2[x_1,y_1,x_2,y_2,x_3,y_3]}{(x_2^\infty,y_2^\infty)}\{\theta_2\}.
		\end{align*}
		The Mixed Cone of Type I corresponding to $\ka_2$ ($p,q\geq 0, b<0$) is given by
		\begin{align*}
			\ka_2\Z/2[x_1,y_1,x_3,y_3,\ka_2] \op  \fc{T}{(x_1y_2y_3+y_1x_2y_3+y_1y_2x_3)T}.
		\end{align*}
	\end{theorem}

	We recall the convention that $a':=-a$ that implies $\pi_{a'}((S^V\wedge H\mathbb{Z}/2)^K)=H^{a+V}_{Br,K}(pt,\Z/2)$.
		The following lemma will be useful in our arguments.
		\begin{lemma}
			\label{coeff}
			The coefficient $c_{a'}$ of $x^{a'}$ in the polynomial $x^s(1+\cd+x^{w-1})(1+\cd+x^{h-1})$ is given by
			\begin{align*}
				c_{a'}  =\begin{cases}
					a'-s+1& \text{if } s\leq a' \leq s+h-1,\\
					h & \text{if } s+h-1\leq a' \leq s+w-1,\\
					s+w+h-1-a' & \text{if } s+w-1\leq a' \leq s+w+h-2,\\
				\end{cases} 
			\end{align*}
			where $w \geq h\geq1$ and $s\in \Z$.
		\end{lemma}
		
		Recall that by Theorem \ref{m1cone}, the cohomology of the Mixed Cone of Type I corresponding to $\ka_2$ is given by the homology of the chain complex
		\begin{align*}
			0\to \ub{\frac{\Z/2[x_1,y_1,x_2,y_2,x_3,y_3]}{(x_2^\infty,y_2^\infty)}\{\kappa_2\}}_{D} \xrightarrow{f} \ub{\frac{\Z/2[x_1,y_1,x_2,y_2,x_3,y_3]}{(x_2^\infty,y_2^\infty)}\{\theta_2\}}_T \to 0.
		\end{align*}
		In other words, the cohomology is $\ker f \op T/fT.$
		Write
		\begin{align*}
			f_{a,p,b,q} : D_{a,p,b,q} \to T_{a+1,p,b,q}
		\end{align*}
		for the restriction of $f$ to the graded piece of degree $(a,p,b,q)$, a $\Z/2$ vector space. Then we have that
		\begin{align*}
			\dim H^{a+p\sigma+b\epsilon+q\sigma\otimes\epsilon}_{Br,K}(pt) &= \dim\ker f_{a,p,b,q} + \dim T_{a,p,b,q}-\dim\im f_{a-1,p,b,q}\\
			&= \dim\ker f_{a,p,b,q}+ \dim\ker f_{a-1,p,b,q} + \dim T_{a,p,b,q}-\dim D_{a-1,p,b,q}.
		\end{align*}
		Since $IP \hra H_{Br,K}^\st(\ek),$ (Proposition \ref{vanishing}) we have that multiplication with $\ka_2$ is an injection in $IP$ because $\ka_2$ is invertible in $H_{Br,K}^\st(\ek),$ because it is an invertible motivic class \cite{DV1}. 
		 Then we have that
		\begin{align*}
			P:= \ka_2\Z/2[\ka_2,x_1,x_3,y_1,y_3] \hra H_{Br,K}^\st(pt).
		\end{align*}
		Indeed, from \cite{BH} we have that $\Z/2[x_1,x_3,y_1,y_3] \hra H_{Br,K}^\st(pt)$. Then if $\al=\ka_2^{n}r_n + \ka_2^{n-1}r_{n-1}+\cdots + r_0= 0$ in $H_{Br,K}^\st(pt)$ where $r_i \in \Z/2[x_1,x_3,y_1,y_3], n \geq 1,$ and $r_n \neq 0,$ then by degree reasons we have that $\ka_2^n r_n=0$, contradicting the injectivity of $\ka_2^n.$ Thus, the generators of $\ka_2\Z/2[\ka_2,x_1,x_3,y_1,y_3]$ are linearly independent in $H_{Br,K}^\st(pt).$
		\begin{lemma}
			\label{kt}
			\begin{align*}
				P \op \fc{T}{\im f }\subseteq ker(f)\oplus \frac{T}{\im f}=\tx{Mixed Cone of Type I corresponding to }\kappa_2 
			\end{align*}
		\end{lemma}
		\begin{proof}
			By definition, the generators are elements of the Mixed Cone of Type I corresponding to $\kappa_2$, so it remains to show that they are linearly independent. Fix a degree $V=(a,p,b,q)$ in the Mixed Cone of Type I corresponding to $\ka_2$. Then the only elements of $P$ in $H^V_{Br,K}(pt)$ are of the form $\ka_2^{-b}\cp \al,$ where $\al\in \Z/2[x_1,y_1,x_3,y_3].$ On the other hand, to belong in the group, any element of $\fc{T}{\im f }$ in $H^V_{Br,K}(pt)$ can only have powers of $x_2$ or $y_2$ of at most $-b-2$ in the denominators of its constituent $\ta_2$-fractions. In particular we can assume $b\leq -2$. Then suppose for contradiction that
			\begin{align*}
				\ka_2^{-b}\cp \al + \be = 0, 
			\end{align*}
			where $\be \in \fc{T}{\im f }\cap H^V(pt)$ and neither summand is zero. Then 
			\begin{align*}
				y_2^{-b-1}(	\ka_2^{-b}\cp \al+ \be) =\ka_2 (y_2\ka_2)^{-b-1}\al+y_2^{-b-1}\be= \ka_2(y_1y_3)^{-b-1} \al+ 0 \neq 0,
			\end{align*}
			since multiplication with $\ka_2$ is injective on the positive cone, and $0\neq (y_1y_3)^{-b-1} \al \in \Z/2[x_1,y_1,x_3,y_3].$ We use above  that $\kappa_2 y_2 = y_1y_3$ and that $y_2^{j+1} \cdot \frac{\theta_2}{x_2^iy_2^j} = 0$ for $j\geq 0$.
		\end{proof}
	
		By Lemma \ref{kt} we have that $\dim P_{a,p,b,q}\leq \dim \ker f_{a,p,b,q}$ as finite-dimensional $\Z/2-$subspaces. Take $B=-b\geq 1$. Then to prove Theorem \ref{kconj}, it suffices to show that
		\begin{align*}
			\dim\pi_{a'}((\Si^{p\si-B\ep+q\so}H\ul{\Z/2})^{{K}}) := \dim H_{Br,K}^{a+p\sigma+b\epsilon+q\sigma\otimes\epsilon}(pt) =\hat \pi_{a'}((\Si^{p\si-B\ep+q\so}H\ul{\Z/2})^{{K}}),
		\end{align*}
		where
		\begin{align*}
			\hat \pi_{a'}((\Si^{p\si-B\ep+q\so}H\ul{\Z/2})^{{K}}) := \dim P_{a,p,b,q}+ \dim P_{a-1,p,b,q} + \dim T_{a,p,b,q}-\dim D_{a-1,p,b,q}.
		\end{align*}
		
		Since $T_{a,p,b,q} =\Z/2\left\{\fc{\ta_2}{x_2^{n_2}y_2^{m_2}}x_1^{n_1}y_1^{m_1}x_3^{n_3}y_3^{m_3}\right\},$ we require that
		\begin{align*}
			a &= 2+m_2-m_1-m_3,\\
			p &= n_1+m_1,\\
			b &= -2-n_2-m_2,\txa\\
			q &= n_3+m_3,
		\end{align*}
		where $n_1,m_1,n_2,m_2,n_3,m_3 \geq 0.$ Simplifying these restrictions, we have that
		\begin{align*}
			\dim T_{a,p,b,q} = \left|\{(n_2,n_3) \in \Z_{\geq0}^2: n_3\leq q \text{ and } n_2 \leq -b-2 \text{ and } a+b+q\leq n_3-n_2=a+b+q+m_1 \leq a + p +b + q\}\right|.
		\end{align*}
		By similar arguments, for $D_{a,p,b,q}=\{\frac{\ka_2}{x_2^{n_2}y_2^{m_2}}x_1^{n_1}y_1^{m_1}x_3^{n_3}y_3^{m_3}\}, n_1,n_2,m_1,m_2,n_3,m_3\geq 0$, we have that
		\begin{align*}
			\dim D_{a,p,b,q} =  |\{(n_2,n_3) \in \Z_{\geq0}^2: n_3\leq q-1 &\text{ and } n_2 \leq -b-1 \\
			&\text{ and } a+b+q+1\leq n_3-n_2= a+b+q+1+m_1\leq a + p +b + q\}|,
		\end{align*}
		and for $P_{a,p,b,q}=\{\ka_2^rx_1^{n_1}x_3^{n_3}y_1^{m_1}y_3^{m_3}\}, r\geq 1,n_1,m_1,n_3,m_3\geq 0$ we have
		\begin{align*}
			\dim P_{a,p,b,q} = \left|\{n_1 \in \Z_{\geq0} : a+p\leq n_1 \leq b+p \text{ and } n_1 \leq {a+b+p+q}\}\right|.
		\end{align*}
		We use the above formulas to evaluate $\hat\pi_{a'}.$ Below is the proof of Theorem \ref{kconj}:
		\begin{proof}
		If $B,p,$ or $q$ are zero, then we are done by degree reasons. If $B=1,$ then by degree reasons there are no classes involving $\ta_2,$ so we can directly evaluate the Poincar\'e series for $\dim P_{*,p,-B,q}$, which is the polynomial
		\begin{align*}
			x^B(1+\cd+x^{p-B})(1+\cd+x^{q-B}).
		\end{align*}
		Taking $B=1,$ we see that it does indeed agree with the formula in Theorem \ref{caseMC1}. Therefore in this case $P_{*,p,-1,q}=Ker(f_{*,p,-1,q})$ for $p,q\geq 0$.
		
		Then we assume that $p,q \geq 1$ and $B \geq 2.$ Our formula for $\hat\pi_{*} ((\Si^{p\si-B\ep+q\so}H\ul{\Z/2})^{{K}})$ is given by
		\begin{align*}
			\hat\pi_{a'}&= \ob{\sum_{j=0}^{q}\sum_{i=0}^{B-2}\ic_{[-a'-B+q,-a'-B+p+q]}(j-i)-\sum_{j=0}^{q-1}\sum_{i=0}^{B-1}\ic_{[-a'-B+q,-a'-B+p+q-1]}(j-i)}^{M:= \dim T_{-a',p,-B,q}-\dim{D}_{-a'-1,p,-B,q}}\\
			&+\ub{\left|[0,-a'-B+p+q]\cap[p-a',p-B]\right|}_{N_1:= \dim P_{-a',p,-B,q}}\\
			&+\ub{\left|[0,-a'-B+p+q-1]\cap[p-a'-1,p-B]\right|}_{N_2:=\dim P_{-a'-1,p,-B,q}},
		\end{align*}
		where intervals of negative length are regarded as empty, and $\ic_{[\ell,k]}$ is the characteristic function of the interval $[\ell,k].$ We now write
			$$M = L + M_1 + M_2,$$
		where
		\begin{align*}
			L &= \sum_{j=0}^{q-1}\sum_{i=0}^{B-2}\ic_{[-a'-B+q,-a'-B+p+q]}(j-i)-\ic_{[-a'-B+q,-a'-B+p+q-1]}(j-i) \\
			&= \left| \{(i,j) : j-i =-a'-B+p+q\} \cap \{ (i,j): i \in [0,B-2], j\in[0,q-1] \}\right|,\\
			M_1 &= -\sum_{j=0}^{q-1}\ic_{[-a'-B+q,-a'-B+p+q-1]}(j-B+1)\\
			&=-\left|[-B+1,q-B]\cap[-a'-B+q,-a'-B+p+q-1]\right|, \txa\\
			M_2 &= \sum_{i=0}^{B-2}\ic_{[-a'-B+q,-a'-B+p+q]}(q-i)\\
			&=\left|[q-B+2,q]\cap[-a'-B+q,-a'-B+p+q]\right|.
		\end{align*}
	 Then we have that
		\begin{align*}
			L &=
			\begin{cases}
				\begin{cases}
					B-p+a' & \text{if } p-B+1\leq a' \leq p+q-B,\\
					q & \text{if } p+q-B\leq a' \leq p-1,\\
					p+q-a'-1 & \text{if } p-1\leq a' \leq p+q-2,\\
				\end{cases}
				& \text {if } q \leq B-1,\\
				\begin{cases}
					B-p+a' & \text{if } p-B+1\leq a' \leq p-1,\\
					B-1 & \text{if } p-1\leq a' \leq p+q-B,\\
					p+q-a'-1 & \text{if } p+q-B\leq a' \leq p+q-2,\\
				\end{cases} & \text {if } q \geq B-1.\\
			\end{cases}\\
		\end{align*}
		Note that this can also be seen geometrically, as shown in Figure \ref{counting}.
		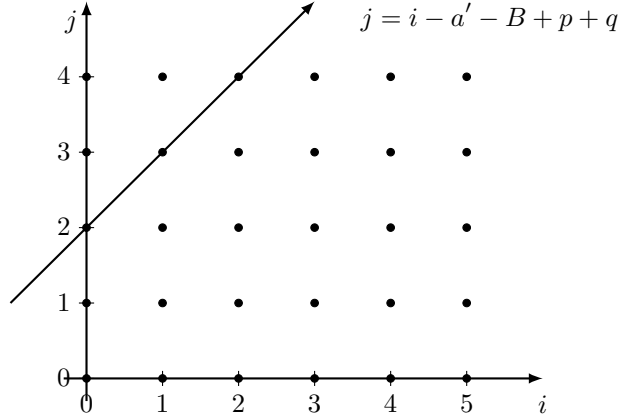
\begin{figure}[H]
			
			\begin{tikzpicture}
				[
				dot/.style={circle,draw=black, fill,inner sep=1pt},
				]
				
				\foreach \y in {0, ..., 4}
				\foreach \x in {0, ..., 5} {
					\node[dot] at (\x,\y){};
				}
				
				\foreach \x in {0,...,5}
				\draw (\x,-0.1) -- node[below,yshift=-1mm] {\x} (\x,0.1);

				\foreach \y in {0,...,4}
				\draw (0.1,\y) -- node[below,yshift=2.5mm,xshift=-3mm] {\y} (-0.1,\y);

				\draw[->,thick,-latex] (0,-0.3) -- (0,5);
				\draw[->,thick,-latex] (-0.3,0) -- (6,0);
				
				\draw[->,thick,-latex] (-1,1) -- (3,5);

				\node[below,xshift=3mm,yshift=1mm] at (5,5) {$j= i-a'-B+p+q$};

				\node[below,yshift=-1mm] at (6,0) {$i$};
				\node[below,xshift=-2mm] at (0,5) {$j$};  
				
				
			\end{tikzpicture}
			\caption{Since $L=\left| \{(i,j) : j-i =-a'-B+p+q\} \cap \{ (i,j): i \in [0,B-2], j\in[0,q-1] \}\right|$, we can compute $L$ by counting the number of points in the intersection of the lattice $[0,B-2]\ti[0,q-1]$ with the line $j= i-a'-B+p+q.$ Here, we show the case where $a'=1,p=q=5,$ and $B=7.$ We see that the intersection consists of $3$ points, which is precisely equal to $B-p+a',$ corresponding to the case in the top line of the previous display. As $a'$ varies, the line $j= i-a'-B+p+q$ shifts up and down, since all the other variables are fixed. Then we see that the cardinality of the intersection coincides precisely with the corresponding cases in the previous display.}
			\label{counting}
		\end{figure}
		Then by Lemma \ref{coeff} the Poincar\'e series for $L$ is given by
		\begin{align*}
			\fc{x^{p-B+1}}{(x-1)^2}(x^q-1)(x^{B-1}-1).
		\end{align*}
		Next, we have that
		\begin{align*}
			M_1 &= \max(1,q-a')-\min(p-a'-1,0)-q-1 \tx{if } 0 \leq a' \leq p+q-2,\\
			M_2 &= \min(p-a'-B,0)-\max(2,-a')+B+1 \tx{if } -B \leq a' \leq p-2,\\
			N_1 &= \min(q-a',0)-\max(p-a',0)+p-B+1 \tx{if } B \leq a' \leq p+q-B \text{ and } B \leq p,q,\\
			N_2 &= \min(q-a'-1,0)-\max(p-a'-1,0)+p-B+1\tx{if } B-1 \leq a' \leq p+q-B-1\text{ and } B \leq p,q.
		\end{align*}
		Then it can be shown that these terms are also given by Poincar\'e series of the form appearing in Lemma \ref{coeff}, and in particular, the Poincar\'e series for $M_1,M_2$, and $N_1+N_2$ are given by
		\begin{align*}
			-&\fc{1}{(x-1)^2}(x^p-1)(x^q-1),\\
			&\fc{x^{-B}}{(x-1)^2}(x^{p+1}-1)(x^{B-1}-1) ,\txa\\
			&\fc{x^{B-1}(1+x)}{(x-1)^2}(x^{p-B+1}-1)(x^{q-B+1}-1) \tx{(only when $B\leq p,q$)}
		\end{align*}
		respectively. Then if $B > p$ or $B> q,$ we have that $N_1= N_2 = 0,$ so the total Poincar\'e series is given by
		\begin{align*}
			M &= L + M_1 + M_2\\
			&= 	\fc{x^{p-B+1}}{(x-1)^2}(x^q-1)(x^{B-1}-1) -\fc{1}{(x-1)^2}(x^p-1)(x^q-1)+\fc{x^{-B}}{(x-1)^2}(x^{p+1}-1)(x^{B-1}-1)\\
			&= \fc{1}{(x-1)^2}\left(-x^{-B + p + q + 1} + x^{-B} + x^p + x^q - \fc{1}{x} - 1\right),
		\end{align*}
		which agrees with the Theorem \ref{caseMC1} (2). Similarly, if $B \leq p,q$ we have that the total Poincar\'e series is given by
		\begin{align*}
			 &L + M_1 + M_2+N_1+N_2=\\
			&= 	\fc{1}{(x-1)^2}\left(-x^{-B + p + q + 1} + x^{-B} + x^p + x^q - \fc{1}{x} - 1\right)+\fc{x^{B-1}(1+x)}{(x-1)^2}(x^{p-B+1}-1)(x^{q-B+1}-1)\\
			&=\fc{1}{(x-1)^2}\left(x^{-B + p + q + 2} + x^{B - 1} + x^{-B} + x^B - x^{p + 1} - x^{q + 1} - \fc{1}{x} - 1\right),
		\end{align*}
		which also agrees with the formula in Theorem \ref{caseMC1} (1).
		
		We conclude that $dim(P_{a,p,b,q})=dim(Ker f_{a,p,b,q})$ for $p,q\geq 0$ and $b<0$. Therefore we conclude the statement of the proposition using Lemma \ref{kt} and Theorem \ref{m1cone} for the case $\kappa_2$.
		\end{proof}
		
		By symmetry, we also have analogous statements for the Mixed Cones of Type I corresponding to $\ka_1$ and $\ka_3.$
		
\subsection{The Mixed Cone of Type II}
Throughout this section, we will only be considering the Mixed Cone of Type II corresponding to $\io_2$ i.e. $p, q\leq -1, b\geq 0.$ We will determine below the generators of $\Z/2$ vector spaces $H^{a+p\sigma+b\epsilon+q\sigma\otimes\epsilon}_{Br,K}(pt,\Z/2)$ that belong to the Mixed Cone of Type II corresponding to $\io_2$.
By symmetry, the results are also valid for the Mixed Cone of Type II corresponding to $\io_1$ and $\io_3$. 

We will use the relations \ref{eq1} for the following lemma.
\begin{lemma}\label{iononzero}
	The products 
	\begin{equation}\label{eq3}
	\begin{aligned}
		\ka_3\ta_1 x_2^{n_2}y_2^{m_2}=\ka_1\ta_3x_2^{n_2}y_2^{m_2}
	\end{aligned}
	\end{equation}
	are nonzero for all $n_2,m_2 \geq 0.$
\end{lemma}
\begin{proof}
	We have
	\begin{align*}
		\fc{\ta_2}{x_2^{n_2}y_2^{m_2}} \cp \ka_1\ta_3x_2^{n_2}y_2^{m_2} = \ka_1 \ta_2 \ta_3 = \Theta \neq 0 \text{ and } \ka_3\theta_1=i_2=\ka_1\theta_3.
	\end{align*}
\end{proof}

\begin{lemma}
	\label{x1case1}
	The map $H_{Br,K}^{a+(p-1)\sigma+b\epsilon+q\sigma\otimes\epsilon}(pt)\to H_{Br,K}^{a+p\sigma+b\epsilon+q\sigma\otimes\epsilon}(pt)$ given by $\al \mt x_1 \cp \al$ is surjective when $p \leq 0 < -q \leq b.$ By symmetry, we also have that the map $H_{Br,K}^{a+p\sigma+b\epsilon+(q-1)\sigma\otimes\epsilon}(pt)\to H_{Br,K}^{a+p\sigma+b\epsilon+q\sigma\otimes\epsilon}(pt)$ given by $\al \mt x_3 \cp \al$ is surjective when $q \leq 0 < -p \leq b.$
\end{lemma}
\begin{proof}
	The cofiber sequence
	\begin{align*}
		\ct_+\to pt_+ \to S^\si
	\end{align*}
	induces the long exact sequence
	\begin{align*}
		\cd \to H_{Br,K}^{a+(p-1)\sigma+b\epsilon+q\sigma\otimes\epsilon}(pt)\xr{\cp x_1} H_{Br,K}^{a+p\sigma+b\epsilon+q\sigma\otimes\epsilon}(pt) \to H_{Br,K}^{a+p\sigma+b\epsilon+q\sigma\otimes\epsilon}(\ct)\xr{\dl} H_{Br,K}^{a+1+(p-1)\sigma+b\epsilon+q\sigma\otimes\epsilon}(pt) \to \cd
	\end{align*}
	in cohomology. Then it suffices to show that $\dl$ is injective. We know from Theorem \ref{middle} that 
	\begin{align*}
		H_{Br,K}^{a+p\sigma+b\epsilon+q\sigma\otimes\epsilon}(\ct) \xl[\sim]{y_1^{p}(t_1^{-1})^{-q}} H_{Br,K}^{a+p+(b+q)\epsilon}(\ct),
	\end{align*}
	where $t_1^{-1}$ is the invertible element in $H_{Br,K}^{\epsilon-\sigma\otimes\epsilon}(\ct)$ and $y_1\in H_{Br,K}^{-1+\sigma}(\ct)$ the invertible element. By \cite{BH} Theorem 4.36, we know that $Res^K_{C_2}(\ka_3)=y_1t_1^{-1},$ so 
	\begin{align*}
		y_1^{p}(t_1^{-1})^{-q} = y_1^{p+q}Res^K_{C_2}(\ka_3)^{-q},
	\end{align*}
	where we identify the last term with $\ka_3^{-q} y_1^{p+q},$ since $Res^K_{C_2}$ is also (by definition) the map giving the $H_{Br,K}^\st(pt)$-module structure. We also know that the only nonzero terms in $$H^{a+p+(b+q)\epsilon}_{Br,K}(\ct)\simeq H^{a+p+(b+q)\epsilon}_{Br,\Sigma_2}(pt)$$ are of the form $x_2^{n_2}y_2^{m_2},$ since $b+q \geq 0.$ Also $Res^K_{C_2}(x_2)=x_2$ and $Res^K_{C_2}(y_2)=y_2$.
	Then it suffices to show that 
	\begin{align*}
		\dl(\ka_3^{-q} y_1^{p+q}x_2^{n_2}y_2^{m_2}) = \ka_3^{-q}x_2^{n_2}y_2^{m_2}\dl( y_1^{p+q}) \neq 0.
	\end{align*}
	We have that
	\begin{align*}
		\ub{H_{Br,K}^{n-n\sigma}(C_2)}_{=\Z/2y_1^{-n}}\xr[\sim]{\dl} \ub{H^{n+1-(n+1)\sigma}_{Br,K}(pt)}_{=\Z/2\fc{\ta_1}{y_1^{n-1}}} \xr{\cp x_1} H^{n+1-n\sigma}_{Br,K}(pt) = 0
	\end{align*}
	when $n \geq 1,$ so 
	\begin{align*}
		\ka_3^{-q}x_2^{n_2}y_2^{m_2}\dl( y_1^{p+q})  = 	\ka_3^{-q}x_2^{n_2}y_2^{m_2}\fc{\ta_1}{y_1^{-(p+q+1)}}.
	\end{align*}
	Suppose for contradiction that the last expression is zero. Then we have
	\begin{align*}
		y_3^{-q-1}\ka_3^{-q}x_2^{n_2}y_2^{m_2}\fc{\ta_1}{y_1^{-(p+q+1)}} &= y_2^{-q-1}y_1^{-q-1} \ka_3 x_2^{n_2}y_2^{m_2}\fc{\ta_1}{y_1^{-(p+q+1)}} = \ka_3 x_2^{n_2}y_2^{m_2-q-1}\fc{\ta_1}{y_1^{-p}} =0,\\
		\implies \ka_3\fc{\ta_1}{y_1^{-p}}x_2^{n_2}y_2^{m_2-q-1} = 0,
	\end{align*}
	contradicting Lemma \ref{iononzero} (after multiplication by $y_1^{-p}$).
\end{proof}

By Lemma \ref{x1case1}, we can define elements 
\begin{align*}
	\fc{\ka_1}{x_1^{n_1}}\fc{\ta_3}{x_3^{n_3}y_3^{m_3}}x_2^{n_2}y_2^{m_2} \in H_{Br,K}^{a-(1+n_1)\sigma+b\epsilon+q\sigma\otimes\epsilon}(pt)
\end{align*}
in the preimages of the elements
\begin{align*}
	\ka_1 \fc{\ta_3}{x_3^{n_3}y_3^{m_3}}x_2^{n_2}y_2^{m_2} \in H_{Br,K}^{a-\sigma+b\epsilon+q\sigma\otimes\epsilon}(pt)
\end{align*}
under multiplication by $x_1^{n_1},$ provided we make a well-defined choice, as in the following proof.
\begin{proposition}
	\label{propcase1v2}
	There is an isomorphism of $\Z/2$-spaces
	\begin{align*}
		\bigoplus_{q\leq-1,p\leq-1,b\geq1,b+q\geq0}\pi_{*}((\Si^{p\si+b\ep+q\so}H\ul{\Z/2})^{\mathcal{K}})&\cong \Z/2\left\{\ka_3^{m_3+1}\cdot \fc{\ta_1}{x_1^{n_1}y_1^{m_3+m_1}}\cdot x_2^{n_2}y_2^{m_2-m_3}\right\}_{m_3\leq m_2,m_1\geq 1}\\
		&\op\Z/2\left\{\fc{\ka_1}{x_1^{n_1}}\fc{\ta_3}{x_3^{n_3}y_3^{m_3}}x_2^{n_2}y_2^{m_2}\right\}_{n_3+m_3\leq n_2+m_2}.
	\end{align*}
\end{proposition}
\begin{proof}
	When $b+q \geq 0,p,q\leq-1$ we have from case (2) of Proposition \ref{caseMC2} that the Poincar\'e series for $\pi_{a'}((\Si^{p\si+b\ep+q\so}H\ul{\Z/2})^{\mathcal{K}})$ is given by
	\begin{align*}
		x^p(1+\cd+x^{-p-2})(1+\cd+x^{b+q})+x^q(1+\cd+x^{b-1})(1+\cd+x^{-q-1}).
	\end{align*}
	When $p=-1$, we have from the same arguments (see Lemma \ref{k1series} below) as in the proof of Theorem $\ref{kconj}$ that $x^q(1+\cd+x^{b-1})(1+\cd+x^{-q-1})$ is the Poincar\'e series for elements of the form
	\begin{align*}
		\ka_1\fc{\ta_3}{x_3^{n_3}y_3^{m_3}}x_2^{n_2}y_2^{m_2}
	\end{align*}
	in $\pi_{a'}((\Si^{-\si+b\ep+q\so}H\ul{\Z/2})^{\mathcal{K}}).$ Note that this identification makes sense because all such elements are linearly independent. Indeed, such elements are nonzero by Lemma \ref{iononzero}, and each element of this form in a fixed degree is uniquely determined by $m_3$. The 4-degree of the above cohomology class is $(1-m_2+m_3,-1,n_2+m_2+1,-1-n_3-m_3)$ so this is in our range if $b\geq 1$ and $b+q\geq 0$ ($n_2+m_2\geq n_3+m_3$).
	
	Similarly (see Lemma \ref{k3series} below), when $p \leq -2,$ we have that $x^p(1+\cd+x^{-p-2})(1+\cd+x^{b+q})$ is the Poincar\'e series for elements of the form
	\begin{align*}
		\ka_3^{m_3+1}\cdot \fc{\ta_1}{x_1^{n_1}y_1^{m_3+m_1}}	\cdot x_2^{n_2}y_2^{m_2-m_3},
	\end{align*}
	in $\pi_{a'}((\Si^{p\si+b\ep+q\so}H\ul{\Z/2})^{\mathcal{K}})$, where $m_3\leq m_2$ and $m_1 \geq 1.$ Note again that such elements are linearly independent. Indeed, an element of this form is nonzero by Lemma \ref{iononzero}, since we can always multiply it by $y_3^{m_3}$, and each element of this form in a fixed degree is uniquely determined by $m_1$.  The above cohomology class has $4-$degree $$(m_1-m_2+m_3+1,-1-n_1-m_1,1+n_2+m_2,-m_3-1)$$ so because $m_3\leq m_2$ and $m_1 \geq 1$ we have $p\leq -2, q<0,b+q\geq 0, b\geq 1$. 
	
	By Lemma \ref{x1case1}, we have that the map
	\begin{align*}
		\varphi_{p-1}:=\cp x_1: \pi_{a'}((\Si^{(p-1)\si+b\ep+q\so}H\ul{\Z/2})^{\mathcal{K}}) \twoheadrightarrow  \pi_{a'}((\Si^{p\si+b\ep+q\so}H\ul{\Z/2})^{\mathcal{K}}) 
	\end{align*}
	is surjective, so by Proposition \ref{caseMC2}, the Poincar\'e series for $\ker\varphi_{p-1}$ is given by (see Lemma \ref{kerphi} below)
	\begin{align*}
		x^{p-1}(1+\cd+x^{b+q}).
	\end{align*}
	It follows that this corresponds to the $b+q+1$ elements of the form
	\begin{align*}
		\ka_3^{m_3+1}\cdot \fc{\ta_1}{y_1^{m_3+m_1}}\cdot x_2^{n_2}y_2^{m_2-m_3}
	\end{align*}
	in $\pi_{*}((\Si^{(p-1)\si+b\ep+q\so}H\ul{\Z/2})^{\mathcal{K}})$. We will now construct the basis in the theorem via induction on $p$. When $p = -1,$ we only have products of the form
	\begin{align*}
		\ka_1\fc{\ta_3}{x_3^{n_3}y_3^{m_3}}x_2^{n_2}y_2^{m_2},
	\end{align*}
	so we are done.
	
	Before giving the inductive step for $p \leq -2,$ we explicitly state our basis elements for $p=-2$ to illustrate the idea.
	When $p=-2,$ we first have elements of the form
	\begin{align*}
			\fc{\ka_1}{x_1}\fc{\ta_3}{x_3^{n_3}y_3^{m_3}}x_2^{n_2}y_2^{m_2}
	\end{align*}
	corresponding to elements which map to
	\begin{align*}
				\ka_1\fc{\ta_3}{x_3^{n_3}y_3^{m_3}}x_2^{n_2}y_2^{m_2}\in H^{a-\sigma+b\epsilon+q\sigma\otimes\epsilon}_{Br,K}(pt)
	\end{align*}
	under multiplication with $x_1.$ However, $x_1$ is in general not injective, so we also have as generators the kernel elements which are of the form
	\begin{align*}
		\ka_3^{-q}\cdot \fc{\ta_1}{y_1^{-q}}\cdot x_2^{n_2'}y_2^{m_2'}\in H^{2-m_2'-2\sigma+(n_2'+m_2'-q)\epsilon+q\sigma\otimes\epsilon}_{Br,K}(pt),
	\end{align*}
	where $n_2'+m_2' = b+q$, $2-m_2'=a=-a'$. 
	
	These generators are associated to the Poincar\'e series $x^{-2}(1+..+x^{b+q})$. 
	
	Then if $a'\in[-2,n_2'+m_2'-2]$ the preimage is necessary  $\mathbb{Z}/2$ so the preimage above requires a choice (otherwise the preimage is one element, so the choice is unique). In the case we need a choice, we choose the element in the preimage above to be that unique element in the fiber that has zero multiplication with $y_1y_3^{-q-1}$. 
	
	Denote one cohomology class in the preimage with $\alpha$. Notice then that the preimage of the above cohomology class is $$\{\alpha,\alpha+\ka_3^{-q}\cdot \fc{\ta_1}{y_1^{-q}}\cdot x_2^{n_2'}y_2^{m_2'}\}$$ and $$y_1y_3^{-q-1}\ka_3^{-q}\cdot \fc{\ta_1}{y_1^{-q}}\cdot x_2^{n_2'}y_2^{m_2'}=k_3\theta_1x_2^{n_2'}y_2^{m_2'-q-1}\in H^{2-m_2'+q-\sigma+b\epsilon-\sigma\otimes\epsilon}_{Br,K}(pt)\simeq \Z/2$$ is the only nonzero cohomology class from  the group. 
	
	But then exactly one out of $\alpha$ and $\alpha+\ka_3^{-q}\cdot \fc{\ta_1}{y_1^{-q}}\cdot x_2^{n_2'}y_2^{m_2'}$ has zero multiplication by $y_1y_3^{-q-1}$. This is chosen as a preimage.

	For the inductive step, assume that $\pi_{a'}((\Si^{p\si+b\ep+q\so}H\ul{\Z/2})^{\mathcal{K}})$ has a basis consisting of previously chosen inverses and elements of the form $	\ka_3^{m_3+1}\cdot \fc{\ta_1}{x_1^{n_1}y_1^{m_3+m_1}}	\cdot x_2^{n_2}y_2^{m_2-m_3}$ for all $a'.$ When ${a' \notin [p-1,b+q+p-1],}$ then $\varphi_{p-1}$ is an isomorphism so we are done. Otherwise, $\ker\varphi_{p-1}$ is generated by an element (the only nontrivial element)
	\begin{align*}
		z = \ka_3^{M_3+1}\cdot \fc{\ta_1}{y_1^{M_3+M_1}}\cdot x_2^{N_2}y_2^{M_2-M_3}\in H^{-a'+(p-1)\sigma+b\epsilon+q\sigma\otimes\epsilon}_{Br,K}(pt),
	\end{align*}
	where
	\begin{align*}
		M_1 &= -p, M_3 = -q-1 \txa  M_2 =a'+M_1+M_3+1, N_2= b-M_2-1.
	\end{align*}
	
	For example, note that
	\begin{align*}
	y_1^{-p}y_3^{-q-1}z=y_1^{M_1}y_3^{M_3} \cp z &= 	y_1^{M_1}y_3^{M_3} \cp \ka_3^{M_3+1}\cdot \fc{\ta_1}{y_1^{M_3+M_1}}\cdot x_2^{N_2}y_2^{M_2-M_3} \\
	&= (y_3^{M_3} \ka_3^{M_3}) \ka_3 \cp \fc{\ta_1}{y_1^{M_3}} \cdot x_2^{N_2}y_2^{M_2-M_3}\\
	&=(y_1^{M_3}y_2^{M_3})\ka_3\cdot  \fc{\ta_1}{y_1^{M_3}} \cdot x_2^{N_2}y_2^{M_2-M_3} = \ka_3\ta_1 x_2^{N_2}y_2^{M_2}\neq 0, 
	\end{align*}
	but for other products $\be\in H^{*+(p-1)\sigma+b\epsilon+q\sigma\otimes\epsilon}_{Br,K}(pt)$ we have that 
	\begin{align*}
			&y_1^{-p}y_3^{-q-1}\cp  \be=y_1^{M_1}y_3^{M_3} \cp  \be = y_1^{M_1}y_3^{M_3} \cp \ka_3^{m_3+1}\cdot \fc{\ta_1}{x_1^{n_1}y_1^{m_3+m_1}}\cdot x_2^{n_2}y_2^{m_2-m_3} \\
			&= 	y_1^{M_1}y_3^{M_3} \cp \ka_3^{M_3+1}\cdot \fc{\ta_1}{x_1^{n_1}y_1^{M_3+m_1}}\cdot x_2^{n_2}y_2^{m_2-M_3} \tx{(since $q=-M_3-1=-m_3-1$ is fixed)}\\
			&= y_1^{M_1-m_1}\cp(y_3^{M_3} \ka_3^{M_3}) \ka_3 \cp \fc{\ta_1}{x_1^{n_1}y_1^{M_3}} \cdot x_2^{n_2}y_2^{m_2-M_3}\\
			&= y_1^{M_1-m_1}\cp(y_1^{M_3}y_2^{M_3}) \ka_3 \cp \fc{\ta_1}{x_1^{n_1}y_1^{M_3}} \cdot x_2^{n_2}y_2^{m_2-M_3}\\
			&= y_1^{M_1-m_1}\cp  \ka_3 \cp \fc{\ta_1}{x_1^{n_1}} \cdot x_2^{n_2}y_2^{m_2} = 0,
	\end{align*}
	since $m_1 < M_1$ if $\be \neq z$ because $n_1>0$.
	
	Then our basis $B$ for $\pi_{a'}((\Si^{(p-1)\si+b\ep+q\so}H\ul{\Z/2})^{\mathcal{K}})$ is given by taking the union of $\{z\}$ with all of the products  ($n_1\geq 1$)
	\begin{align*}
		\ka_3^{m_3+1}\cdot \fc{\ta_1}{x_1^{n_1}y_1^{m_3+m_1}}\cdot x_2^{n_2}y_2^{m_2-m_3}\in  \varphi_{p-1}^{-1}\left(	\ka_3^{m_3+1}\cdot \fc{\ta_1}{x_1^{n_1-1}y_1^{m_3+m_1}}\cdot x_2^{n_2}y_2^{m_2-m_3}\right),
	\end{align*}
	and the generators
	\begin{align*}
		&\fc{\ka_1}{x_1^{-p}}\fc{\ta_3}{x_3^{n_3}y_3^{m_3}}x_2^{n_2}y_2^{m_2} \in \varphi_{p-1}^{-1}\left(\fc{\ka_1}{x_1^{-p-1}}\fc{\ta_3}{x_3^{n_3}y_3^{m_3}}x_2^{n_2}y_2^{m_2}\right)
	\end{align*}
	whose product with $y_1^{M_1}y_3^{M_3}$ is zero. Note that this is well-defined because exactly one element of the fiber with two elements has this property. Indeed, we know that $\ker\varphi_{p-1}=\Z/2z$ in this degree, so 
	\begin{align*}
		 \varphi_{p-1}^{-1}\left(\fc{\ka_1}{x_1^{-p-1}}\fc{\ta_3}{x_3^{n_3}y_3^{m_3}}x_2^{n_2}y_2^{m_2}\right) = \{\al,\al+z\},
	\end{align*}
	and we know that $\ka_1\ta_3x_2^{N_2}y_2^{M_2}$ is the only nonzero element in degree $(1-M_2,-1,1+N_2+M_2,-1)$ by  Theorem \ref{caseMC2}(2), so if $\al$ is an arbitrary element of $\ker\varphi_{p-1}$ we must have that if
	\begin{align*}
		&y_1^{M_1}y_3^{M_3}\cp \al \neq  0 \implies y_1^{M_1}y_3^{M_3}\cp \al = \ka_1\ta_3x_2^{N_2}y_2^{M_2}\\
		\implies & y_1^{M_1}y_3^{M_3}\cp(\al+z) = \ka_1\ta_3x_2^{N_2}y_2^{M_2}+\ka_1\ta_3x_2^{N_2}y_2^{M_2}=0.
	\end{align*}
	
	Note that $B$ is linearly independent, since $B\setminus \{z\}$ is linearly independent by induction, and $z$ is the only element of $B$ that does not vanish under multiplication with $y_1^{M_1}y_3^{M_3}$.
	We have previously established that our chosen products involving $\ka_3$ have the Poincar\'e series
	\begin{align*}
		x^{p-1}(1+\cd+x^{-p-1})(1+\cd+x^{b+q}),
	\end{align*}
	so by dimensionality, $B$ is indeed a basis.
\end{proof}

	Recall that the cohomology of the Mixed Cone of Type II corresponding to $\io_2$  ($b\geq 1$, $p,q\leq -1$) is given by the homology of the chain complex
	\begin{align*}
		0\to \ub{\frac{\F_2[x_1,y_1,x_2,y_2,x_3,y_3]}{(x_1^\infty,y_1^\infty,x_3^\infty,y_3^\infty)}\{\io_2\}}_{I} \xrightarrow{f} \ub{\frac{\F_2[x_1,y_1,x_2,y_2,x_3,y_3]}{(x_1^\infty,y_1^\infty,x_3^\infty,y_3^\infty)}\{\ta_1\ta_3\}}_T \to 0.
	\end{align*}
	Write
	\begin{align*}
		f_{a,p,b,q} : I_{a,p,b,q} \to T_{a+1,p,b,q}
	\end{align*}
	for the restriction of $f$ to the graded piece of degree $(a,p,b,q).$ We have now the following theorem:
\begin{theorem} \label{ioseries}
	The Poincar\'e series for $\dim\ker f_{a,p,b,q}$ in the Mixed Cone of Type II corresponding to $\io_2$ is given by:\\
	(1) \textbf{Case} $-p,-q\geq b+1$:
	\begin{align*}
		(x-1)^{-2}x^{-b - 1} (x^b - 1) (x^{b + 1} - 1).
	\end{align*}
	This is the second term of the  case (1) in Proposition \ref{caseMC2} ($a\leq b+1$).
	
	(2) \textbf{Case} $-q\leq b$:
	\begin{align*}
		(x-1)^{-2}\left( x^p(x^{-p-1} - 1) (x^{b +q+ 1} - 1) + x^q(x^b-1)(x^{-q}-1)\right).
	\end{align*}
	This is the Poincar\'e series of the case (2) in Proposition \ref{caseMC2}.
\end{theorem}
As the ideas are similar to the Mixed Cone of Type I case, we defer the reader to Theorem \ref{ioproof}  for the proof.

\begin{corollary}\label{excl}
	(1) When $b+p <0, b+q<0$ the cohomology classes represented by $\ta_1\ta_3$-chains ($a\geq b+4$) and $\io_2$-chains ($a\leq b+1$) live in mutually exclusive degrees. The range $b+2\leq a\leq b+3$ is zero.
	
	(2) When $b+q\geq 0,$ the Poincar\'e series from case (2) Proposition \ref{caseMC2} coincides with case (2) of Theorem \ref{ioseries}. This means that all $\theta_1\theta_3$-classes live in the image of $f$, therefore the cohomology in these degrees is only given by the kernel elements i.e. $\io_2$-chains. It means that the cohomology classes 
	$$\frac{\theta_1}{x_1^{n_1}y_1^{m_1}}\frac{\theta_3}{x_3^{n_3}y_3^{m_3}}x_2^{n_2}y_2^{m_2}$$
	in the $b+q\geq 0$, $b\geq 1$ and $p,q\leq -1$ are zero in Mixed Cone of Type II corresponding to $\io_2$. The same is for $b+p\geq 0$, $b\geq 1$ and $p,q\leq -1$. It also means that Proposition \ref{propcase1v2} gives a description of  the cohomology classes in $Ker f_{-a',p,b,q}$ for $b+q\geq 0, p,q\leq -1$. 
\end{corollary}
\begin{proof}
 We see that the Poincar\'e series in Theorem \ref{ioseries}(1) coincides with the second summand of the Poincar\'e series in case (1) in Proposition \ref{caseMC2}. The result follows from the fact that the two summands live in mutually exclusive cohomological degrees.
\end{proof}

\begin{lemma}
	\label{x1case2v2}
	Let $-q > b \geq 1.$ The map $H_{Br,K}^{a+(p-1)\sigma+b\epsilon+q\sigma\otimes\epsilon}(pt)\to H_{Br,K}^{a+p\sigma+b\epsilon+q\sigma\otimes\epsilon}(pt)$ given by $\al \mt x_1 \cp \al$ is
	\begin{itemize}
		\item  surjective when $b+p < 0$ and $a > b+3,$ 
		\item  an isomorphism when $b+ p \leq 0$ and $a \leq b+1,$
		\item  an isomorphism when $-q = b+1$.
	\end{itemize}
	By symmetry, when $-p > b \geq 1$, we also have that the map $H_{Br,K}^{a+p\sigma+b\epsilon+(q-1)\sigma\otimes\epsilon}(pt)\to H_{Br,K}^{a+p\sigma+b\epsilon+q\sigma\otimes\epsilon}(pt)$ given by $\al \mt x_3 \cp \al$ is
	\begin{itemize}
		\item  surjective when $b+q < 0$ and $a > b+3,$ 
		\item  an isomorphism when $b+ q \leq 0$ and $a \leq b+1,$
		\item  an isomorphism when $-p = b+1$.
	\end{itemize}
\end{lemma}

\begin{proof}
	We prove the result for the map $\al \mt x_1 \cp \al.$ When $b+p <0, b+q<0$ and $a>b+3$ we have from Corollary \ref{excl} that the surjectivity result only applies to cohomology classes arising from $\ta_1\ta_3$-chains. In this case, the result follows immediately since we can always divide each term by $x_1.$ 
	
	The next statement only applies to cohomology classes arising from $\io_2$-chains because $a\leq b+1$. This statement can be reduced, by dimensionality, to showing that the groups $H_{Br,K}^{a+p\sigma+b\epsilon+q\sigma\otimes\epsilon}(\ct)$ in the long exact sequence
	\begin{align*}
		\cd \to H_{Br,K}^{a+(p-1)\sigma+b\epsilon+q\sigma\otimes\epsilon}(pt)\xr{\cp x_1} H_{Br,K}^{a+p\sigma+b\epsilon+q\sigma\otimes\epsilon}(pt) \to H_{Br,K}^{a+p\sigma+b\epsilon+q\sigma\otimes\epsilon}(\ct)\to \cd
	\end{align*}
	are zero.
	We know that (Theorem \ref{middle})
	\begin{align*}
		H_{Br,K}^{a+p\sigma+b\epsilon+q\sigma\otimes\epsilon}(\ct) \xl[\sim]{y_1^{p}(t_1^{-1})^{-q}} H_{Br,K}^{a+p+(b+q)\epsilon}(\ct),
	\end{align*}
	and by assumption we have that $a+p \leq b+p+1\leq 1$ and $b+q<0,$ so
	\begin{align*}
		H_{Br,K}^{a+p+(b+q)\epsilon}(\ct) =H_{Br,\Sigma_2}^{a+p+(b+q)\epsilon}(pt)= 0
	\end{align*}
	as desired. Similarly, the last statement follows from the observation that
	\begin{align*}
		H_{Br,K}^{a+p+(b-(b+1))\epsilon}(\ct) = H_{Br,K}^{a+p-\epsilon}(\ct)=H_{Br,\Sigma_2}^{a+p-\epsilon}(pt) =0.
	\end{align*}
\end{proof}
	In the range of Lemma \ref{x1case2v2}, we can define elements
	\begin{align*}
		\left(\fc{\ka_1}{x_1^{n_2+m_2+n_1+1}}\fc{\ta_3}{x_3^{n_3}y_3^{m_3}}x_2^{n_2}y_2^{m_2}\right)/x_3^{n+1} \in H_{Br,K}^{a+p\sigma+b\epsilon+(-b-n-1)\sigma\otimes\epsilon}(pt)
	\end{align*}
	in the preimages of the elements
	\begin{align*}
\fc{\ka_1}{x_1^{n_2+m_2+n_1+1}}\fc{\ta_3}{x_3^{n_3}y_3^{m_3}}x_2^{n_2}y_2^{m_2} \in H_{Br,K}^{a+p\sigma+b\epsilon-b\sigma\otimes\epsilon}(pt)
	\end{align*}
	under multiplication by $x_3^{n+1},$ as in the proposition below.
	\begin{proposition}
	\label{propcase2v2}
	There is an isomorphism of $\Z/2$-spaces
	\begin{align*}
		\scalebox{0.9}{$
			\bigoplus_{b\geq1,b+p<0,b+q<0}\fc{\pi_{*}((\Si^{p\si+b\ep+q\so}H\ul{\Z/2})^{\mathcal{K}})}{(\ta_1\ta_3\text{-chains})}\cong\Z/2\ub{\left\{\fc{\ka_3^{m_3+1}}{x_3^{n_3+1}}\cdot \fc{\ta_1}{x_1^{n_1}y_1^{m_3+1+n}}\right\}}_{0 \leq n \leq m_3\leq n+n_1}
			\op\Z/2\ub{\left\{\left(\fc{\ka_1}{x_1^{n_2+m_2+n_1+1}}\fc{\ta_3}{x_3^{n_3}y_3^{m_3}}x_2^{n_2}y_2^{m_2}\right)/x_3^{n+1}\right\}}_{n_3+m_3=n_2+m_2}.$}
	\end{align*}
\end{proposition}
\begin{proof}
	By Lemma \ref{x1case2v2}, we have that
	\begin{align*}
		\psi_{q-1}	:= \cp x_3: H_{Br,K}^{a+p\sigma+b\epsilon+(q-1)\sigma\otimes\epsilon}(pt)\to H_{Br,K}^{a+p\sigma+b\epsilon+q\sigma\otimes\epsilon}(pt)
	\end{align*}
	is an isomorphism on cohomology classes arising from $\io_2$-chains ($a\leq b+1$) when $ b+q \leq 0 > b+p$. When $b+q = 0,$ we have from Proposition \ref{propcase1v2} that
	
	\begin{align*}
		\im \psi_{-b-1} = H_{Br,K}^{*+p\sigma+b\epsilon-b\sigma\otimes\epsilon}(pt) \cong \Z/2\ub{\left\{\ka_3^{m_3+1}\cdot \fc{\ta_1}{x_1^{n_1}y_1^{m_3+1+n}}\right\}}_{0 \leq n \leq m_3\leq n+n_1}
		\op\Z/2\ub{\left\{\fc{\ka_1}{x_1^{n_2+m_2+n_1+1}}\fc{\ta_3}{x_3^{n_3}y_3^{m_3}}x_2^{n_2}y_2^{m_2}\right\}}_{n_3+m_3=n_2+m_2}.
	\end{align*}
	Then since $\psi_{-b-1}$ is an isomorphism, we have that the domain of $\psi_{-b-1}$ is generated by the desired elements, and the result follows by induction on $q.$
\end{proof}
In the above range all elements are uniquely divisible by both $x_1$ and $x_3$.
\subsection{The Negative Cone}
In this section, we assume that the index $(a,p,b,q)$ is in the range of the Negative Cone; in particular we have that $p,b,q\leq -1$ and $a\in\Z$. We will determine below the generators of the $\Z/2$-vector spaces $H^{a+p\sigma+b\epsilon+q\sigma\otimes\epsilon}_{Br,K}(pt,\Z/2)$ that belong to the Negative Cone. 

\begin{lemma}
	\label{gen}
	All products of the form
	\begin{align*}
		\ka_1^{-p}\fc{\ta_2}{x_2^{n_2}y_2^{m_2}}\fc{\ta_3}{x_3^{n_3}y_3^{m_3}} \in H_{Br,K}^{a+p\sigma+b\epsilon+q\sigma\otimes\epsilon}(pt)
	\end{align*}
	that belong to the negative cone are nonzero whenever $m_2 \geq -p-1$ or $m_3 \geq -p-1$.
\end{lemma}
\begin{proof}
	
	Denote $n:=-p-1$. Without loss of generality, assume that $m_3 \geq n.$
	Let 
	\begin{align*}
		i_3 &= n_3\\
		j_3 &= m_3-n.
	\end{align*}
	Since the product is in the negative cone, we must have that $n_2+m_2 \geq n.$ 
	
	Without loss of generality, assume that $m_2 \geq n_2$. If $m_2 \geq n,$ set
	\begin{align*}
		j_1 = n, j_2 = m_2-n, i_2=n_2,i_1= 0,
	\end{align*}
	otherwise we have the case $m_2,n_2< n$ and $m_2+n_2\geq n$ so we set 
	\begin{align*}
		j_1 = n-n_2, j_2 = m_2+n_2-n, i_2 = 0,i_1 = n_2.
	\end{align*}
	
	Then
	\begin{align*}
		\ka_1^{-p}\fc{\ta_2}{x_2^{n_2}y_2^{m_2}}\fc{\ta_3}{x_3^{n_3}y_3^{m_3}} = 
		\kappa_1^{i_1+j_1+1} \frac{\theta_2}{x_2^{i_1+i_2} y_2^{j_1+j_2}} \frac{\theta_3}{x_3^{i_3} y_3^{i_1+j_1+j_3}},
	\end{align*}
	and multiplication by $x_1^{i_1} y_1^{j_1} x_2^{i_2} y_2^{j_2} x_3^{i_3} y_3^{j_3}$ results in the expression $\ka_1\theta_2\ta_3=\Theta\neq 0$ using the relations \ref{eq2}.
	
\end{proof}
The following theorem generalizes Proposition 4.25 of \cite{BH} which is the case $p=-1$.

\begin{theorem}
	\label{negcone}
	The negative cone is given by
	\begin{align*}
		\Z/2{\left\{\ka_1^{-p}\fc{\ta_2}{x_2^{n_2}y_2^{m_2}}\fc{\ta_3}{x_3^{n_3}y_3^{m_3}}\right\}_{m_3 \geq -p-1 \text{ or } m_2\geq -p}}.
	\end{align*}
	Note that since we are working in the negative cone, we must implicitly have that $n_2+m_2 \geq -p-1$ and $n_3+m_3 \geq -p -1.$ 
	By symmetry, we can also write the above generators in terms of $\ka_2$ and $\ka_3.$
\end{theorem}
\begin{proof}
	Observe that the result is equivalent to the statement that the negative cone is given by
	\begin{align*}
		\Z/2\left\{{\left\{\ka_1^{-p}\fc{\ta_2}{x_2^{n_2}y_2^{m_2}}\fc{\ta_3}{x_3^{n_3}y_3^{m_3}}\right\}_{m_3 \geq -p-1 \text{ or } m_2\geq -p-1 }}\setminus{\left\{\ka_1^{-p}\fc{\ta_2}{x_2^{n_2}y_2^{-p-1}}\fc{\ta_3}{x_3^{n_3}y_3^{m_3}}\right\}_{m_3 \leq -p-2}}\right\}.
	\end{align*}
	We calculate the Poincar\'e series given by these generators. By Proposition \ref{k1all}, we have that the Poincar\'e series for all elements of the form
	\begin{align*}
		\ka_1^{-p}\fc{\ta_2}{x_2^{n_2}y_2^{m_2}}\fc{\ta_3}{x_3^{n_3}y_3^{m_3}}
	\end{align*}
	is given by
	\begin{align*}
		x^{p+b+q}(1+\cd+x^{-p-b-2})(1+\cd+x^{-p-q-2}).
	\end{align*}
	Subtracting the Poincar\'e series from Proposition \ref{NC}, we have
	
	\begin{align*}
		&x^{p+b+q}(1+\cd+x^{-p-b-2})(1+\cd+x^{-p-q-2})\\
		-&x^{p+b+q}\left[(1+\cd+x^{-p-b-2})(1+\cd+x^{-q-2})+x^{-q-1}(1+\cd+x^{-b-1})(1+\cd+x^{-p-1})\right]\\
		=& x^{p-1}(1+\cd+x^{-p-1})(1+\cd+x^{-p-2}) =: P.
	\end{align*}
	By similar counting arguments (see Proposition \ref{k1less}), we see that the Poincar\'e series for elements of the form
	\begin{align*}
		\Z/2\left\{\ka_1^{-p}\fc{\ta_2}{x_2^{n_2}y_2^{m_2}}\fc{\ta_3}{x_3^{n_3}y_3^{m_3}}\right\}_{m_3 < -p-1 \text{ and } m_2< -p-1 }
	\end{align*}
	is given by 
	\begin{align*}
		P- x^{p-1}(1+\cd+x^{-p-2})=x^{p}(1+...+x^{-p-2})(1+...+x^{-p-2}).
	\end{align*}
	
	But by similar counting arguments again , we have that 
	\begin{align*}
		x^{p-1}(1+\cd+x^{-p-2})
	\end{align*}
	is precisely the Poincar\'e series for products of the form
	\begin{align*}
		{{\ka_1^{-p}\fc{\ta_2}{x_2^{n_2}y_2^{-p-1}}\fc{\ta_3}{x_3^{n_3}y_3^{m_3}}}},
	\end{align*}
	where $m_3 \leq -p-2$.
	Thus, our generators correspond precisely to the Poincar\'e series of Proposition \ref{NC}, so it remains to show that they are linearly independent. Suppose for contradiction that
	\begin{align*}
		\ka_1^{-p}\fc{\ta_2}{x_2^{i_1}y_2^{j_1}}\fc{\ta_3}{x_3^{\ka_1}y_3^{\ell_1}} + \cd +	\ka_1^{-p}\fc{\ta_2}{x_2^{i_M}y_2^{j_M}}\fc{\ta_3}{x_3^{k_M}y_3^{\ell_M}} +  \ka_1^{-p}\fc{\ta_2}{x_2^{i_1'}y_2^{j_1'}}\fc{\ta_3}{x_3^{k_1'}y_3^{\ell_1'}} +\cd + \ka_1^{-p}\fc{\ta_2}{x_2^{i_N'}y_2^{j_N'}}\fc{\ta_3}{x_3^{k_N'}y_3^{\ell_N'}} = 0,
	\end{align*}
	where without loss of generality
	\begin{align*}
		k_n< k_m \txa k_n' &< k_m' \tx{if $n < m$},\\
		\ell_n &\geq -p-1,\\
		\txa \ell_n' & \leq -p-2\tx{ and } j_n'\geq -p
	\end{align*}
	for all $n,m \geq 1.$ Notice that because the elements are in the same degree if $k_n=k_m$ then the elements of the above form are identical. Recall that $\ka_1^{-p}\fc{\ta_2}{x_2^{i_1}y_2^{j_1}}\fc{\ta_3}{x_3^{k_1}y_3^{\ell_1}}$ is in 4-degree $$(p+4+l_1+j_1,p,-p-2-i_1-j_1,-p-2-k_1-l_1),$$ and this is the degree of the above sum.
	
	
	If $N= 0,$ then $M \geq 2.$ But then
	\begin{align*}
		&\ka_1^{-p}\fc{\ta_2}{x_2^{i_1}y_2^{j_1}}\fc{\ta_3}{x_3^{k_1}y_3^{\ell_1}} + \cd +	\ka_1^{-p}\fc{\ta_2}{x_2^{i_M}y_2^{j_M}}\fc{\ta_3}{x_3^{k_M}y_3^{\ell_M}}  = 0\\
		\implies & \ka_1^{-p}\fc{\ta_2}{x_2^{i_1}y_2^{j_1}}\fc{\ta_3}{x_3^{k_1}y_3^{\ell_1}} + \cd +	\ka_1^{-p}\fc{\ta_2}{x_2^{i_{M-1}}y_2^{j_{M-1}}}\fc{\ta_3}{x_3^{k_{M-1}}y_3^{\ell_{M-1}}} = \ka_1^{-p}\fc{\ta_2}{x_2^{i_M}y_2^{j_M}}\fc{\ta_3}{x_3^{k_M}y_3^{\ell_M}}\\
		\implies& x_3^{k_{M}}\left(\ka_1^{-p}\fc{\ta_2}{x_2^{i_1}y_2^{j_1}}\fc{\ta_3}{x_3^{k_1}y_3^{\ell_1}} + \cd +	\ka_1^{-p}\fc{\ta_2}{x_2^{i_{M-1}}y_2^{j_{M-1}}}\fc{\ta_3}{x_3^{k_{M-1}}y_3^{\ell_{M-1}}} \right)=  x_3^{k_{M}}\ka_1^{-p}\fc{\ta_2}{x_2^{i_M}y_2^{j_M}}\fc{\ta_3}{x_3^{k_M}y_3^{\ell_M}}\\
		\implies &\ka_1^{-p}\fc{\ta_2}{x_2^{i_M}y_2^{j_M}}\fc{\ta_3}{y_3^{\ell_M}} = 0,
	\end{align*}
	which contradicts Lemma \ref{gen} because $l_M\geq -p-1$. Similarly, if $M = 0,$ then $N \geq 2.$ Note that the condition
	\begin{align*}
		k_n' < k_m' \tx{ if } n < m
	\end{align*}
	implies that 
	\begin{align*}
		i_n' > i_m' \tx{ if } n < m
	\end{align*}
	so we have that
	\begin{align*}
		&x_2^{i_1'}\left(\ka_1^{-p}\fc{\ta_2}{x_2^{i_1'}y_2^{j_1'}}\fc{\ta_3}{x_3^{k_1'}y_3^{\ell_1'}} +\cd + \ka_1^{-p}\fc{\ta_2}{x_2^{i_N'}y_2^{j_N'}}\fc{\ta_3}{x_3^{k_N'}y_3^{\ell_N'}} \right)=  0\\
		\implies & \ka_1^{-p}\fc{\ta_2}{y_2^{j_1'}}\fc{\ta_3}{x_3^{k_1'}y_3^{\ell_1'}} = 0,
	\end{align*}
	which contradicts Lemma \ref{gen} because $j_1'\geq -p$. 
	
	Now suppose that $M,N\geq 1.$ Since $\ell_n' \leq -p-2,$ we have that $j_n' \geq -p$ for all $n.$ Then we have that
	
	\begin{equation*}
		\resizebox{\textwidth}{!}{%
			$\begin{aligned}
				&x_2^{i_1'+1}\left(\ka_1^{-p}\fc{\ta_2}{x_2^{i_1}y_2^{j_1}}\fc{\ta_3}{x_3^{k_1}y_3^{\ell_1}} + \cd +	\ka_1^{-p}\fc{\ta_2}{x_2^{i_M}y_2^{j_M}}\fc{\ta_3}{x_3^{k_M}y_3^{\ell_M}} +  \ka_1^{-p}\fc{\ta_2}{x_2^{i_1'}y_2^{j_1'}}\fc{\ta_3}{x_3^{k_1'}y_3^{\ell_1'}} +\cd + \ka_1^{-p}\fc{\ta_2}{x_2^{i_N'}y_2^{j_N'}}\fc{\ta_3}{x_3^{k_N'}y_3^{\ell_N'}}\right) = 0,\\
				&\implies x_2^{i_1'+1}\left(\ka_1^{-p}\fc{\ta_2}{x_2^{i_1}y_2^{j_1}}\fc{\ta_3}{x_3^{k_1}y_3^{\ell_1}} + \cd +	\ka_1^{-p}\fc{\ta_2}{x_2^{i_M}y_2^{j_M}}\fc{\ta_3}{x_3^{k_M}y_3^{\ell_M}} \right) = 0,\\
				&\implies \ka_1^{-p}\fc{\ta_2}{x_2^{i_1-i_1'-1}y_2^{j_1}}\fc{\ta_3}{x_3^{k_1}y_3^{\ell_1}} + \cd +	\ka_1^{-p}\fc{\ta_2}{x_2^{i_M-i_1'-1}y_2^{j_M}}\fc{\ta_3}{x_3^{k_M}y_3^{\ell_M}} = 0,
			\end{aligned}$
		}
	\end{equation*}
	since $i_1> i_2> ... > i_M> i_1'>...>i_N'.$ The inequality $i_M>i_1'$ follows from $l_1'< l_M$. But this reduces to the $N=0$ case, so we are done.
\end{proof}

\begin{corollary} \label{ncm1}
	The Negative Cone has zero multiplication with the Mixed Cone of Type II. 
\end{corollary}
\begin{proof}
	Without reducing the generality, consider the case of Mixed cone of Type II corresponding to $\io_2$. The case of a $\theta_1\theta_3$-chain is trivially satisfied. Then we can always multiply by $x_1$ or by $x_3$ to obtain zero. For example
	$$\ka_2^{-p}\frac{\theta_1}{x_1^{n_1}y_1^{m_1}}\frac{\theta_3}{x_3^{n_3}y_3^{m_3}}\frac{k_1}{x_1^{n_1}}\frac{\theta_3}{x_3^{n_3}y_3^{m_3}}x_2^{n_2}y_2^{m_2}=$$
	$$=\ka_2^{-p}\frac{\theta_1}{x_1^{n_1+n_1'}y_1^{m_1}}x_1^{n_1'}\frac{\theta_3}{x_3^{n_3}y_3^{m_3}}\frac{k_1}{x_1^{n_1'}}\frac{\theta_3}{x_3^{n_3'}y_3^{m_3'}}x_2^{n_2'}y_2^{m_2'}=0,$$and similarly with products with the generators from Proposition \ref{propcase2v2}.
	
\end{proof}
Our results in this section also establish the following corollaries.
\begin{corollary} All cohomology classes from the Negative Cone and from the Mixed Cone of Type II are nilpotent. 
\end{corollary}
\begin{proof} The classes in the Negative Cone are obviously nilpotent as they are generated by nilpotent elements. The classes from the Mixed Cone of Type II are nilpotent because of the products in Proposition \ref{po} and the proof of Proposition \ref{propcase2v2} (because multiplication by $x_3$ is an isomorphism in the range of Proposition \ref{propcase2v2} and the output is nilpotent.)
\end{proof}
We notice that the Negative Cone is generated by multiplications of cohomology classes from the Mixed Cone of Type I. Therefore we have the following corollary.
\begin{corollary} \label{ncm3}
	In general, the Negative Cone has nonzero multiplication with the non-nilpotents from Mixed Cone of Type I. 
	
	It has zero multiplication with the nilpotents from the Mixed Cone of Type I. 
\end{corollary}

For example, 
\begin{align*}
	\ka_1^{-p}\fc{\ta_2}{x_2^{n_2}y_2^{m_2}}\fc{\ta_3}{x_3^{n_3}y_3^{m_3}}
\end{align*}
with $\ka_i, x_i,$ or $y_i$ is in general nonzero, but the product with $\theta_i$ is zero. 
\begin{corollary} \label{ncm2} Multiplication in the Negative Cone is always zero. 

There are elements in the Mixed Cone of Type II whose multiplication is non-zero and there are elements in the Mixed Cone of Type II whose multiplication with nilpotents from the Mixed Cone of Type I is non-zero.

There are also nilpotents in the Mixed Cone of Type I that have their multiplication nonzero e.g. $\theta_1\theta_3\neq 0$.
\end{corollary}
We recall that in the case of $RO(C_2)-$graded cohomology of a point the multiplication among the nilpotents is always zero because the set of nilpotents in this ring coincides with the Negative Cone of $RO(C_2)-$graded cohomology of a point.

For example, a product between  a nilpotent in Mixed Cone of Type II and a nilpotent in the Mixed Cone of Type I is 
$$(\ka_1\frac{\theta_3}{y_3^2})\theta_2$$ 
and this multiplication is non-zero because the multiplication with $y_3^2$ gives $\ka_1\theta_2\theta_3=\Theta\neq 0$. None of the above cohomology classes are $C_2$-motivic (see Section 5 for a description of motivic classes).

For example in the Mixed Cone of Type II $$(\ka_1\frac{\theta_3}{y_3^2})(\ka_3\frac{\theta_1}{y_1^2})=y_2^2\frac{\theta_3}{y_3^2}\frac{\theta_1}{y_1^2}\neq 0,$$
because its multiplication with $\ka_2^2$ equals $\theta_1\theta_3\neq 0$. We will prove in Section 5 that if the cohomology classes from the Mixed Cone of Type II are also motivic then their multiplication is necessary zero. In this example $\ka_1\frac{\theta_3}{y_3^2}$ is non-zero and it is not motivic because the corresponding motivic group is zero (Notice also that $\ka_1$ is motivic, but  $\frac{\theta_3}{y_3^2}$ is not motivic as we will see in Section 5).

Combining the results of \cite{BH} and \cite{KHo} with the results of this section, we obtain the following theorem which describes the additive generators of the Poincar\'e series of $H_{Br,K}^\st(pt).$
\begin{theorem} \label{tc}
	The additive structure of the Poincar\'e series appearing in Theorems \ref{PC}-\ref{NC} in terms of the generators of $H_{Br,K}^\st(pt)$ is given as follows:

\sss*{The Two-Index Case}
Let $l,n\geq 0$ and $i,j\geq 0$. Let $\alpha$ and $\beta$ be two distinct nontrivial irreducible $C_2\times \Sigma_2$-representations. The Poincar\'e series for $H^{-*+V}_{Br}(pt,\Z/2)$ is

a) If $V=0$ then $1$.

b) If $V=n\alpha$ then $1+x+x^2+...+x^n$.

c) If $V=-n\alpha$ then $x^{-n}+...+x^{-3}+x^{-2}$.

d) If $V=n\alpha+l\beta$ then $(1+x+...+x^n)(1+x+...+x^l)$.

e) If $V=n\alpha-j\beta$ then $(1+x+...+x^n)(x^{-j}+...+x^{-2})$.

f) If $V=-n\alpha-j\beta$ then $(x^{-n}+...+x^{-2})(x^{-j}+...+x^{-2})$.

Cases a), b), and c) correspond to the case where $H_{Br, G}^\st(pt) \hra H^{\st}_{Br,K}(pt,\Z/2),$
where $G\hra K$ is a cyclic subgroup of order $2$. For cases d), e), and f), we can assume without loss of generality that $\al = \si$ and $\be = \so.$ Then the series in d)  (i.e. any for $n,l\geq 0$) corresponds to the ring 
\begin{align*}
	\Z/2[x_1,y_1,x_3,y_3].
\end{align*}
The series in e) corresponds to elements of the form
\begin{align*}
	x_1^{n_1}{y_1}^{m_1} \fc{\ta_3}{x_3^{n_3}y_3^{m_3}}
\end{align*}
and the series in f) corresponds to elements of the form
\begin{align*}
	\fc{\ta_1}{x_1^{n_1}{y_1}^{m_1}}\fc{\ta_3}{x_3^{n_3}y_3^{m_3}}
\end{align*}
\sss*{The Positive Cone}
Let $p,b,q\geq 0$. The Poincar\'e series for $H^{-*+p\si+b\ep+q\so}_{Br,K}(pt,\Z/2)$ is given by
	$$(1+x+...+x^p)(1+x+...+x^b)+x(1+x+...+x^{p+b})(1+...+x^{q-1})$$
	and the entire positive cone corresponds to the subring
	\begin{align*}
			\frac{\F_2[x_1,y_1,x_2,y_2,x_3,y_3]}{(x_1y_2y_3+y_1x_2y_3+y_1y_2x_3)}.
	\end{align*}
\sss*{The Mixed Cone of Type I}
Without loss of generality, we consider the Mixed Cone of Type I corresponding to $\ka_2.$

Let $p,q\geq 0,b\leq -1$. If $-b\leq p,q$ then the Poincar\'e series for $H^{-*+p\si+b\ep+q\so}_{Br,K}(pt,\Z/2)$ is
	$$(\frac{1}{x^{-b}}+...+\frac{1}{x})(1+x+...+x^{-b-2})+x^{-b}(1+...+x^{p+b})(1+...+x^{q+b}). $$ 
	In the case $-b>p$ the Poincar\'e series for $H^{-*+p\si+b\ep+q\so}_{Br,K}(pt,\Z/2)$ is
	$$\frac{1}{x^{p+1}}(1+...+x^p)(1+...+x^{p-1})+\frac{1}{x^{-b}}(1+...+x^{-b-p-2})(1+...+x^{p+q}).$$
	Swapping the role of $p$ and $q$ gives the case $-b>q$. Then
		\begin{align*}
		\oplus_{p,q\geq 0,b\leq -1,a\in\Z}H^{a+p\sigma+b\epsilon+q\sigma\otimes\epsilon}_{Br,K}(pt)=\ka_2\Z/2[x_1,y_1,x_3,y_3,\ka_2] \op  \fc{T}{(x_1y_2y_3+y_1x_2y_3+y_1y_2x_3)T},
	\end{align*}
	where
		\begin{align*}
		T := \frac{\Z/2[x_1,y_1,x_2,y_2,x_3,y_3]}{(x_2^\infty,y_2^\infty)}\{\theta_2\}.
	\end{align*}
\sss*{The Mixed Cone of Type II}
Without loss of generality, we consider the Mixed Cone of Type II corresponding to $\io_2.$ Let $p,q\leq -1, b\geq 0$. Then the Poincar\'e series for $H^{-*+p\si+b\ep+q\so}_{Br}(pt,\Z/2)$ is 
	$$\frac{1}{x^{-p-q-b}}(1+...+x^{-p-b-2})(1+...+x^{-q-b-2})+\frac{1}{x^{b+1}}(1+...+x^b)(1+...+x^{b-1}),$$
	corresponding to the elements of
	\begin{align*}
		\scalebox{0.9}{$
			\bigoplus_{b\geq1,b+p<0,b+q<0}\fc{\pi_{*}((\Si^{p\si+b\ep+q\so}H\ul{\Z/2})^{\mathcal{K}})}{(\ta_1\ta_3\text{-chains})}\cong\Z/2\ub{\left\{\fc{\ka_3^{m_3+1}}{x_3^{n_3+1}}\cdot \fc{\ta_1}{x_1^{n_1}y_1^{m_3+1+n}}\right\}}_{0 \leq n \leq m_3\leq n+n_1}
			\op\Z/2\ub{\left\{\left(\fc{\ka_1}{x_1^{n_2+m_2+n_1+1}}\fc{\ta_3}{x_3^{n_3}y_3^{m_3}}x_2^{n_2}y_2^{m_2}\right)/x_3^{n+1}\right\}}_{n_3+m_3=n_2+m_2}$}
	\end{align*}
	if $-p,-q\geq b+1$ or 
	$$\frac{1}{x^{-p}}(1+...+x^{-p-2})(1+...+x^{b+q})+\frac{1}{x^{-q}}(1+...+x^{b-1})(1+...+x^{-q-1}),$$
	corresponding to the elements of
	\begin{align*}
		\bigoplus_{q\leq-1,p\leq-1,b\geq1,b+q\geq0}\pi_{*}((\Si^{p\si+b\ep+q\so}H\ul{\Z/2})^{\mathcal{K}})&\cong \Z/2\left\{\ka_3^{m_3+1}\cdot \fc{\ta_1}{x_1^{n_1}y_1^{m_3+m_1}}\cdot x_2^{n_2}y_2^{m_2-m_3}\right\}_{m_3\leq m_2,m_1\geq 1}\\
		&\op\Z/2\left\{\fc{\ka_1}{x_1^{n_1}}\fc{\ta_3}{x_3^{n_3}y_3^{m_3}}x_2^{n_2}y_2^{m_2}\right\}_{n_3+m_3\leq n_2+m_2}.
	\end{align*}
	if $b\geq -q$. Swapping the role of $p$ and $q$ gives the case $b\geq -p$.
	The $\theta_1\theta_3$-chains can be nonzero only in the first case and are generated by elements of the form $$\frac{\theta_1}{x_1^{n_1}{y_1^{m_1}}}\frac{\theta_3}{x_3^{n_3}{y_3^{m_3}}}x_2^{n_2}y_2^{m_2}$$ with $n_2+m_2\leq n_3+m_3+2$ and $n_2+m_2\leq n_1+m_1+2$ modulo the multiplication by $f=x_1y_2y_3+y_1x_2y_3+y_1y_2x_3$. In the particular case $b=0$ and $p,q\leq -1$ the only cohomology classes generators are $$\frac{\theta_1}{x_1^{n_1}{y_1^{m_1}}}\frac{\theta_3}{x_3^{n_3}{y_3^{m_3}}}$$ for any $n_1,m_1,n_2,m_2\geq 0.$
\sss*{The Negative Cone}
Let $p,b,q\leq -1$. The Poincar\'e series for $H^{-*+p\si+b\ep+q\so}_{Br,K}(pt,\Z/2)$ is given by
	$$\frac{1}{x^{-p-b-q}}[(1+x+...+x^{-b-q-2})(1+...+x^{-p-2})+x^{-p-1}(1+...+x^{-q-1})(1+...+x^{-b-1})]$$
	and the entire negative cone corresponds to the elements of
	\begin{align*}
	\Z/2{\left\{\ka_1^{-p}\fc{\ta_2}{x_2^{n_2}y_2^{m_2}}\fc{\ta_3}{x_3^{n_3}y_3^{m_3}}\right\}_{m_3 \geq -p-1 \text{ or } m_2\geq -p}}.
\end{align*}
that belong to the negative cone i.e. $n_2+m_2\geq -p-1$ and $n_3+m_3\geq -p-1$. In particular, we have $\Theta=\kappa_1\theta_2\theta_3=\iota_3\theta_3\in H^{3-\sigma-\epsilon-\sigma\otimes\epsilon}_{Br,K}(pt,\Z/2)$. By symmetry, we have two more equivalent sets of generators.
\sss*{Multiplicative structure} The multiplication of the above generators fulfill the following relations (\cite{BH}):
\begin{equation} \label{rl}
\begin{aligned}
          \ka_i\theta_j=\ka_j\theta_i,\\
          \kappa_i x_i= x_j y_k + y_j x_k,\\
	  \kappa_i y_i= y_j y_k,\\
	  \kappa_i\kappa_j=y_k^2\\
        \kappa_i\theta_j=\kappa_j\theta_i,\\
         \theta_ix_i=\theta_iy_i=0,\\
         \theta_i^2=0.\\
\end{aligned}
\end{equation}
\end{theorem}
\begin{remark}
All elements of the $RO(C_2\times\Sigma_2)-$graded cohomology ring of a point can be expressed as sums of multiplications of the given generators modulo the relations \ref{rl}.  Notice that 
\begin{equation}\label{eq9}
\begin{aligned}
    & x_1 y_2 y_3 + y_1 x_2 y_3 + y_1 y_2 x _3 = 0 \\
    & \iff  x_1 y_2 y_3 + y_1 x_2 y_3 = y_1 y_2 x _3\\
    & \iff y_3 (x_1 y_2+y_1 x_2) = (y_1 y_2) x_3\\
    & \iff y_3 \kappa_3 x_3 = \kappa_3 y_3 x_3.
\end{aligned}
\end{equation}
therefore the relations \ref{rl} imply that $x_1 y_2 y_3 + y_1 x_2 y_3 + y_1 y_2 x _3 = 0$.

One can ask a more specific question: how can we express a certain multiplication in terms of the given additive generators from the corresponding cohomology group? Some multiplications are easy to determine like all from the Mixed Cone of Type I corresponding to $\ka_2$. For example
$$\ka_2\frac{\theta_2}{x_2y_2}x_1y_1x_3y_3=\ka_2y_2\frac{\theta_2}{x_2y^2_2}x_1y_1x_3y_3=\frac{\theta_2}{x_2y^2_2}x_1y^2_1x_3y^2_3$$
using the relation $\ka_2y_2=y_1y_3$.

Others are harder. Some nontrivial multiplications between generators of the Positive Cone (e.g. $y_1,y_3,x_3$) and generators of Mixed Cone of Type II are given in the Appendix in terms of generators (of Mixed Cone of Type II) because they are used in the previous sections, but the process is tedious and we don't attempt in this paper to answer completely this question.
\end{remark}

\section{$RO(C_2\times \Sigma_2)$-graded cohomology of $E_{\Sigma_2}C_2$}
             In this section, we apply the previous results to compute and discuss the Mackey functor $\ul{H}_{Br,K}^\st(E_{\Sigma_2}C_2,\Z/2)$ and in particular the $RO(C_2\times \Sigma_2)$-graded cohomology of $E_{\Sigma_2}C_2$ with $\Z/2$ coefficients.
             \subsection{The middle and lower levels of the Mackey functor $\ul{H}_{Br,K}^\st(\ek,\Z/2)$}

		\begin{lemma}
			Let $M$ be a constant Mackey functor. For any based $K$-space $X$, we have
			\begin{align*}
				\hr_{Br,K}^{a+p\si+b\ep+q\si\ot\ep}({C_2}_+\wedge X;M) \simeq \hr_{Br,\Si_2}^{a+p+(b+q)\epsilon}(X;M).
			\end{align*}
			\label{orb}
			By symmetry, we have analogous statements for the other two orbits.
					\end{lemma}
		We now compute the middle level of the Mackey functor $\ul{H}^\st_{Br,K}(\ek,\Z/2).$
		From Lemma \ref{orb}, we have
		\begin{align*}
			\hr^{a+p\sigma+b\epsilon+q\sigma\otimes\epsilon}_{Br,K}({C_2}_+ \wedge \ec) &= \hr^{a+p\sigma+b\epsilon+q\sigma\otimes\epsilon}_{Br,K}({K/\Si_2}_+ \wedge \ec) = \hr^{a+p+(b+q)\epsilon}_{Br,\Si_2} (S(\infty\rho_{\Sigma_2})_+),\\
			\hr^{a+p\sigma+b\epsilon+q\sigma\otimes\epsilon}_{Br,K}({\Si_2}_+ \wedge \ec) &= \hr^{a+p\sigma+b\epsilon+q\sigma\otimes\epsilon}_{Br,K}({K/C_2}_+ \wedge \ec) = \hr^{a+b+(p+q)\sigma}_{Br,C_2} (S(\infty\sigma)_+),\\
			\hr^{a+p\sigma+b\epsilon+q\sigma\otimes\epsilon}_{Br,K}({(C_2 \ti \Si_2)/\Delta}_+ \wedge \ec) &= \hr^{a+p\sigma+b\epsilon+q\sigma\otimes\epsilon}_{Br,K}({K/\Delta}_+ \wedge \ec) = \hr^{a+q+(p+b)\sigma\otimes\epsilon}_{Br,\Delta} (S(\infty\rho_{\Delta})_+).
		\end{align*}
		Note that in $S(\si+\si\ot\ep)$, the subgroup $\Si_2$ acts trivially on $\si,$ so identifying $\si \ot \ep$ with the sign representation of $\Si_2$ we obtain the regular representation $\rho_{\Sigma_2}=1+\epsilon$ in the first line above. Similarly, the subgroup $\Delta$ acts trivially on $\si \ot \ep,$ so identifying $\si$ with the sign representation of $\Delta,$ we obtain the regular representation $\rho_{\Delta}=1+\sigma\otimes\epsilon$ in the third line above.
		
		Now we can generalise the classical argument to show that $S(\infty \rho_{\Sigma_2})$ is equivariantly contractible. We regard $S(n\rho_{\Sigma_2})$ as a subspace of $\R^{n\rho_{\Sigma_2}}=\R\oplus\R^\ep\oplus \R \oplus \R^\ep \oplus \cdots \oplus \R \oplus \R^\ep$ (the case $\sigma\otimes\epsilon$ is identical). Let  
		\begin{align*}
			T: S(\infty \rho_{\Sigma_2}) &\to S(\infty \rho_{\Sigma_2})\\
			(x_1,x_2,\ld) &\mt (0,0,x_1,x_2,\ld)
		\end{align*}
		be the map that shifts all coordinates down twice. Note that $T$ is $\Sigma_2-$equivariant as an endomorphism of $\R^{\infty\rho_{\Sigma_2}}.$ Indeed, since the sign and trivial representations are alternating in $\R^{\infty\rho_{\Sigma_2}}$ under our identification, any shift by an even number will be equivariant since the coordinates corresponding to the trivial representations will still be trivial in the image, and similarly for the sign representations. Then for any $x\in S(\infty \rho_{\Sigma_2})$, we have that
		\begin{align*}
			T(x) \neq \lambda x,\lambda\in \mathbb{R},
		\end{align*}
		so 
		\begin{align*}
			\{(1-t)x+tT(x):t \in [0,1]\} \cap \{(0,0,\ld)\} = \emptyset.
		\end{align*}
		We can then define a homotopy from the identity on $S(\infty \rho_{\Sigma_2})$ to $T$:
		\begin{align*}
			H(x,t) = \frac{(1-t)x+tT(x)}{\norm{(1-t)x+tT(x)}}.
		\end{align*}
		The numerator is equivariant, since the group action is a linear map given by a diagonal matrix, which commutes with $T.$ Then we also have that the denominator is invariant under the group action, since the action is orthogonal. Hence $H$ is equivariant.
		
		Then we can equivariantly contract the image of $T$, which is a codimension 2 sphere, to the point $p:=(1,0,0,0,0,...)$ via the homotopy
		\begin{align*}
					H(x,t) = \frac{(1-t)T(x)+tp}{\norm{(1-t)T(x)+tp}}.
		\end{align*}
		 Note that this is possible since the first coordinate is the trivial representation; if, instead, we used only the sign representation (e.g. for $EC_2 = S(\infty \si)$), then this path would not be equivariant.

		Then we have
		\begin{align*}
			\hr^{a+p\si+b\ep+q\si\ot\ep}_{Br,K}({C_2}_+ \wedge \ec) =  \hr^{a+p+(b+q)\epsilon}_{Br,\Si_2} (pt_+),
		\end{align*}
		and
		\begin{align*}
			\hr^{a+p\si+b\ep+q\si\ot\ep}_{Br,K}({(C_2 \ti \Si_2)/\Delta}_+\wedge \ec) =  \hr^{a+q+(b+p)\sigma\otimes\epsilon}_{Br,\Delta} (pt_+).
		\end{align*}
		Since $S(\infty \sigma) \simeq EC_2,$ we have that
		\begin{align*}
			\hr^{a+p\si+b\ep+q\si\ot\ep}_{Br,K}({\Si_2}_+ \wedge \ec) = H^{a+b+(p+q)\sigma}_{Br,C_2} (EC_2).
		\end{align*}
		\subsection{The ring structure of the middle level of the Mackey functor $\ul{H}_{Br,K}^\st(E_{\Sigma_2}C_2,\Z/2)$}
		The ring map induced by the projection ${\Si_2}_+ \wedge {E_{\Si_2}C_2}_+\to {\Si_2}_+$ preserves units, so we have that the images of $y_2$ and $t_2$ are units in $\hr^\star_{Br,K}({\Si_2}_+ \wedge {E_{\Si_2}C_2}_+)$ denoted by the same symbols. For the notations $y_i, t_i$ with $i\in\{1,2,3\}$ see \ref{eq4}.
		
		Applying Theorem \ref{middle} in the other two cases, we have the following isomorphisms of rings 
		\begin{theorem} \label{ml}
		\begin{align*}
			\hr^{\st}_{Br,K}({C_2}_+ \wedge \ec) &\cong  \hr^{\st}_{Br,\Si_2} (pt_+)[y_1^{\pm 1},t_1^{\pm 1}],\\
			\hr^{\st}_{Br,K}({\Si_2}_+ \wedge \ec) &\cong H^{\st}_{Br,C_2} (EC_2)[y_2^{\pm 1},t_2^{\pm 1}],\\
			\hr^{\st}_{Br,K}({(C_2 \ti \Si_2)/\Delta}_+\wedge \ec) &\cong  \hr^{\st}_{Br,\Delta} (pt_+)[y_3^{\pm 1},t_3^{\pm 1}].
		\end{align*}
		\end{theorem}

		We see that the statement from the Theorem \ref{ml}  is given by (after taking direct sums over all $a,p,q,a',p',q'$)
		\begin{align*}
			\hr^{a+*\sigma+*\epsilon+q\sigma\otimes\epsilon}_{Br,K}({C_2}_+ \wedge \ec) &\cong  \hr^{a'+q'\epsilon}_{Br,\Si_2} (pt_+)[y_1^{\pm 1},t_1^{\pm 1}],\\
			\hr^{a+p\sigma+*\epsilon+*\sigma\otimes\epsilon}_{Br,K}({\Si_2}_+ \wedge \ec) &\cong H^{a'+p'\sigma}_{Br,C_2} (EC_2)[y_2^{\pm 1},t_2^{\pm 1}],\\
			\hr^{a+p\sigma+*\epsilon+*\sigma\otimes\epsilon}_{Br,K}({(C_2 \ti \Si_2)/\Delta}_+\wedge \ec) &\cong  \hr^{a'+p'\sigma\otimes\epsilon}_{Br,\Delta} (pt_+)[y_3^{\pm 1},t_3^{\pm 1}].
		\end{align*}
		That is to say, an element in degree $(a,*,*,q)$ can be reduced to an element in degree $(a',0,0,q')$ via multiplication with $y_1$ and $t_1$.  The analogous statements apply for the other two cases. Note that we made a choice in the index to reduce, since $|t_1| = (0,0,-1,1),$ but this does not affect the form of the expression on the right.\\
		Furthermore, in the second case, we have that $$H^{\st}_{Br,C_2}(EC_2) \cong H^*_{sing}(\rp^\infty)[y_1^{\pm 1}]=\Z/2[\tau,y_1^{\pm 1}]$$ where $|y_1|=(-1,1,0,0),$ and $|\tau| = (1,0,0,0),$ so we have the isomorphism of rings
		\begin{align*}
			\hr^{\st}_{Br,K}({\Si_2}_+ \wedge \ec) &\cong \Z/2[\tau,y_1^{\pm 1},y_2^{\pm 1},t_2^{\pm 1}].
		\end{align*}
		For the lower level, we first observe that
		\begin{align*}
			K_+ \sh \ek_+ \se K_+
		\end{align*}
		as $K$-spaces, so
		\begin{align*}
			\hr_{Br,K}^\st(K_+\sh \ek_+) \simeq \hr_{Br,K}^\st(K_+).
		\end{align*}
		We also know \cite{Car:op} that 
		\begin{align*}
			\hr_{Br,K}^p(K_+) \simeq \hr_{sing}^{p}(pt_+)
		\end{align*}
		for all $p\in \Z,$ so we have that
		\begin{align*}
			\hr_{Br,K}^\st(K_+) = \Z/2[y_1^{\pm1},y_2^{\pm1},y_3^{\pm1}],
		\end{align*}
		since the restriction maps in a Mackey functor are ring maps. 
		\begin{proposition} We have
		$$ \hr_{Br,K}^\st(K_+\sh \ek_+)\simeq \Z/2[y_1^{\pm1},y_2^{\pm1},y_3^{\pm1}],$$
          \end{proposition}

		\subsection{The top level of the Mackey functor $\ul{H}_{Br,K}^\st(\ek,\Z/2)$}
		    \begin{lemma}\label{quot}
				Let $M$ be a constant Mackey functor. For any $K$-space $Y$, we have 
				\begin{align*}
					H^{a+b\ep}_{Br,K}(Y;M) &\simeq H^{a+b\epsilon}_{Br,\Si_2} (Y/{C_2};M),\\
					H^{a+p\si}_{Br,K}(Y;M) &\simeq H^{a+p\sigma}_{Br,C_2} (Y/{\Si_2};M),\text{ and }\\
					H^{a+q\si\ot\ep}_{Br,K}(Y;M) &\simeq H^{a+q\sigma\otimes\epsilon}_{Br,K/\Delta} (Y/{\Delta};M).
				\end{align*}
			\end{lemma}

%
%
The next theorem is the main theorem of this section.
\begin{theorem} \label{cEC} As rings,
$$H^{\st}_{Br,K}(\ek,\Z/2)\simeq\frac{H^{*+*\ep}_{Br,\Sigma_2}(pt,\Z/2)[x_1,y_1,x_3,y_3,\ka_2^{\pm 1}]}{(\ka_2x_2=x_1y_3+x_3y_1,\ka_2y_2=y_1y_3)}.$$
\end{theorem}
\begin{proof}
		From Theorem \ref{vanishing}, it follows that the short exact sequence of $\Z/2$-vector spaces
		\begin{align*}
			0\to\hc{a+p\si+b\ep+q\so}_{Br,K}(pt) \xrightarrow{g^*}\hc{a+p\si+b\ep+q\so}_{Br,K}(\ek)\to \tc{a+1+b\ep}_{Br,K}(\et)\to  0
			\tag{S1}
		\end{align*}
		splits as $\Z/2$-vector spaces for $p,q\geq 0.$ We also have that the Mixed Cone of Type I corresponding to $\ka_2$,  along with the positive cone of $\hs_{Br,K}(pt)$, injects into $\hs_{Br,K}(\ek)$ as a subring, since $g^*$ is a ring map and by Theorem \ref{vanishing}. We note that $g^*$ is not always injective, since, for example, we have that
		\begin{align*}
			g^*(\io_2) = 0,
		\end{align*}
		because $\io_2$ is the image of the element $\fc{\susp c}{x_1x_3}$ under the previous map in the $C_2\times\Sigma_2-$topological isotropy sequence; see Lemma \ref{xid}.
		
		We recall that we denote the direct sum of these cones as $$IP=\oplus_{p,q\geq 0}\hc{a+p\si+b\ep+q\so}_{Br,K}(pt).$$ Notice that this is a subring in $H^\st_{Br,K}(pt)$ which is a direct sum of two subrings $IP_1=\oplus_{p,q\geq 0,b\geq 0}\hc{a+p\si+b\ep+q\so}_{Br,K}(pt)$ and $IP_2=\oplus_{p,q\geq 0,b<0}\hc{a+p\si+b\ep+q\so}_{Br,K}(pt)$. The subring $IP_1$ is the positive cone
		$$IP_1=\frac{\Z/2[x_1,y_1,x_2,y_2,x_3,y_3]}{(x_1y_2y_3+y_1x_2y_3+y_1y_2x_3)}.$$
	Denote $f=x_1y_2y_3+y_1x_2y_3+y_1y_2x_3$. By Theorem \ref{kconj}, we have
		\begin{align*}
			IP_2 = \ka_2\Z/2[x_1,y_1,x_2,y_2,x_3,y_3,\ka_2]\bigoplus \frac{T}{fT}\mathlarger{\mathlarger{\mathlarger{/\sim}}}\subset H^{\st,\st}_{Br,K}(E_{\Sigma_2}C_2),
		\end{align*}
		where the relations are 
		$$\ka_2x_2=x_1y_3+x_3y_1,$$
		$$\ka_2y_2=y_1y_3.$$
and $T=\Z/2\left\{\fc{\ta_2}{x_2^{n_2}y_2^{m_2}}\right\}_{n_2,m_2\geq 0}[x_1,y_1,x_2,y_2,x_3,y_3]$ as in Theorem \ref{kconj}. Notice that in $IP$ we have that $f$ is zero. This is because in the view of the relations \ref{eq2} we have $$f=x_1y_2y_3+y_1x_2y_3+y_1y_2x_3=y_2(x_1y_3+y_1x_3)+y_1x_2y_3=$$
$$=\ka_2x_2y_2+y_1x_2y_3=\ka_2x_2y_2+\ka_2y_2x_2=0.$$
It implies that $$IP=\frac{H^{*+*\ep}_{Br,\Sigma_2}(pt,\Z/2)[x_1,y_1,x_3,y_3,\ka_2]}{(\ka_2x_2=x_1y_3+x_3y_1,\ka_2y_2=y_1y_3)}.$$
		
		 It is obvious that $$IP[\ka_2^{-1}]\subset H^{\st,\st}_{Br,K}(E_{\Sigma_2}C_2),$$ because $\ka_2$ is an invertible element in $H^{\st,\st}_{Br,K}(E_{\Sigma_2}C_2)$ being the image of the invertible motivic element $u\in H^{2\sigma-2,\sigma-1}_{C_2}(\textbf{E}C_2)$ \cite{HOV1}. Now suppose that $\al \in H_{Br,K}^\st(\ek)$.
		Let
		\begin{align*}
			|\ka_2^n \al | = (-n,n,-n,n)+|\al|.
		\end{align*}
		Then by Theorem \ref{et}, we know that 
		\begin{align*}
			\hr^{|\ka_2^n \al |}_{Br,K}(\et) \cong \hr_{Br,\Sigma_2}^{-n-1-n\ep+|\al|\sd}(\bc) = 0
		\end{align*}
		 where $|\al|\sd$ is the $RO(\Si_2)$-degree of $\al.$ For sufficiently large $n$, we have that 
		\begin{align*}
			|\ka_2^n \al | = a+p\si+b\ep+q\so,
		\end{align*}
		where $p,q \geq 0.$ Thus, for sufficiently large $n$, the sequence (S1) reduces to 
		\begin{align*}
						0\to\hc{|\ka_2^n \al |}_{Br,K}(pt) \xrightarrow[\sim]{g^*}\hc{|\ka_2^n \al |}_{Br,K}(\ek)\to  0.
		\end{align*}
		Then we have that if $\alpha\in H_{Br,K}^\st(\ek)$  then $\ka_2^n\al\in IP$ for some $n>>0$ and therefore we conclude that  $$H^{*,*}_{Br,K}(E_{\Sigma_2}C_2)\subset IP[\ka_2^{-1}].$$
		
	In conclusion we have 
				\begin{align*}
				\hs_{Br,K}(\ek){\se} IP[\ka_2^{-1}]
			\end{align*}
	which concludes the theorem.
\end{proof}
		We have now determined all the levels of the Mackey functor $\ul{H}^{*,*}_{Br,K}(E_{\Sigma_2}C_2)$.
		\newline		
		
		The following remark describes the Borel motivic cohomology cycles in the above ring. 
		
		\begin{remark} \label{Breal} 
			In \cite{DV1}, we showed that over the real numbers the realization map
			$$H^{\st,\st}_{C_2}(\EG C _2,\Z/2)\simeq H^{*+*\sigma} _{Br,C_2}(pt,\Z/2)[x _3,y _3,\kappa_2^{\pm 1}] \hra \hs_{Br,K}(\ek)$$
			is injective. The generators $x_i,y_i,$ and $\ka_2^{\pm 1}$ all map naturally into themselves in $\hs_{Br,K}(\ek)$ via the realization. However, the generators $\fc{\ta_1}{x_1^{n_1}y_1^{m_1}}$ are in general identified with a non-trivial linear combination of nilpotents in the codomain. For example, from the isotropy sequence we see that the element $\theta_1$ is identified with the nilpotent element $\ta_2 y_3^2 \ka_2^{-2}.$
		\end{remark}
		
		
	\section{Bredon motivic cohomology ring of real numbers}
	In this section, we apply our previous results to describe  the Bredon motivic cohomology ring of the real numbers in terms of cohomology classes of $RO(C_2\times\Sigma_2)-$graded cohomology of a point.
	
	Let $$R:=\oplus_{b\geq 0,b+q\geq 0}H^{a+p\sigma,b+q\sigma}_{C_2}(\R,\Z/2)\hookrightarrow H^{\st,\st}_{C_2}(\R,\Z/2),$$
which is a cohomology subring on which the realization maps give an isomorphism. 

We know from \cite{DV1} that
$R$ is ring isomorphic to a subring in the $RO(C_2\times \Sigma_2)$-graded cohomology of a point. We have that the subring
$$R\simeq \oplus_{a,p\in \mathbb{Z},b\geq 0,b+q\geq 0} H^{a+p\sigma+b\ep+q\sigma\otimes\ep}_{Br,K}(pt,\Z/2)$$
is given by a direct sum of four distinct pieces depending on the signs of $p$ and $q$. 

One piece that is contained above is simply the positive cone (see Theorem \ref{pcon}), which corresponds to the case where $p,q\geq 0$ (topological; the motivic relation is $p\geq q\geq 0$) i.e.
$$\frac{\Z/2[x_1,y_1,x_2,y_2,x_3,y_3]}{(x_1y_2y_3+y_1x_2y_3+y_1y_2x_3)}.$$
However, because the other pieces of $R$ contain $\ka_1$ and $\ka_3$ (one of these is actually enough) the above relation is trivial and this piece becomes  $\Z/2[x_1,y_1,x_2,y_2,x_3,y_3]$.
For example, if we have the cohomology class $\ka_3$ then the relations \ref{eq2} and \ref{eq9} imply that the above relation is trivial.

The second piece is given by the Mixed Cone of Type I corresponding to $\ka_1$. This is described in Theorem \ref{kconj} as
		\begin{align*}
			\ka_1\Z/2[x_2,y_2,x_3,y_3,\ka_1] \op  \fc{T}{(x_1y_2y_3+y_1x_2y_3+y_1y_2x_3)T},
		\end{align*}
where
\begin{align*}
			T := \frac{\Z/2[x_1,y_1,x_2,y_2,x_3,y_3]}{(x_1^\infty,y_1^\infty)}\{\theta_1\}.
		\end{align*}
The third piece is the part of the Mixed Cone of Type I corresponding to $\ka_3$ subject to the condition $b+q\geq 0$. Applying Theorem \ref{kconj} again, this is given by
\begin{align*}
			\ka_3\Z/2[x_2,y_2,x_1,y_1,\ka_3] \op  \fc{T}{(x_1y_2y_3+y_1x_2y_3+y_1y_2x_3)T}.
		\end{align*}
where
\begin{align*}
			T := \frac{\Z/2[x_1,y_1,x_2,y_2,x_3,y_3]}{(x_3^\infty,y_3^\infty)}\{\theta_3\},
		\end{align*}

with the nonzero monomial elements
\begin{equation}\label{theta3}
\begin{aligned}
\frac{\theta_3}{x^{n_3}_3y^{m_3}_3}x^{n_1}_1y^{m_1}_1x^{n_2}_2y^{m_2}_2,\\
 n_2+m_2\geq n_3+m_3+2.\\
\end{aligned}
\end{equation}
 In particular, notice that $\theta_3\notin R$.

The fourth piece is the part of the Mixed Cone of Type II  corresponding to $\io_2,$ ($p,q<0,b\geq 0$) subject to the condition $b+q \geq 0$. According to the description of Proposition \ref{propcase1v2} we have that

\begin{align*}
		\bigoplus_{q\leq-1,p\leq-1,b\geq1,b+q\geq0}\pi_{*}((\Si^{p\si+b\ep+q\so}H\ul{\Z/2})^{\mathcal{K}})&\cong \Z/2\left\{\ka_3^{m_3+1}\cdot \fc{\ta_1}{x_1^{n_1}y_1^{m_3+m_1}}\cdot x_2^{n_2}y_2^{m_2-m_3}\right\}_{m_3\leq m_2,m_1\geq 1}\\
		&\op\Z/2\left\{\fc{\ka_1}{x_1^{n_1}}\fc{\ta_3}{x_3^{n_3}y_3^{m_3}}x_2^{n_2}y_2^{m_2}\right\}_{n_3+m_3\leq n_2+m_2}.
	\end{align*}
From Proposition \ref{propcase1v2}, we also have the following proposition. This also follows from Corollary \ref{excl} in a different way.
\begin{proposition} \label{po3}
	The products
	\begin{align*}
		\fc{\ta_1}{x_1^{n_1}y_1^{m_1}}\fc{\ta_3}{x_3^{n_3}y_3^{m_3}}x_2^{n_2}y_2^{m_2} \in H_{Br,K}^{a+p\sigma+b\epsilon+q\sigma\otimes\epsilon}(pt)
	\end{align*}
	are zero when
	\begin{align*}
		p,q\leq -1\text{ and } b+q \geq 0 \tx{ i.e. }  n_2+m_2\geq n_3+m_3+2.
	\end{align*}
	By symmetry, we have analogous results for the other Mixed Cones of Type II.
\end{proposition}
\begin{proof}
	Suppose for contradiction that
	\begin{align*}
		\fc{\ta_1}{x_1^{n_1}y_1^{m_1}}\fc{\ta_3}{x_3^{n_3}y_3^{m_3}}x_2^{n_2}y_2^{m_2}  = \sum_{i=1}^{M} \fc{\ka_1}{x_1^{n_1+m_1+1}}\fc{\ta_3}{x_3^{j_i}y_3^{j_i'}}x_2^{j_i''}y_2^{j_i'''} + \sum_{i=1}^N \ka_3^{n_3+m_3+1}\fc{\ta_1}{x_1^{k_i}y_1^{k_i'+n_3+m_3}}x_2^{k_i''}y_2^{k_i'''},
	\end{align*}
	where $M+N > 0$ and $k_i' \geq 1.$ Since $k_i+k_i' = n_1+m_1+1,$ we must have that $k_i \leq n_1+m_1.$ If $M > 0,$ then 
	\begin{align*}
		&	0 = x_1^{n_1+m_1+1} \cp\fc{\ta_1}{x_1^{n_1}y_1^{m_1}}\fc{\ta_3}{x_3^{n_3}y_3^{m_3}}x_2^{n_2}y_2^{m_2}   = \\
		&x_1^{n_1+m_1+1} \cp \left(\sum_{i=1}^{M} \fc{\ka_1}{x_1^{n_1+m_1+1}}\fc{\ta_3}{x_3^{j_i}y_3^{j_i'}}x_2^{j_i''}y_2^{j_i'''} + \sum_{i=1}^N \ka_3^{n_3+m_3+1}\fc{\ta_1}{x_1^{k_i}y_1^{k_i'+n_3+m_3}}x_2^{k_i''}y_2^{k_i'''}\right)\\
		&= \sum_{i=1}^{M} \ka_1\fc{\ta_3}{x_3^{j_i}y_3^{j_i'}}x_2^{j_i''}y_2^{j_i'''} \neq 0,
	\end{align*}
	a contradiction. Then we must have that $M= 0$ and $N >0.$ Without loss of generality, assume that $k_1 > k_i$ for all $i>1$. Notice that there is one $k_i$ strictly greater than all because if $k_i=k_j$ then the cohomology classes $\ka_3^{n_3+m_3+1}\fc{\ta_1}{x_1^{k_i}y_1^{k_i'+n_3+m_3}}x_2^{k_i''}y_2^{k_i'''}$ are identical for $i$ and $j$. This is because $k_i=k_j$ implies $k_i'=k_j'$ which, by looking to the cohomological degree, means $k_i'''=k_j'''$. Because both cohomological classes have to live in the same degree it implies now that also $k_i''=k_j''$.
	
	But then we have
	\begin{align*}
		&0 =  x_1^{k_1}y_1^{k_1'} \cp\fc{\ta_1}{x_1^{n_1}y_1^{m_1}}\fc{\ta_3}{x_3^{n_3}y_3^{m_3}}x_2^{n_2}y_2^{m_2}  = x_1^{k_1}y_1^{k_1'}\cp \sum_{i=1}^N \ka_3^{n_3+m_3+1}\fc{\ta_1}{x_1^{k_i}y_1^{k_i'+n_3+m_3}}x_2^{k_i''}y_2^{k_i'''}\\
		& =  \ka_3^{n_3+m_3+1}\fc{\ta_1}{y_1^{n_3+m_3}}x_2^{k_i''}y_2^{k_i'''}\neq 0,
	\end{align*}
	a contradiction. 
\end{proof}
We have from \cite{DV1}, Corollary 5.9 that $$\oplus_{b<0}H^{a+p\sigma,b+q\sigma}_{C_2}(\EG C_2,\Z/2)=\ka_2 H^{*+*\sigma}_{Br,C_2}(pt)[x_3,y_3]$$
and therefore $$\oplus_{b<0}H^{a+p\sigma,b+q\sigma}_{C_2}(\R,\Z/2)=\ka_2 H^{*+*\sigma}_{Br,C_2}(pt)[x_3,y_3]\hookrightarrow H^*_{Br,K}(pt,\Z/2),$$
according to \cite{DV1}.

Here $H^{*+*\sigma}_{Br,C_2}(pt)=\Z/2[x_1,y_1]\oplus \Z/2\{\frac{\theta_1}{x^{n_1}_1y^{m_1}_1}\}$.
It implies that $\Z/2[x_i,y_i,\ka_i]$ modulo the relations \ref{eq2} is a subring in  $H^{\st,\st}_{C_2}(\R,\Z/2)$.

Moreover
\begin{theorem}\label{decom} $$(R,\ka_2)\simeq \frac{\Z/2[x_i,y_i,\ka_i]}{(\ka_i\ka_j=y_k^2,\ka_ix_i=x_jy_k+x_ky_j,\ka_iy_i=y_ky_j)}+ nilpotents.$$ The component $nilpotents$ has only nilpotent elements and multiplications between these elements give zero.
\end{theorem}
The first part of the Theorem \ref{decom} follows from the discussion above. The second part of  Theorem \ref{decom} follows from the Propositions \ref{pz} and \ref{po}, Proposition \ref{po1} and Proposition \ref{po3}.
\begin{proposition} \label{pz}
The products
	\begin{align*}
		\fc{\ta_1}{x_1^{n_1}y_1^{m_1}}\left(\frac{\ka_1}{x_1^N}\fc{\ta_3}{x_3^{n_3}y_3^{m_3}}x_2^{n_2}y_2^{m_2}\right) \in H_{Br,K}^{a+p\sigma+b\epsilon+q\sigma\otimes\epsilon}(pt)
	\end{align*}
	are zero when $n_3+m_3\leq n_2+m_2$.
\end{proposition}
\begin{proof} Consider \begin{align*}
		\alpha=\fc{\ta_1}{x_1^{n_1}y_1^{m_1}}\left(\frac{\ka_1}{x_1^{N+n_1+1}}\fc{\ta_3}{x_3^{n_3}y_3^{m_3}}x_2^{n_2}y_2^{m_2}\right) \in H_{Br,K}^{a+p\sigma+b\epsilon+q\sigma\otimes\epsilon}(pt)
	\end{align*}
Then $$x_1^{n_1+1}\alpha=0=\fc{\ta_1}{x_1^{n_1}y_1^{m_1}}\left(\frac{\ka_1}{x_1^{N}}\fc{\ta_3}{x_3^{n_3}y_3^{m_3}}x_2^{n_2}y_2^{m_2}\right)$$ by using the associativity and the commutativity of the multiplication.
	\end{proof}
\begin{proposition}\label{po} The products $$\left(\frac{\ka_1}{x_1^M}\fc{\ta_3}{x_3^{n_3}y_3^{m_3}}x_2^{n_2}y_2^{m_2}\right)\left(\frac{\ka_1}{x_1^N}\fc{\ta_3}{x_3^{n_3'}y_3^{m_3'}}x_2^{n_2'}y_2^{m_2'}\right)$$
are zero when $n_3+m_3\leq n_2+m_2$ and  $n_3'+m_3'\leq n_2'+m_2'$. In particular, the fourth part of $R$ contains only nilpotents with zero multiplication. 
\end{proposition}	

\begin{proof} This is obvious for $N=M=0$. Let $$\alpha=\left(\frac{\ka_1}{x_1^{M+N}}\fc{\ta_3}{x_3^{n_3}y_3^{m_3}}x_2^{n_2}y_2^{m_2}\right)\left(\frac{\ka_1}{x_1^{N}}\fc{\ta_3}{x_3^{n_3'}y_3^{m_3'}}x_2^{n_2'}y_2^{m_2'}\right).$$ Let $n_3>n_3'$.
Then $$x_1^{N}x_3^{n_3'+1}\alpha=0=\left(\frac{\ka_1}{x_1^{M}}\fc{\ta_3}{x_3^{n_3-n_3'-1}y_3^{m_3}}x_2^{n_2}y_2^{m_2}+x_1^N\sum a_i\ka_3^{-q}\fc{\ta_1}{x_1^{n_1'}y_1^{m_1'}}x_2^{n_2''}y_2^{m_2''}\right)\left(\frac{\ka_1}{x_1^{N}}\fc{\ta_3}{x_3^{n_3'}y_3^{m_3'}}x_2^{n_2'}y_2^{m_2'}\right)=$$
$$=\left(\frac{\ka_1}{x_1^{M}}\fc{\ta_3}{x_3^{n_3-n_3'-1}y_3^{m_3}}x_2^{n_2}y_2^{m_2}\right)\left(\frac{\ka_1}{x_1^{N}}\fc{\ta_3}{x_3^{n_3'}y_3^{m_3'}}x_2^{n_2'}y_2^{m_2'}\right)$$ by using the associativity and the commutativity of the multiplication and Proposition \ref{pz}.  The first bracket in the above product follows from the fact that 
$$x_1^{N+M}\left( x_3^{n_3'+1}\frac{\ka_1}{x_1^{N+M}} \frac{\theta_3}{x_3^{n_3}y_3^{m_3}} x_2^{n_2}y_2^{m_2}+\frac{\ka_1}{x_1^{M+N}}\frac{\theta_3}{x_3^{n_3-n_3'-1}y_3^{m_3}} x_2^{n_2}y_2^{m_2}\right)=2\ka_1 \frac{\theta_3}{x_3^{n_3-n_3'-1}y_3^{m_3}} x_2^{n_2}y_2^{m_2}=0$$
and the kernel of the multiplication with $x_1^{N+M}$  (in the Mixed Cone of Type II) is generated by the elements of the form $ a_i\ka_3^{-q}\fc{\ta_1}{x_1^{n_1'}y_1^{m_1'}}x_2^{n_2''}y_2^{m_2''}$ with $a_i\in\{0,1\}$ and
\begin{align*}
	\deg\left(\frac{\ka_1}{x_1^{M+N}}\fc{\ta_3}{x_3^{n_3-n_3'-1}y_3^{m_3}}x_2^{n_2}y_2^{m_2}\right)=(a,p,b,q) = \deg\left(\ka_3^{-q}\fc{\ta_1}{x_1^{n_1'}y_1^{m_1'}}x_2^{n_2''}y_2^{m_2''}\right)
\end{align*}
because of Proposition \ref{propcase1v2} (the other type of generators don't vanish at the multiplication by $x_1^{N+M}$). Because the powers in $\alpha$ are arbitrary we obtain the result.
\end{proof}
\begin{proposition} \label{po1} The products 
	$$\frac{\theta_3}{x^{n_3}_3y^{m_3}_3}x^{n_1}_1y^{m_1}_1x^{n_2}_2y^{m_2}_2\frac{\ka_1}{x^{p_1}_1}\frac{\theta_3}{x_3^{p_3}y_3^{q_3}}x_2^{p_2}y_2^{q_2}=0$$
	where $n_2+m_2\geq n_3+m_3+2$ and $p_2+q_2\geq p_3+q_3$.
	
\end{proposition} 
\begin{proof} We will prove first the result for $m_2 \geq 1.$ It is enough to prove that 
	$$\frac{\theta_3}{x^{n_3}_3y^{m_3}_3}x^{n_2}_2y^{m_2}_2\frac{\ka_1}{x^{p_1}_1}\frac{\theta_3}{x_3^{p_3}y_3^{q_3}}x_2^{p_2}y_2^{q_2}=0$$ with $n_2+m_2\geq n_3+m_3+2$ and $p_2+q_2\geq p_3+q_3$. Because $\ka_1y_1=y_2y_3$ we have that 
	$$\frac{\theta_3}{x^{n_3}_3y^{m_3}_3}x^{n_2}_2y^{m_2}_2=x_1y_1\frac{\ka_1}{x_1}\frac{\theta_3}{x_3^{n_3}y_3^{m_3+1}}x_2^{n_2}y_2^{m_2-1}$$ with $n_2+m_2-1\geq n_3+m_3+1$. It implies the above product is given by 
	$$x_1y_1\left(\frac{\ka_1}{x_1}\frac{\theta_3}{x_3^{n_3}y_3^{m_3+1}}x_2^{n_2}y_2^{m_2-1}\right)\left(\frac{\ka_1}{x^{p_1}_1}\frac{\theta_3}{x_3^{p_3}y_3^{q_3}}x_2^{p_2}y_2^{q_2}\right)$$ 
	which is zero according to Proposition \ref{po}. 
 
 When $m_2 = 0,$ the proof is more complicated. By Corollary \ref{ist} and Proposition \ref{y3case} from the appendix, we have that
 \begin{align*}
 	y_3^M \frac{\ka_1}{x^{p_1}_1}\frac{\theta_3}{x_3^{p_3}y_3^{q_3}}x_2^{p_2}y_2^{q_2} = \sum a_i \fc{\ta_1}{x_1^{n_1'}y_1^{m_1'}}x_2^{n_2'}y_2^{m_2'}x_3^{n_3'}y_3^{m_3'}
 \end{align*}
 where $a_i \in \{0,1\}$ for sufficiently large $M.$
 
Then we have that
	\begin{align*}
		\alpha := &\frac{\theta_3}{x^{n_3}_3y^{m_3}_3}x^{n_2}_2y^{m_2}_2 \frac{\ka_1}{x^{p_1}_1}\frac{\theta_3}{x_3^{p_3}y_3^{q_3}}x_2^{p_2}y_2^{q_2} = \frac{\theta_3}{x^{n_3}_3y^{m_3+M}_3}x^{n_2}_2y^{m_2}_2 \left(y_3^M \frac{\ka_1}{x^{p_1}_1}\frac{\theta_3}{x_3^{p_3}y_3^{q_3}}x_2^{p_2}y_2^{q_2}\right) \\
		&= \frac{\theta_3}{x^{n_3}_3y^{m_3+M}_3}x^{n_2}_2y^{m_2}_2 \left(\sum a_i \fc{\ta_1}{x_1^{n_1'}y_1^{m_1'}}x_2^{n_2'}y_2^{m_2'}x_3^{n_3'}y_3^{m_3'}\right)  \\&=\sum a_i \fc{\ta_1}{x_1^{n_1'}y_1^{m_1'}}\fc{\ta_3}{x_3^{n_3''}y_3^{m_3''}}x_2^{n_2''}y_2^{m_2''} =: \beta,
	\end{align*}
	
	where $a_i \in \{0,1\}$ and  $n_2''+m_2''\geq n_3''+m_3''+2,$ since $\deg(\al)=\deg(\be),$ and by assumption we have that $n_2+m_2\geq n_3+m_3+2$ and $p_2+q_2\geq p_3+q_3.$
	But by Proposition \ref{po3}, each summand is zero.	
\end{proof}
Let $$NC:=\oplus _{b\geq 0, b+q<0}H^{\star,b+q\sigma} _{C_2}(\R,\Z/2),$$ which is a $\MMt$-submodule of the Bredon motivic cohomology of $\R$. From \cite{DV1} we have
$$NC\subset\oplus _{2\leq a\leq 2b+1}\rH^{a-b+b\ep} _{Br,\Sigma_2}(\Sigma B_{\Sigma_2}C_2,\Z/2)[x_1^{\pm 1},x^{\pm 1}_3]\simeq \oplus _{1-b\leq a\leq b}\rH^{a+b\ep} _{Br,\Sigma_2}(B_{\Sigma_2}C_2,\Z/2)[x_1^{\pm 1},x^{\pm 1}_3].$$
We recall that $H^{a+p\sigma,b+q\sigma} _{C_2}(\R,\Z/2)=0$ if $a>2b+1$, $b\geq 0$ and $b+q<0$. Also $H^{a+p\sigma,b+q\sigma} _{C_2}(\R,\Z/2)=0$ if $a\leq 1$, $b\geq 0$ and $b+q<0$.

This implies, using the Bredon cohomology of $B_{\Sigma_2}C_2$, that $$NC\subset \{x^n_2y^m_2\Sigma (b^pc),x^n_2y^m_2\Sigma (b^p)\}[x^{\pm 1}_1,x^{-1}_3],$$
the subset with the degree of $x^{-1}_3$ given by $q\leq -b\leq 0,$ where $b$ is the degree of $\ep$ i.e.
 $$NC=\oplus _{a\leq b+1} x^{-b-1}_3\rH^{a+b\ep}_{Br,\Sigma_2}(\Sigma B_{\Sigma_2}C_2,\Z/2)[x^{\pm 1}_1,x^{-1}_3]=$$
 $$=\left\{x^{-n-m-p-2}_3x^n_2y^m_2\Sigma (b^pc),x^{-n-m-p-1}_3x^n_2y^m_2\Sigma b^p\right\}[x^{\pm 1}_1,x^{-1}_3].$$
 
 In [\cite{DV1}, Theorem 3.5, Theorem 6.2], it was shown that $NC$ is a  subset in $H_{Br,K}^\st(pt,\Z/2)$ and that we have an isomorphism of $\MMt$-algebras 
 $$H^{\st,\st}_{C_2}(\R,\Z/2)\simeq (R,\ka_2)\oplus NC.$$  In particular, $NC$ is isomorphic with its image under the realisation map and all the elements of $NC$ are nilpotents and infinitely divisible by $x_1$ and $x_3$ giving a subset of infinitely divisible elements by $x_1$ and $x_3$ in the $RO(C_2\times \Sigma_2)$ graded cohomology of a point. 
 
 Analyzing the $C_2\times\Sigma_2$ topology isotropy sequence,  we can deduce the following simpler description of $NC$ in terms of the nilpotents from the $RO(C_2\times \Sigma_2)$ graded cohomology of a point.
  \begin{theorem}
 	\label{ncg}
 	The image of $NC$ under the realisation map is generated by the elements
 	\begin{align*}
 		\fc{\ka_3^{n+1}}{x_3^{n_3+1}}\fc{\ta_1}{x_1^{n_1}y_1^{2n+1}}	\txa  \fc{\ka_3^{n+1}}{x_3^{n_3+1}}\fc{\ta_1}{x_1^{n_1+1}y_1^{2n}}
 	\end{align*}
 	under multiplication with $x_1,x_2,$ and $y_2,$ where $n,n_1,n_3 \geq 0$ are subject to the obvious degree restrictions imposed by $NC$.
 \end{theorem}
 \begin{proof}
Since $x_3$ is an isomorphism on $ \hr_{Br,K}^\st(\et),$ it suffices to prove that
 	\begin{align*}
 		\fc{\susp b^n}{x_1^{n+1}x_3^{n+1}}\stackrel{p^*}{\mt} \fc{\ka_3^n}{x_3}\fc{\ta_1}{y_1^{2n-1}}	\txa \fc{\susp cb^{n-1}}{x_1^{n+1}x_3^{n+1}} \stackrel{p^*}{\mt}  \fc{\ka_3^n}{x_3}\fc{\ta_1}{x_1y_1^{2n-2}}
 	\end{align*}
 	under the map
 	\begin{align*}
 		p^* :  \hr_{Br,K}^\st(\et) \to H_{Br,K}^\st(pt)
 	\end{align*}
 	for all $n\geq 1.$  According to \cite{DV1}, the realization map and the isotropy map $p^*$ coincide in the motivic range $a\leq 2b+1$ and $b+q<0$.

 	By studying the isotropy sequence, from Lemma \ref{xid}, we have
 	\begin{align*}
 		\fc{\susp b^n}{x_3^{n+1}}\stackrel{p^*}{\mt} \fc{\ta_3}{y_3^{n-1}}x_2^n	\txa  \fc{\susp cb^{n-1}}{x_3^{n+1}} \stackrel{p^*}{\mt} \fc{\ta_3}{y_3^{n-1}}y_2x_2^{n-1}
 	\end{align*}
 	under $p^*.$
 	
 	Then by Lemma \ref{x1case2v2} it suffices to show that 
 	\begin{align*}
 		x_1^{n+1}\cp\fc{\ka_3^n}{x_3}\fc{\ta_1}{y_1^{2n-1}} = \fc{\ta_3}{y_3^{n-1}}x_2^n	\txa 	x_1^{n+1}\cp\fc{\ka_3^n}{x_3}\fc{\ta_1}{x_1y_1^{2n-2}} = \fc{\ta_3}{y_3^{n-1}}y_2x_2^{n-1}.
 	\end{align*}
 	 We have that
 	\begin{align*}
 		x_1 \cp \fc{\ka_3^n}{x_3}\fc{\ta_1}{y_1^{2n-1}}  = \ka_1^n \fc{\ta_3}{y_3^{2n-1}},
 	\end{align*}
 	so
 	\begin{align*}
 		x_1^{n+1}\cp\fc{\ka_3^n}{x_3}\fc{\ta_1}{y_1^{2n-1}} = \fc{\ta_3}{y_3^{n-1}}x_2^n
 	\end{align*}
 	as desired. The first equality follows because the multiplication by $$x_1: H^{n+1-(n+1)\sigma+n\epsilon-(n+1)\sigma\otimes\epsilon}_{Br,K}(pt)=\Z/2\simeq H^{n+1-n\sigma+n\epsilon-(n+1)\sigma\otimes\epsilon}_{Br,K}(pt)=\Z/2$$ is an isomorphism from Lemma \ref{x1case2v2}.
	
	 Next, we consider the second case. If $n=1,$ the result follows from the isotropy sequence (see Lemma \ref{xid}) and that $$\frac{\ka_1}{x_1}\frac{\theta_3}{x_3}=\frac{\ka_3}{x_3}\frac{\theta_1}{x_1}.$$
	 
	 Let now $n\geq2.$ We have that
 	\begin{align*}
 		x_1\cp\fc{\ka_3^n}{x_3}\fc{\ta_1}{x_1y_1^{2n-2}} = \fc{\ka_3}{x_3^n}\fc{\ta_1}{y_1^{n-1}}x_2^{n-1}+\ka_1^n \fc{\ta_3}{x_3y_3^{2n-2}}.
 	\end{align*}
	To prove this equality notice that $x_3x_1\frac{k_3^n}{x_3}\frac{\theta_1}{x_1y_1^{2n-2}}=\ka_3^n\frac{\theta_1}{y_1^{2n-2}}$ and multiplication by $x_3$ is an isomorphism in these indexes by Lemma \ref{x1case2v2}. On the other side $$x_3(\fc{\ka_3}{x_3^n}\fc{\ta_1}{y_1^{n-1}}x_2^{n-1}+\ka_1^n \fc{\ta_3}{x_3y_3^{2n-2}})=\fc{\ka_3}{x_3^{n-1}}\fc{\ta_1}{y_1^{n-1}}x_2^{n-1}+\ka_1^n \fc{\ta_3}{y_3^{2n-2}}=\ka_3^n\frac{\theta_1}{y_1^{2n-2}}.$$
	The last equality follows from the fact that $\alpha=\fc{\ka_3}{x_3^{n-1}}\fc{\ta_1}{y_1^{n-1}}x_2^{n-1}+\ka_3^n\frac{\theta_1}{y_1^{2n-2}}$ is in the kernel of multiplication by $x_3^{n-1}$. By induction on $n\geq 2$, it is in the kernel of multiplication by $x_3$ and is a nontrivial class as $y_1^{n-1}y_3^{n-1}\alpha\neq 0$. The equality follows now.
	
 	Then
 	\begin{align*}
 		&x_1 \cp \left(\fc{\ka_3}{x_3^n}\fc{\ta_1}{y_1^{n-1}}x_2^{n-1}+\ka_1^n \fc{\ta_3}{x_3y_3^{2n-2}}\right) =x_1\cp \fc{\ka_3}{x_3^n}\fc{\ta_1}{y_1^{n-1}}x_2^{n-1}+\ka_1^{n-1} \fc{\ta_3}{y_3^{2n-2}}y_2 +\ka_1^{n-1} \fc{\ta_3}{x_3y_3^{2n-3}}x_2\\
 		&=\ka_1^{n-1} \fc{\ta_3}{x_3y_3^{2n-3}}x_2+\ka_1^{n-1} \fc{\ta_3}{y_3^{2n-2}}y_2 +\ka_1^{n-1} \fc{\ta_3}{x_3y_3^{2n-3}}x_2 \tx{(Proposition \ref{x3case2}, applied to $x_1$)}\\
 		&= \ka_1^{n-1} \fc{\ta_3}{y_3^{2n-2}}y_2,\\
 		&\implies x_1^{n+1} \cp \fc{\ka_3^n}{x_3}\fc{\ta_1}{x_1y_1^{2n-2}} = \fc{\ta_3}{y_3^{n-1}}y_2x_2^{n-1}.
 	\end{align*}
 \end{proof}
 \begin{remark}
 We remark that $NC$ is generated (by multiplication) by cohomology classes given in Proposition \ref{propcase2v2}. All cohomology classes of Proposition \ref{propcase2v2}  are  infinitely uniquely divisible by both $x_1$ and $x_3$.
 \end{remark}
  We can conclude that
 \begin{theorem} \label{d1}
 We have a $\MMt$-algebra isomorphism
 $$H^{\star,\star}_{C_2}(\R,\Z/2)\simeq \frac{\Z/2[x_i,y_i,\ka_i]}{(\ka_i\ka_j=y_k^2,\ka_ix_i=x_jy_k+x_ky_j,\ka_iy_i=y_ky_j)}+ nilpotents_1.$$ The component $nilpotent_1$ is formed  by sums of nilpotent generators given in Theorem \ref{decom} and from NC.
 
 The component $nilpotents_1$ has zero products and strictly contains $NC$. All motivic classes are sums of multiplications of the above generators modulo relations \ref{eq2} displayed in the quotient above together with $\theta_1^2=\theta_1x_1=\theta_1y_1=0$ and with $\theta_3-$classes from \ref{theta3} being $2$-nilpotent. 
\end{theorem}
\begin{proof} The only thing left to prove is that the nilpotents from $NC$ multiplied by nilpotents from Theorem \ref{decom} give zero. The nilpotents from Theorem \ref{decom} have four possible forms:
\begin{enumerate}
\item $\ka_3^{m_3+1}\frac{\theta_1}{x_1^{n_1}y_1^{m_3+m_1}}x_2^{n_2}y_2^{m_2-m_3}$ with $m_3\leq m_2$ and $m_1\geq 1$.
\item $\frac{\ka_1}{x_1^{n_1}}\frac{\theta_3}{x_3^{n_3}y_3^{n_3}}x_2^{n_2}y_2^{m_2}$  with $n_3+m_3\leq n_2+m_2$.
\item $\frac{\theta_3}{x_3^{n_3}y_3^{m_3}}x_1^{n_1}y_1^{m_1}x_2^{n_2}y_2^{m_2}$ with $n_2+m_2\geq n_3+m_3+2$.
\item $\frac{\theta_1}{x_1^{n_1}y_1^{m_1}}x_2^{n_2}y_2^{m_2}x_3^{n_3}y_3^{m_3}$.
\end{enumerate}
The nilpotents $\alpha$ that are contained in $NC$ are all divisible by $x_1$ and $x_3$.   This is enough to give zero for all the needed multiplications. For example 
$$\alpha \frac{\theta_1}{x_1^{n_1}y_1^{m_1}}= (\frac{\alpha}{x_1^N})x_1^{N}\frac{\theta_1}{x_1^{n_1}y_1^{2n+1}}=0$$ 
for $N>>0,$ therefore the products of $\alpha$ with form 1. and form 4. above are zero. Similarly, using division by $x_3$ we find the products of $\alpha$ with form 3. is zero. In the end we find 
$$\alpha\frac{\ka_1}{x_1^{n_1}}\frac{\theta_3}{x_3^{n_3}y_3^{n_3}}x_2^{n_2}y_2^{m_2}=(\frac{\alpha}{x_1^{n_1}x_3^{n_3+1}})x_1^{n_1}x_3^{n_3+1}\frac{\ka_1}{x_1^{n_1}}\frac{\theta_3}{x_3^{n_3}y_3^{n_3}}x_2^{n_2}y_2^{m_2}=0$$
so multiplication of $\alpha$ with form 2. is zero.
\end{proof}
From \cite{DV1} we have that $$H^{\star,\star}_{C_2}(\R,\Z/2)\hookrightarrow H^\star_{Br,K}(pt,\Z/2).$$ We showed above that one difference between these two rings is that $H^\star_{Br,K}(pt,\Z/2)$ has a bigger set of nilpotents. Another difference is that the nilpotents in $H^\star_{Br,K}(pt,\Z/2)$ do not necessarily have zero products as we can easily see from $\theta_1\theta_3\neq 0$ (\cite{BH}).

We can write the $RO(C_2\times \Sigma_2)$-graded cohomology ring of a point as follows:
 \begin{theorem} \label{Bmt}
 We have a $\MMt$-algebra isomorphism
 $$H^{\star}_{Br,K}(pt,\Z/2)\simeq \frac{\Z/2[x_i,y_i,\ka_i]}{(\ka_i\ka_j=y_k^2,\ka_ix_i=x_jy_k+x_ky_j,\ka_iy_i=y_ky_j)}+ nilpotents_2.$$ The component $nilpotents_2$ (completely described in Section 3 and summarized in  Theorem \ref{tc}) contains strictly the component $nilpotents_1$ of Theorem \ref{d1} and its products are not necessary zero. The relations that are satisfied in this ring are the relations \ref{rl}.
\end{theorem}
In Theorem \ref{d1}  the products in each of the direct sum components are already given above in the text (the products between non-nilpotents follow from \cite{BH}; the products from nilpotents are described in this section). The products between nilpotents and the subring component are also nilpotent, but their description in terms of the chosen nilpotent generators is not in general trivial and can be very complicated. We  give in the appendix few of these products that we used in the main text. 
\section{Appendix}

\subsection{Motivic products of {$y_1$}, {$y_3$} and $x_3$ with motivic generators from the Mixed Cone of Type II}
In this Appendix we will give examples of products (used in the main text) in the Mixed Cone of Type II between motivic cohomology classes that are expressed in terms of motivic generators. More types of results like these can be found in the appendix of \cite{BD}.

We start with a lemma about the $C_2\times \Sigma_2$ isotropy long exact sequence.

\begin{lemma}
	\label{xid}
	We have that
	\begin{align*}
	       \fc{\susp c}{x_1x_3}\stackrel{p^*}{\mt} \io_2,
         \end{align*}
         \begin{align*}
		 \fc{\susp c}{x_1^{n_1+1}x_3^{n_3+1}}x_2^{n_2}y_2^{m_2} \stackrel{p^*}{\mt} \fc{\ka_1}{x_1^{n_1}}\fc{\ta_3}{x_3^{n_3}}x_2^{n_2}y_2^{m_2}
	\end{align*}
	
	\begin{align*}
	       \fc{\susp b^{n}}{x_3^{n+1}}\stackrel{p^*}{\mt} \frac{\theta_3}{y_3^{n-1}}x_2^n,
         \end{align*}
         \begin{align*}
	       \fc{\susp cb^{n-1}}{x_3^{n+1}}\stackrel{p^*}{\mt} \frac{\theta_3}{y_3^{n-1}}y_2x_2^{n-1},
         \end{align*}
         in the $C_2\times\Sigma_2$ isotropy long exact sequence, where $p^*: \tilde{H}^*_{Br,K}(\tilde{E}_{\Sigma_2}C_2)\rightarrow H^*_{Br,K}(pt)$.
\end{lemma}
\begin{proof} The first statement follows from the fact that  $\fc{\susp c}{x_1x_3}\in \tilde{H}^{1-\sigma+\epsilon-\sigma\otimes\epsilon}_{Br,K}(\tilde{E}_{\Sigma_2}C_2)$ is the generator
and the isotropy sequence map $f:H^{1-\sigma+\epsilon-\sigma\otimes\epsilon}_{Br,K}(pt)\rightarrow H^{1-\sigma+\epsilon-\sigma\otimes\epsilon}_{Br,K}(E_{\Sigma_2}C_2)$ sends $\iota_2$ to zero. Indeed, suppose $f(\iota_2)\neq 0$. Because $$H^{1-\sigma+\epsilon-\sigma\otimes\epsilon}_{Br,K}(E_{\Sigma_2}C_2)\simeq \Z/2$$ we have $f(\iota_2)=\ka_2^{-1}$ (see Theorem \ref{cEC}). This means that $f(\ka_2\iota_2)=1$, but because $\ka_2\iota_2=\ka_2\ka_3\theta_1=y_1^2\theta_1=0$ we obtain a contradiction.

	For the second statement, we proceed by induction on $n_1.$ When $n_1=0,$ we have from the above and Proposition \ref{caseMC2} that $\fc{\susp c}{x_1x_3^{n_3+1}}\mt \ka_1\fc{\ta_3}{x_3^{n_3}}$ and this implies $\fc{\susp c}{x_1x_3^{n_3+1}}x_2^{n_2}y_2^{m_2}\mt \ka_1\fc{\ta_3}{x_3^{n_3}}x_2^{n_2}y_2^{m_2}.$
	
	Now assume that the result holds for all $n_1 \leq  N_1.$ Then, by induction,
	\begin{align*}
		p^*\left(\fc{\susp c}{x_1^{N_1+2}x_3^{n_3+1}}x_2^{n_2}y_2^{m_2}\right)  +  \fc{\ka_1}{x_1^{N_1+1}}\fc{\ta_3}{x_3^{n_3}}x_2^{n_2}y_2^{m_2}\in \ker\varphi_{-N_1-2}.
	\end{align*}
	with $\varphi_{-N_1-2}$ the multiplication with $x_1$.
	
	If we suppose that  $p^*\left(\fc{\susp c}{x_1^{N_1+2}x_3^{n_3+1}}x_2^{n_2}y_2^{m_2}\right)  +  \fc{\ka_1}{x_1^{N_1+1}}\fc{\ta_3}{x_3^{n_3}}x_2^{n_2}y_2^{m_2} \neq 0$ then
     $$p^*\left(\fc{\susp c}{x_1^{N_1+2}x_3^{n_3+1}}x_2^{n_2}y_2^{m_2}\right)  +  \fc{\ka_1}{x_1^{N_1+1}}\fc{\ta_3}{x_3^{n_3}}x_2^{n_2}y_2^{m_2}  = \ka_3^{n_3+1}\fc{\ta_1}{y_1^{N_1+1+n_3}}x_2^{n_2-N_1-1-n_3}y_2^{m_2+N_1+1}.$$ 
     By definition, we have that $\fc{\ka_1}{x_1^{N_1+1}}\fc{\ta_3}{x_3^{n_3}}x_2^{n_2}y_2^{m_2} \cp y_1^{N_1+1}y_3^{n_3} =0$, so
		$p^*\left(\fc{\susp c}{x_1^{N_1+2}x_3^{n_3+1}}x_2^{n_2}y_2^{m_2}\right) \cp   y_1^{N_1+1}y_3^{n_3}= \ka_3^{n_3+1}\fc{\ta_1}{y_1^{N_1+1+n_3}}x_2^{n_2-N_1-1-n_3}y_2^{m_2+N_1+1}\cp   y_1^{N_1+1}y_3^{n_3}= \ka_3\ta_1 x_2^{n_2-N_1-1-n_3}y_2^{m_2+N_1+1+n_3}.$
      But we have
	$$\fc{\susp c}{x_1^{N_1+2}x_3^{n_3+1}}x_2^{n_2}y_2^{m_2} \cp y_1 = 0,$$
	because $y_1\Sigma c=0$ therefore a contradiction.
	
For the third statement we have that $H^{n+1+n\epsilon-(n+1)\sigma\otimes\epsilon}_{Br,K}(pt)=\Z/2\frac{\theta_3}{y_3^{n-1}}x_2^n$ and because $p^*$ is an isomorphism for these indexes we have $p^*(\frac{\susp b^n}{x_3^{n+1}})=\frac{\theta_3}{y_3^{n-1}}x_2^{n}$. 

For the fourth statement notice that $H^{n+n\epsilon-(n+1)\sigma\otimes\epsilon}_{Br,K}(pt)\simeq \Z/2\frac{\theta_3}{y_3^{n-1}}y_2x_2^{n-1}\oplus\Z/2\frac{\theta_3}{y_3^{n-2}x_3}x^{n}_2$ and the map $x_1:H^{n+n\epsilon+(-n-1)\sigma\otimes\epsilon}_{Br,K}(pt)\rightarrow H^{n+\sigma+n\epsilon-(n+1)\sigma\otimes\epsilon}_{Br,K}(pt)$ is an isomorphism from Lemma \ref{x1case2v2}. Also notice that
$$0=x_2^{n-1}k_3x_3\frac{\theta_3}{y_3^{n-1}}=x_1\frac{\theta_3}{y_3^{n-1}}x_2^{n-1}y_2+y_1x_2^n\frac{\theta_3}{y_3^{n-1}}$$
therefore  $x_1\frac{\theta_3}{y_3^{n-1}}x_2^{n-1}y_2=y_1x_2^n\frac{\theta_3}{y_3^{n-1}}$. But on the other side, from the above,
$$x_1p^*(\frac{\susp cb^{n-1}}{x_3^{n+1}})=p^*(x_1\frac{\susp cb^{n-1}}{x_3^{n+1}})=p^*(y_1\frac{\susp b^n}{x_3^{n+1}})=y_1\frac{\theta_3}{y_3^{n-1}}x_2^n.$$
The conclusion follows because the multiplication by $x_1$ is an isomorphism.
\end{proof}
Directly from Lemma \ref{xid} we have the following corollary:
\begin{corollary}\label{ist}
	\begin{align*}
		y_1 \cp \fc{\ka_1}{x_1^{n_1}}\fc{\ta_3}{x_3^{n_3}}x_2^{n_2}y_2^{m_2} =	y_3 \cp \fc{\ka_1}{x_1^{n_1}}\fc{\ta_3}{x_3^{n_3}}x_2^{n_2}y_2^{m_2} = 0.
	\end{align*}
\end{corollary}
The following proposition gives a formula for a non-trivial product between the motivic cohomology class $y_1$ and a motivic cohomology class from the Mixed Cone of Type II. The above corollary gives the case $m_3=0$.
\begin{proposition}\label{y1case}
	Suppose that $m_3,n_1 \geq 1,$ and $M:=n_2+m_2\geq n_3+m_3.$  Then we have that
	\begin{equation*}
		\resizebox{\textwidth}{!}{%
			$\begin{aligned}
				(1) \quad y_1 \cp	\fc{\ka_1}{x_1^{n_1}}\fc{\ta_3}{x_3^{n_3}y_3^{m_3}}x_2^{n_2}y_2^{m_2}	= &\sum_{j=n_2}^{m_3-1+n_2}	\fc{\ka_1}{x_1^{n_1-1}} \fc{\ta_3}{x_3^{n_3-n_2+j+1}y_3^{m_3+n_2-j-1}}x_2^{j}y_2^{M-j}\tx{if $m_2 \geq m_3-1\geq 0$ },\\
				(2) \quad y_1 \cp  \fc{\ka_1}{x_1^{n_1}}\fc{\ta_3}{x_3^{n_3}y_3^{m_3}}x_2^{n_2}y_2^{m_2}	= &\sum_{j=n_2-n_3-1}^{n_2-1}\fc{\ka_1}{x_1^{n_1-1}} \fc{\ta_3}{x_3^{n_3-n_2+j+1}y_3^{m_3+n_2-j-1}}x_2^{j}y_2^{M-j}\tx{if $m_2<m_3-1$}.
			\end{aligned}$
		}
	\end{equation*}
\end{proposition}
\begin{proof} We proceed by induction on $n_1.$ When $n_1=1,$ we have
	\begin{align*}
		x_1 \cp y_1 \cp	\fc{\ka_1}{x_1}\fc{\ta_3}{x_3^{n_3}y_3^{m_3}}x_2^{n_2}y_2^{m_2} = y_2y_3\fc{\ta_3}{x_3^{n_3}y_3^{m_3}}x_2^{n_2}y_2^{m_2} = \fc{\ta_3}{x_3^{n_3}y_3^{m_3-1}}x_2^{n_2}y_2^{m_2+1}.
	\end{align*}
	If $m_2 \geq m_3-1\geq 0$, then 
	\begin{equation*}
				x_1\sum_{j=n_2}^{m_3-1+n_2}{\ka_1} \fc{\ta_3}{x_3^{n_3-n_2+j+1}y_3^{m_3+n_2-j-1}}x_2^{j}y_2^{M-j}=\fc{\ta_3}{x_3^{n_3}y_3^{m_3-1}}x_2^{n_2}y_2^{m_2+1},
	\end{equation*}
	so 
	\begin{align*}
		\ub{y_1 \cp	\fc{\ka_1}{x_1}\fc{\ta_3}{x_3^{n_3}y_3^{m_3}}x_2^{n_2}y_2^{m_2}}_{=:\al} +\ub{\sum_{j=n_2}^{m_3-1+n_2}{\ka_1} \fc{\ta_3}{x_3^{n_3-n_2+j+1}y_3^{m_3+n_2-j-1}}x_2^{j}y_2^{M-j}}_{=:\be}\in \ker \varphi.
	\end{align*}
	Here $\varphi$ is multiplication by $x_1$. If $\ker\varphi = 0$ in this degree, we are done. If $\ker\varphi\neq 0$ in this degree, then the product
	\begin{align*}
		\gamma := \ka_3^{m_3+n_3+1}\fc{\ta_1}{y_1^{m_3+n_3}}x_2^{n_2-n_3-1}y_2^{m_2-m_3+1}
	\end{align*}
	can be identified with the generator of $\ker \varphi$ by degree reasons.
	Suppose for contradiction that $\al + \be = \gamma.$ Then we have that $y_3^{n_3+m_3} \cp \gamma \neq 0$ and $y_3^{n_3+m_3} \cp \be = 0,$ so $y_3^{n_3+m_3} \cp \al \neq 0.$ This implies that $y_3^{n_3+m_3} \cp y_1 \cp \fc{\ka_1}{x_1}\fc{\ta_3}{x_3^{n_3}y_3^{m_3}}x_2^{n_2}y_2^{m_2} \neq 0.$ But this contradicts the definition of $\fc{\ka_1}{x_1}\fc{\ta_3}{x_3^{n_3}y_3^{m_3}}x_2^{n_2}y_2^{m_2},$ which is the element defined such that the multiplication by $y_3^{n_3+m_3}y_1$ is zero.
		
	If $m_2<m_3-1$, then
	\begin{equation*}
		x_1 \cp \sum_{j=n_2-n_3-1}^{n_2-1}\ka_1\fc{\ta_3}{x_3^{n_3-n_2+j+1}y_3^{m_3+n_2-j-1}}x_2^{j}y_2^{M-j}=\fc{\ta_3}{x_3^{n_3}y_3^{m_3-1}}x_2^{n_2}y_2^{m_2+1},
	\end{equation*}
	so 
	\begin{align*}
		\ub{y_1 \cp	\fc{\ka_1}{x_1}\fc{\ta_3}{x_3^{n_3}y_3^{m_3}}x_2^{n_2}y_2^{m_2}}_{=:\al} +\ub{\sum_{j=n_2-n_3-1}^{n_2-1} {\ka_1} \fc{\ta_3}{x_3^{n_3-n_2+j+1}y_3^{m_3+n_2-j-1}}x_2^{j}y_2^{M-j}}_{=:\be}\in \ker \varphi.
	\end{align*}
	The condition $m_2<m_3-1$ implies that $\ker\varphi = 0$ in this degree from Proposition \ref{kerphi}, so we are done. 
	
	Now suppose that the result holds for $1\leq n_1 \leq N_1.$
	
	If $m_2 \geq m_3-1\geq 0$, then 
	\begin{align*}
		\ub{y_1 \cp	\fc{\ka_1}{x_1^{N_1+1}}\fc{\ta_3}{x_3^{n_3}y_3^{m_3}}y_2^{m_2}x_2^{n_2}}_{=:\al}+ &\ub{\sum_{j=n_2}^{m_3-1+n_2}\fc{\ka_1}{x_1^{N_1}} \fc{\ta_3}{x_3^{n_3-n_2+j+1}y_3^{m_3+n_2-j-1}}y_2^{M-j}x_2^{j}}_{=:\be}\in \ker\varphi 
	\end{align*}
	by the inductive hypothesis. If $\ker \varphi = 0$ in this degree, then we are done. Otherwise let
	\begin{align*}
		\gamma := \ka_3^{m_3+n_3+1}\fc{\ta_1}{y_1^{m_3+n_3+N_1}}x_2^{n_2-n_3-1-N_1}y_2^{m_2-m_3+1+N_1}
	\end{align*}
	be the generator of $\ker \varphi$. Suppose for contradiction that $\al+\be = \gamma.$ By definition, each of the summands of $\be$ are zero under multiplication with $y_1^{N_1}y_3^{n_3+m_3},$ so we must have that $y_1^{N_1}y_3^{n_3+m_3} \cp \al \neq 0$ which implies  $y_1^{N_1+1}y_3^{n_3+m_3} \cp \fc{\ka_1}{x_1^{N_1+1}}\fc{\ta_3}{x_3^{n_3}y_3^{m_3}}y_2^{m_2}x_2^{n_2} \neq 0.$ But this contradicts the definition of $\fc{\ka_1}{x_1^{N_1+1}}\fc{\ta_3}{x_3^{n_3}y_3^{m_3}}x_2^{n_2}y_2^{m_2}.$
		
	The case  $m_2<m_3-1$ follows similarly using the induction step.
\end{proof}
The following proposition gives a non-trivial product between the motivic cohomology class $y_3$ and a motivic cohomology class from the Mixed Cone of Type II.
\begin{proposition}
	\label{y3case}
	Let $m_3,n_1\geq 1.$ Suppose that $m_2+n_2 \geq m_3+n_3.$
	We have that
	\begin{equation*}
		\resizebox{\textwidth}{!}{%
			$\begin{aligned}
				&(1) \quad y_3\cp  \fc{\ka_1}{x_1^{n_1}}\fc{\ta_3}{x_3^{n_3}y_3^{m_3}}x_2^{n_2}y_2^{m_2} = \fc{\ka_1}{x_1^{n_1}}\fc{\ta_3}{x_3^{n_3}y_3^{m_3-1}}x_2^{n_2}y_2^{m_2}+\\
				&+{n_3+m_3-m_2-2 \choose n_3}\ka_3^{n_3+m_3}\fc{\ta_1}{x_1^{m_2-m_3+n_1+1}y_1^{2m_3+n_3-m_2-2}}x_2^{n_2+m_2-n_3-m_3+1} \tx{if $2\leq m_3-m_2\leq n_1+1$ },\\
				&(2) \quad y_3\cp  \fc{\ka_1}{x_1^{n_1}}\fc{\ta_3}{x_3^{n_3}y_3^{m_3}}x_2^{n_2}y_2^{m_2} = \fc{\ka_1}{x_1^{n_1}}\fc{\ta_3}{x_3^{n_3}y_3^{m_3-1}}x_2^{n_2}y_2^{m_2}  \tx{otherwise}.
			\end{aligned}$
		}
	\end{equation*}
\end{proposition}
\begin{proof}
	We proceed by induction over $n_1$. Let $n_1=1.$ Then
	\begin{align*}
		\ub{y_3\cp  \fc{\ka_1}{x_1}\fc{\ta_3}{x_3^{n_3}y_3^{m_3}}x_2^{n_2}y_2^{m_2}}_{=:\al} + \ub{\fc{\ka_1}{x_1}\fc{\ta_3}{x_3^{n_3}y_3^{m_3-1}}x_2^{n_2}y_2^{m_2}}_{=:\be} \in \ker\varphi.
	\end{align*}
	If $m_3-m_2 > 2$ or $m_3-m_2 < n_3+m_3-n_2-m_2+1,$ then $\ker \varphi = 0$, so we are done. If ${m_3-m_2 \in  [n_3+m_3-n_2-m_2+1,1]},$ then $\ker \varphi = \Z/2,$ so either $\al+\be=0$ or $\al+\be$ is the generator of $\ker\varphi.$ Suppose for contradiction that $\al+\be = \gamma,$ where 
	\begin{align*}
		\gamma := \ka_3^{n_3+m_3}\fc{\ta_1}{y_1^{n_3+m_3}}x_2^{n_2-n_3-1}y_2^{m_2-m_3+2}.
	\end{align*}
	By definition, we have that $y_1y_3^{n_3+m_3-1} \cp \be = 0,$ so $y_1y_3^{n_3+m_3-1} \cp \al \neq 0,$ and this implies $$y_1y_3^{n_3+m_3} \cp \fc{\ka_1}{x_1}\fc{\ta_3}{x_3^{n_3}y_3^{m_3}}x_2^{n_2}y_2^{m_2} \neq 0.$$
	
	But by Proposition \ref{y1case} (1), we have that
	\begin{equation*}
		\resizebox{\textwidth}{!}{%
			$\begin{aligned}
				y_1y_3^{n_3+m_3} \cp \fc{\ka_1}{x_1}\fc{\ta_3}{x_3^{n_3}y_3^{m_3}}x_2^{n_2}y_2^{m_2} = y_3^{n_3+m_3}\sum_{j=n_2}^{m_3-1+n_2}	\ka_1 \fc{\ta_3}{x_3^{n_3-n_2+j+1}y_3^{m_3+n_2-j-1}}x_2^{j}y_2^{m_2+n_2-j} = 0,
			\end{aligned}$
		}
	\end{equation*}
	a contradiction.\\
	If $m_3-m_2 = 2,$ then $\ker \varphi = \Z/2.$ Then suppose for contradiction that
		$$\ub{y_3\cp  \fc{\ka_1}{x_1}\fc{\ta_3}{x_3^{n_3}y_3^{m_2+2}}x_2^{n_2}y_2^{m_2}}_{=:\al} + \ub{\fc{\ka_1}{x_1}\fc{\ta_3}{x_3^{n_3}y_3^{m_2+1}}x_2^{n_2}y_2^{m_2}}_{=:\be} \neq\ub{\ka_3^{n_3+m_2+2}\fc{\ta_1}{y_1^{n_3+m_2+2}}x_2^{n_2-n_3-1}}_{=:\gamma}$$
This implies $\al + \be = 0.$
	By definition, we have that $y_1y_3^{n_3+m_2+1} \cp \be = 0,$ so $$y_1y_3^{n_3+m_2+1} \cp \al = 0=y_1y_3^{n_3+m_2+2} \cp \fc{\ka_1}{x_1}\fc{\ta_3}{x_3^{n_3}y_3^{m_2+2}}x_2^{n_2}y_2^{m_2} .$$
	But by Proposition \ref{y1case} (2), we have that
		$$y_1y_3^{n_3+m_2+2} \cp \fc{\ka_1}{x_1}\fc{\ta_3}{x_3^{n_3}y_3^{m_2+2}}x_2^{n_2}y_2^{m_2}= y_3^{n_3+m_2+2} \sum_{j=n_2-n_3-1}^{n_2-1}	\ka_1 \fc{\ta_3}{x_3^{n_3-n_2+j+1}y_3^{m_2+n_2-j+1}}x_2^{j}y_2^{m_2+n_2-j}=$$$$=\ka_1\theta_3x_2^{n_2-n_3-1}y_2^{m_2+n_3+1} \neq 0,$$
	a contradiction. 
	
	Now suppose that the result holds for all $1\leq n_1 \leq N_1.$ If $m_3-m_2\geq N_1+2>N_1+1,$ then by the inductive hypothesis we have that
	\begin{align*}
		\ub{y_3\cp  \fc{\ka_1}{x_1^{N_1+1}}\fc{\ta_3}{x_3^{n_3}y_3^{m_3}}x_2^{n_2}y_2^{m_2}}_{=:\al} + \ub{\fc{\ka_1}{x_1^{N_1+1}}\fc{\ta_3}{x_3^{n_3}y_3^{m_3-1}}x_2^{n_2}y_2^{m_2}}_{=:\be} \in \ker\varphi.
	\end{align*}
	If $m_3-m_2 > N_1+2,$ then $\ker\varphi = 0$, so we are done. If $m_3-m_2 = N_1+2,$ then we have that
	\begin{align*}
		\ker \varphi = \Z/2\left\{\ka_3^{n_3+m_2+N_1+2}\fc{\ta_1}{y_1^{2N_1+m_2+2+n_3}}x_2^{n_2-n_3-N_1-1}\right\}.
	\end{align*}
	By definition, we have that $y_1^{N_1+1}y_3^{n_3+m_3-1} \cp \be = 0,$ so
	\begin{equation*}
		\resizebox{\textwidth}{!}{%
			$\begin{aligned}
				y_1^{N_1+1}y_3^{n_3+m_3-1} \cp (\al+\be)=0 \iff y_1^{N_1+1}y_3^{n_3+m_3-1}\cp \al = 0 \iff y_1^{N_1+1}y_3^{n_3+m_3} \cp \fc{\ka_1}{x_1^{N_1+1}}\fc{\ta_3}{x_3^{n_3}y_3^{m_3}}x_2^{n_2}y_2^{m_2} = 0.
			\end{aligned}$
		}
	\end{equation*}
	Since $m_3-m_2 =N_1+2 > 1,$ by Proposition \ref{y1case} (2) we have that
	\begin{equation*}
		\resizebox{\textwidth}{!}{%
			$\begin{aligned}
				y_1^{N_1+1}y_3^{n_3+m_3} \cp \fc{\ka_1}{x_1^{N_1+1}}\fc{\ta_3}{x_3^{n_3}y_3^{m_3}}x_2^{n_2}y_2^{m_2} = y_3^{n_3+m_3}y_1^{N_1}\sum_{j=n_2-n_3-1}^{n_2-1}\fc{\ka_1}{x_1^{N_1}} \fc{\ta_3}{x_3^{n_3-n_2+j+1}y_3^{m_3+n_2-j-1}}x_2^{j}y_2^{m_2+n_2-j}.
			\end{aligned}$
		}
	\end{equation*}
	
	Then $m_3+n_2-j-1-(m_2+n_2-j) = m_3-m_2-1 = N_1 + 1 > 1,$ since $N_1 \geq 1,$ so we can repeatedly apply Proposition \ref{y1case} (2) to the RHS until we are left with
	\begin{equation*}
		\resizebox{\textwidth}{!}{%
			$\begin{aligned}
				&y_3^{n_3+m_3} \cp \left({n_3+N_1 \choose n_3}\ka_1 \fc{\ta_3}{y_3^{m_3+n_3}}x_2^{n_2-n_3-N_1-1}y_2^{m_2+n_3+N_1+1}+\text{terms $\ka_1\fc{\ta_3}{x_3^{k_i'}y_3^{k_i}}x_2^{k_i''}y_2^{k_i'''}$ where $k_i < m_3+n_3$}\right)\\
				=& {n_3+N_1 \choose n_3}\ka_1\ta_3 x_2^{n_2-n_3-N_1-1}y_2^{m_2+n_3+N_1+1}.
			\end{aligned}$
		}
	\end{equation*}
	We conclude that 
	\begin{align*}
		\al + \be = {n_3+N_1 \choose n_3}\ka_3^{n_3+m_2+N_1+2}\fc{\ta_1}{y_1^{2N_1+m_2+2+n_3}}x_2^{n_2-n_3-N_1-1}.
	\end{align*}
	If $2 \leq m_3-m_2 \leq N_1+1,$ then by the inductive hypothesis we have that
	\begin{align*}
		&\ub{y_3\cp  \fc{\ka_1}{x_1^{N_1+1}}\fc{\ta_3}{x_3^{n_3}y_3^{m_3}}x_2^{n_2}y_2^{m_2}}_{=:\al} + \ub{\fc{\ka_1}{x_1^{N_1+1}}\fc{\ta_3}{x_3^{n_3}y_3^{m_3-1}}x_2^{n_2}y_2^{m_2}}_{=:\be_1}\\
		&+ \ub{{n_3+m_3-m_2-2 \choose n_3}\ka_3^{n_3+m_3}\fc{\ta_1}{x_1^{m_2-m_3+N_1+2}y_1^{2m_3+n_3-m_2-2}}x_2^{n_2+m_2-n_3-m_3+1}}_{=:\be_2} \in \ker\varphi.
	\end{align*}
	If $\ker\varphi = 0$ in this degree, then we are done. Otherwise let
	\begin{align*}
		\gamma := \ka_3^{n_3+m_3}\fc{\ta_1}{y_1^{N_1+m_3+n_3}}x_2^{n_2-n_3-N_1-1}y_2^{m_2-m_3+N_1+2}
	\end{align*}
	be the generator of $\ker \varphi,$ whose multiplication by $y_1^{N_1+1}y_3^{n_3+m_3-1}$ is nonzero. Then
	\begin{align*}
		\al +\be_1+\be_2 = 0 \iff y_1^{N_1+1}y_3^{n_3+m_3-1} \cp (\al +\be_1+\be_2) = 0.
	\end{align*}
	By definition, we have that
	\begin{align*}
		y_1^{N_1+1}y_3^{n_3+m_3-1} \cp \be_1 = 0.
	\end{align*}
	Similarly, we also have by definition that 
	\begin{align*}
		y_1^{N_1+1}y_3^{n_3+m_3-1} \cp \al = y_1^{N_1+1}y_3^{n_3+m_3}\fc{\ka_1}{x_1^{N_1+1}}\fc{\ta_3}{x_3^{n_3}y_3^{m_3}}x_2^{n_2}y_2^{m_2} = 0,
	\end{align*}
	since the condition $m_3-m_2 \leq N_1+1$ implies that if  $$\ka_3^{n_3+m_3}\fc{\ta_1}{y_1^{N_1+m_3+n_3}}x_2^{n_2-n_3-N_1-1}y_2^{m_2-m_3+N_1+2}$$ is a kernel element, then so is $$\ka_3^{n_3+m_3+1}\fc{\ta_1}{y_1^{N_1+m_3+n_3+1}}x_2^{n_2-n_3-N_1-1}y_2^{m_2-m_3+N_1+1}.$$ Finally, we have that ${y_1^{N_1+1}y_3^{n_3+m_3-1} \cp \be_2 = 0,}$ since the condition ${m_3-m_2 \leq N_1+1}$ implies that ${m_3-m_2-1 \leq N_1 < N_1+1.}$\\
	If $m_3-m_2 <  2 \leq N_1+1,$ then by the inductive hypothesis we have that $\al+\be \in \ker\varphi$ as before. If $\ker\varphi = 0$ in this degree, then we are done. Otherwise let $\gamma$
	be the generator of $\ker \varphi$ as before. Then
	\begin{align*}
		\al +\be = 0 \iff y_1^{N_1+1}y_3^{n_3+m_3-1} \cp (\al +\be) = 0.
	\end{align*}
	We have by definition that $y_1^{N_1+1}y_3^{n_3+m_3-1} \cp \be = 0.$ Similarly, we also have by definition that 
	\begin{align*}
		y_1^{N_1+1}y_3^{n_3+m_3-1} \cp \al = y_1^{N_1+1}y_3^{n_3+m_3}\fc{\ka_1}{x_1^{N_1+1}}\fc{\ta_3}{x_3^{n_3}y_3^{m_3}}x_2^{n_2}y_2^{m_2} = 0,
	\end{align*}
	since the condition $m_3-m_2 < 2\leq  N_1+1$ implies that if $$\ka_3^{n_3+m_3}\fc{\ta_1}{y_1^{N_1+m_3+n_3}}x_2^{n_2-n_3-N_1-1}y_2^{m_2-m_3+N_1+2}$$ is a kernel element, then so is $$\ka_3^{n_3+m_3+1}\fc{\ta_1}{y_1^{N_1+m_3+n_3+1}}x_2^{n_2-n_3-N_1-1}y_2^{m_2-m_3+N_1+1}.$$
	
\end{proof}

\begin{proposition}
	\label{x3case2}
	Let $n_1,m_3\geq 1.$ Suppose that $m_2+n_2 \geq m_3.$
	We have that
	\begin{equation*}
		\resizebox{\textwidth}{!}{%
			$\begin{aligned}
				(1) \quad x_3\cp  \fc{\ka_1}{x_1^{n_1}}\fc{\ta_3}{y_3^{m_3}}x_2^{n_2}y_2^{m_2} &= \ka_3^{m_3}\fc{\ta_1}{x_1^{m_2-m_3+n_1}y_1^{2m_3-m_2-1}}x_2^{n_2+m_2-m_3+1} \tx{if $2\leq m_3-m_2\leq n_1$ },\\
				(2) \quad x_3\cp  \fc{\ka_1}{x_1^{n_1}}\fc{\ta_3}{y_3^{m_3}}x_2^{n_2}y_2^{m_2} &=  \ka_3^{m_3}\fc{\ta_1}{x_1^{n_1-1}y_1^{m_3}}x_2^{n_2}y_2^{m_2-m_3+1} \tx{if $m_3-m_2\leq 1$ },\\
				(3) \quad x_3\cp  \fc{\ka_1}{x_1^{n_1}}\fc{\ta_3}{y_3^{m_3}}x_2^{n_2}y_2^{m_2} &= 0  \tx{otherwise}.
			\end{aligned}$
		}
	\end{equation*}
\end{proposition}
\begin{proof}
	When $n_1 = 1,$ we have that
	\begin{align*}
		\ub{x_3\cp  \fc{\ka_1}{x_1}\fc{\ta_3}{y_3^{m_3}}x_2^{n_2}y_2^{m_2}}_{=:\al} \in \ker \varphi.
	\end{align*}
	If ${m_3-m_2 > 1}$, then $\ker\varphi = 0,$ so we are done. Otherwise, we have that 
	\begin{align*}
		\al=\ga := \ka_3^{m_3}\fc{\ta_1}{y_1^{m_3}}x_2^{n_2}y_2^{m_2-m_3+1}\iff y_1y_3^{m_3-1} \cp \al \neq 0.
	\end{align*}
	Since $m_3-m_2 \leq 1,$ we can apply Proposition $\ref{y1case}$ (1) to obtain
	\begin{align*}
		y_1y_3^{m_3-1} \cp \al  = y_3^{m_3-1}x_3 \sum_{j=n_2}^{m_3-1+n_2}\ka_1 \fc{\ta_3}{x_3^{-n_2+j+1}y_3^{m_3+n_2-j-1}}x_2^{j}y_2^{m_2+n_2-j} \neq 0,
	\end{align*}
	so we are done. For $n_1 \geq 2,$ we will proceed by induction. Let $n_1=2.$  When $m_3-m_2 \leq 1$, we have from the $n_1=1$ case that
	\begin{align*}
		\ub{x_3\cp  \fc{\ka_1}{x_1^2}\fc{\ta_3}{y_3^{m_3}}x_2^{n_2}y_2^{m_2}}_{=:\al} +\ub{\ka_3^{m_3}\fc{\ta_1}{x_1y_1^{m_3}}x_2^{n_2}y_2^{m_2-m_3+1}}_{=:\be} \in \ker \varphi.
	\end{align*}
	
	If $\ker\varphi = 0$ in this degree, then we are done. Otherwise, we have that
	\begin{align*}
		\al+\be = \ga := \ka_3^{m_3}\fc{\ta_1}{y_1^{m_3+1}}x_2^{n_2-1}y_2^{m_2-m_3+2}\iff y_1^2y_3^{m_3-1} \cp \al \neq 0,
	\end{align*}
	since $y_1^2y_3^{m_3-1} \cp \be = 0.$  Now we can apply Proposition \ref{y1case}(1) twice to $\al$ to obtain
	\begin{align*}
		y_1^2y_3^{m_3-1} \cp \al  = y_3^{m_3-1} x_3\cp \left(\text{terms $\ka_1\fc{\ta_3}{x_3^{k_i'}y_3^{k_i}}x_2^{k_i''}y_2^{k_i'''}$ where $k_i\leq m_3-2< m_3-1$}\right)=0, 
	\end{align*}
	so we are done. If $m_3-m_2=2,$ then
	we have from the $n_1=1$ case that
	\begin{align*}
		\ub{x_3\cp  \fc{\ka_1}{x_1^2}\fc{\ta_3}{y_3^{m_3}}x_2^{n_2}y_2^{m_2}}_{=:\al} \in \ker \varphi.
	\end{align*}
	Then 
	\begin{align*}
		\al = \ga := \ka_3^{m_3}\fc{\ta_1}{y_1^{m_3+1}}x_2^{n_2-1}\iff y_1^2y_3^{m_3-1} \cp \al \neq 0.
	\end{align*}
	We can apply Proposition \ref{y1case} (2) to obtain
	\begin{align*}
		&y_1^2y_3^{m_3-1} \cp \al  = y_1y_3^{m_3-1}x_3\sum_{j=n_2-1}^{n_2-1}\fc{\ka_1}{x_1} \fc{\ta_3}{x_3^{-n_2+j+1}y_3^{m_3+n_2-j-1}}x_2^{j}y_2^{m_2+n_2-j}\\
		&= y_1y_3^{m_3-1}x_3\fc{\ka_1}{x_1} \fc{\ta_3}{y_3^{m_3}}x_2^{n_2-1}y_2^{m_2+1}\\
		&= y_3^{m_3-1}x_3\sum_{j=n_2-1}^{m_3+n_2-2}\ka_1\fc{\ta_3}{x_3^{-n_2+2+j}y_3^{m_3+n_2-2-j}}x_2^jy_2^{m_2+n_2-j}\tx{(Proposition \ref{y1case}(1))}\\
		&= y_3^{m_3-1}x_3 \ka_1\fc{\ta_3}{x_3y_3^{m_3-1}}x_2^{n_2-1}y_2^{m_2+1} = \ka_1\ta_3 x_2^{n_2-1}y_2^{m_2+1} \neq 0,
	\end{align*}
	so we are done. Now suppose that the result holds for $2\leq n_1 \leq N_1.$ If $m_3-m_2 \geq N_1+1,$ then 
	\begin{align*}
		\ub{x_3\cp  \fc{\ka_1}{x_1^{N_1+1}}\fc{\ta_3}{y_3^{m_3}}x_2^{n_2}y_2^{m_2}}_{=:\al} \in \ker \varphi
	\end{align*}
	by the inductive hypothesis. If ${m_3-m_2> N_1+1,}$ then $\ker\varphi=0$ in this degree so we are done. If ${m_3-m_2=N_1+1,}$ then
	\begin{align*}
		\al = \ga := \ka_3^{m_3}\fc{\ta_1}{y_1^{N_1+m_3}}x_2^{n_2-N_1}\iff y_1^{N_1+1}y_3^{m_3-1} \cp \al \neq 0.
	\end{align*}
	Then we can apply Proposition \ref{y1case} (2) $N_1$ times to obtain
	\begin{equation*}
		\resizebox{\textwidth}{!}{%
			$\begin{aligned}
				&	y_1^{N_1+1}y_3^{m_3-1} \cp \al = 	y_1^{N_1+1}y_3^{m_3-1} x_3\cp  \fc{\ka_1}{x_1^{N_1+1}}\fc{\ta_3}{y_3^{m_3}}x_2^{n_2}y_2^{m_2}=y_3^{m_3-1}x_3y_1 \cp\fc{\ka_1}{x_1} \fc{\ta_3}{y_3^{m_3}}x_2^{n_2-N_1}y_2^{m_2+N_1}\\
				&= y_3^{m_3-1}x_3 \cp \sum_{j=n_2-N_1}^{m_3+n_2-N_1-1}{\ka_1} \fc{\ta_3}{x_3^{N_1-n_2+j+1}y_3^{m_3+n_2-N_1-1-j}}x_2^{j}y_2^{m_2+n_2-j}\tx{(Proposition \ref{y1case}(1))}\\
				&= y_3^{m_3-1}x_3 \cp \ka_1\fc{\ta_3}{x_3y_3^{m_3-1}}x_2^{n_2-N_1}y_2^{m_2+N_1} =  \ka_1 \ta_3 x_2^{n_2-N_1}y_2^{m_2+N_1}\neq 0,
			\end{aligned}$
		}
	\end{equation*}
	so we are done. If $2 \leq m_3-m_2 \leq N_1,$ then by the inductive hypothesis we have that
	\begin{align*}
		\ub{x_3\cp  \fc{\ka_1}{x_1^{N_1+1}}\fc{\ta_3}{y_3^{m_3}}x_2^{n_2}y_2^{m_2}}_{=:\al} &+\ub{\ka_3^{m_3}\fc{\ta_1}{x_1^{m_2-m_3+N_1+1}y_1^{2m_3-m_2-1}}x_2^{n_2+m_2-m_3+1}}_{=:\be}\in \ker \varphi.
	\end{align*}
	If $\ker\varphi = 0$ in this degree, then we are done. Otherwise, we have that
	\begin{align*}
		\al+\be= \ga := \ka_3^{m_3}\fc{\ta_1}{y_1^{N_1+m_3}}x_2^{n_2-N_1}y_2^{m_2-m_3+N_1+1}\iff y_1^{N_1+1}y_3^{m_3-1}\cp (\al+\be) \neq 0.
	\end{align*}
	We have that ${y_1^{N_1+1}y_3^{m_3-1} \cp \be = 0,}$ since the condition ${m_3-m_2 \leq N_1}$ implies that\\ ${m_3-m_2 < N_1+1.}$ Then it remains to show that ${y_1^{N_1+1}y_3^{m_3-1}\cp \al = 0.}$ We have that
	\begin{equation*}
		\resizebox{\textwidth}{!}{%
			$\begin{aligned}
				&y_1^{N_1+1}y_3^{m_3-1} \cp \al = y_1^{N_1+1}y_3^{m_3-1} x_3 \cp \fc{\ka_1}{x_1^{N_1+1}}\fc{\ta_3}{y_3^{m_3}}x_2^{n_2}y_2^{m_2}\\
				&=y_1^{N_1-m_3+m_2}y_1^{2}y_1^{m_3-m_2-1}y_3^{m_3-1} x_3 \cp \fc{\ka_1}{x_1^{N_1+1}}\fc{\ta_3}{y_3^{m_3}}x_2^{n_2}y_2^{m_2} \tx{(since $2\leq m_3-m_2\leq N_1$)}\\
				&= y_1^{N_1-m_3+m_2}y_1^{2}y_3^{m_3-1} x_3\cp\left(\text{terms $\fc{\ka_1}{x_1^{N_1-m_3+m_2+2}}\fc{\ta_3}{x_3^{k_i'}y_3^{k_i}}x_2^{k_i''}y_2^{k_i'''}$ where $k_i \leq m_3$}\right) \tx{(Proposition \ref{y1case} (2))}\\
				&= y_1^{N_1-m_3+m_2}y_3^{m_3-1} x_3\cp\left(\text{terms $\fc{\ka_1}{x_1^{N_1-m_3+m_2}}\fc{\ta_3}{x_3^{k_i'}y_3^{k_i}}x_2^{k_i''}y_2^{k_i'''}$ where $k_i \leq m_3-2$}\right) \tx{(Proposition \ref{y1case} (1))}\\
				&= y_3^{m_3-1} x_3\cp\left(\text{terms ${\ka_1}\fc{\ta_3}{x_3^{k_i'}y_3^{k_i}}x_2^{k_i''}y_2^{k_i'''}$ where $k_i \leq m_3-2$}\right) \tx{(Proposition \ref{y1case} (1))}\\
				&= 0.
			\end{aligned}$
		}
	\end{equation*}
	
	If $m_3-m_2 < 2 \leq N_1,$ then by the inductive hypothesis we have that
	\begin{align*}
		\ub{x_3\cp  \fc{\ka_1}{x_1^{N_1+1}}\fc{\ta_3}{y_3^{m_3}}x_2^{n_2}y_2^{m_2}}_{=:\al}+\ub{\ka_3^{m_3}\fc{\ta_1}{x_1^{N_1}y_1^{m_3}}x_2^{n_2}y_2^{m_2-m_3+1}}_{=:\be} \in \ker \varphi.
	\end{align*}
	If $\ker\varphi = 0$ in this degree, then we are done. Otherwise, we have that
	\begin{align*}
		\al = \ga := \ka_3^{m_3}\fc{\ta_1}{y_1^{N_1+m_3}}x_2^{n_2-N_1}y_2^{m_2-m_3+N_1+1}\iff y_1^{N_1+1}y_3^{m_3-1}\cp (\al+\be) \neq 0.
	\end{align*}
	We have that ${y_1^{N_1+1}y_3^{m_3-1}\cp\be = 0}$ so it remains to check $\al.$ We have that
	\begin{align*}
		&y_1^{N_1+1}y_3^{m_3-1} \cp \al = y_1^{N_1+1}y_3^{m_3-1} x_3 \cp \fc{\ka_1}{x_1^{N_1+1}}\fc{\ta_3}{y_3^{m_3}}x_2^{n_2}y_2^{m_2}\\
		&=y_1^{N_1-1}y_1^{2}y_3^{m_3-1} x_3 \cp \fc{\ka_1}{x_1^{N_1+1}}\fc{\ta_3}{y_3^{m_3}}x_2^{n_2}y_2^{m_2} \tx{(since $N_1\geq 2$)}\\
		&= y_3^{m_3-1} x_3\cp\left(\text{terms ${\ka_1}\fc{\ta_3}{x_3^{k_i'}y_3^{k_i}}x_2^{k_i''}y_2^{k_i'''}$ where $k_i \leq m_3-2$}\right) \tx{(Proposition \ref{y1case} (1))}\\
		&= 0,
	\end{align*}
	so we are done.
\end{proof}

	
	

%
\subsection{Poincar\'e Series and Generators}
In this subsection we prove some results used in the main text about the correspondence of Poincar\'e series and generators.
\begin{lemma}
	\label{k1series}
	We have that  $x^q(1+\cd+x^{b-1})(1+\cd+x^{-q-1})$ is the Poincar\'e series for products of the form
	\begin{align*}
		\ka_1\fc{\ta_3}{x_3^{n_3}y_3^{m_3}}x_2^{n_2}y_2^{m_2}
	\end{align*}
	when $b+q \geq 0$ in the Mixed Cone of Type II corresponding to $\io_2.$ More precisely, the number of products in $H^{-a'+p\sigma+b\epsilon+q\sigma\otimes\epsilon}_{Br,K}(pt)$, $b+q\geq 0, p,q\leq -1,b\geq 1$ is given by the coefficient of $x^{a'}$ in the above Poincar\'e series.
\end{lemma}

\begin{proof}
	From the equations
	\begin{align*}
		a' &= m_2 - m_3 - 1\\
		p &= -1\\
		b &= n_2 + m_2 + 1\\
		q &= -n_3-m_3-1
	\end{align*}
	we obtain the equations
	\begin{align*}
		m_2 &= b-n_2-1\\
		m_3 &= b-n_2-2-a'\\
		n_3 &= -q-m_3-1= -(b+q)+n_2+a'+1. 
	\end{align*}
	We then have the equivalences
	\begin{align*}
		m_2 \geq 0 &\iff b-1 \geq n_2\\
		m_3 \geq 0 &\iff n_2 \leq b-2-a'\\
		n_3 \geq 0 &\iff n_2 \geq b+q-a'-1,
	\end{align*}
	from which we see that 
	\begin{align*}
		n_2 \in [0,b-1] \cap [b+q-a'-1, b-2-a'].
	\end{align*}
	Then 
	\begin{align*}
		\# \text{products }_{a'}
		&=  \left|[0,b-1] \cap [b+q-a'-1, b-2-a']\right|\\ &=\min(b-1,b-2-a')-\max(b+q-a'-1,0)+1.
	\end{align*}
	If the last expression is negative, then it should be interpreted as zero. Evaluating this expression, we obtain
	\begin{align*}
				\# \text{products }_{a'} =\begin{cases}
					a'-q+1& \text{if } q\leq a' \leq -1,\\
					-q & \text{if } -1\leq a' \leq b+q-1,\\
				b-1-a' & \text{if } b+q-1\leq a' \leq b-2,\\
				\end{cases} 
			\end{align*}
	and from Lemma \ref{coeff} we see that $\# \text{products }_{a'}$ coincides with the coefficient of $x^{a'}$ in the Poincar\'e series 
	\begin{align*}
		x^q(1+\cd+x^{b-1})(1+\cd+x^{-q-1}).
	\end{align*}
\end{proof}

\begin{lemma}
	\label{k3series}
	We have that $x^p(1+\cd+x^{-p-2})(1+\cd+x^{b+q})$ is the Poincar\'e series for products of the form
	\begin{align*}
		\ka_3^{m_3+1}\cdot \fc{\ta_1}{x_1^{n_1}y_1^{m_3+m_1}}	\cdot x_2^{n_2}y_2^{m_2-m_3},
	\end{align*}
	when $b+q \geq 0$ in the Mixed Cone of Type II corresponding to $\io_2,$ and $0\leq m_3\leq m_2$ and $m_1 \geq 1.$ 
	
	More precisely, the coefficient of $x^{a'}$ in the above Poincar\'e series gives the number of cohomology classes of the above form in $H^{-a'+p\sigma+b\epsilon+q\sigma\otimes\epsilon}_{Br,K}(pt)$, $b+q\geq 0$, $p,q\leq -1$ and $b\geq 1$.
\end{lemma}

\begin{proof}
	From the equations
	\begin{align*}
		a' &= m_2 - m_1-m_3 - 1\\
		p &= -1-n_1-m_1\\
		b &= n_2 + m_2 + 1\\
		q &= -m_3-1
	\end{align*}
	we obtain the equations
	\begin{align*}
		m_1 &= m_2 + q - a'\\
		n_1 &= -1+a'-m_2-q-p\\
		n_2 &= b-1-m_2.
	\end{align*}
	We then have the equivalences
	\begin{align*}
		n_1 \geq 0 &\iff m_2 \leq  -1+a'-q-p\\
		n_2 \geq 0 & \iff m_2 \leq b-1\\
		m_3 \leq m_2 & \iff m_2 \geq -q-1\\
		m_1 \geq 1 & \iff m_2 \geq a'-q+1,
	\end{align*}
	from which we see that 
	\begin{align*}
		m_2 \in [-q-1,b-1] \cap [a'-q+1, a'-q-p-1].
	\end{align*}
	Then 
	\begin{align*}
		\# \text{products }_{a'}
		&=  \left| [-q-1,b-1] \cap [a'-q+1, a'-q-p-1]\right|\\ &=\min(b-1,a'-q-p-1)-\max(-q-1,a'-q+1)+1.
	\end{align*}
	If the last expression is negative, then it should be interpreted as zero. Evaluating this expression, we obtain 
		\begin{align*}
				\# \text{products }_{a'}  =\begin{cases}
					a'-p+1& \text{if } p\leq a' \leq b+q+p,\\
					-q & \text{if } b+q+p\leq a' \leq -2,\\
				b-1-a'+q & \text{if } -2\leq a' \leq b+q-2,\\
				\end{cases} 
			\end{align*}
when we assume $b+q+1\leq -p-1$ (the other case is symmetric) and therefore we see that by Lemma \ref{coeff} that $\# \text{products }_{a'}$  in $H^{-a'+p\sigma+b\epsilon+q\sigma\otimes\epsilon}_{Br,K}(pt)$ agrees with the coefficient of $x^{a'}$ in the Poincar\'e series 
	\begin{align*}
		x^p(1+\cd+x^{-p-2})(1+\cd+x^{b+q})
	\end{align*}
\end{proof}

\begin{lemma}
	\label{kerphi}
	The Poincar\'e series for the dimension of  kernel of $$\varphi_{p-1}:=\cp x_1: \pi_{a'}((\Si^{(p-1)\si+b\ep+q\so}H\ul{\Z/2})^{\mathcal{K}}) \twoheadrightarrow  \pi_{a'}((\Si^{p\si+b\ep+q\so}H\ul{\Z/2})^{\mathcal{K}})$$ is given by
	\begin{align*}
		x^{p-1}(1+\cd+x^{b+q}),
	\end{align*}
	when $b+q \geq 0$ in the Mixed Cone of Type II corresponding to $\io_2.$ 
	\end{lemma}
\begin{proof}
	By Lemma \ref{x1case1}, we have that $\varphi_{p-1}$ is surjective, so by Theorem \ref{caseMC2} (2), the Poincar\'e series for $\ker\varphi_{p-1}$ is given by
	\begin{align*}
		&{x^{p-1}}(1+...+x^{-p-1})(1+...+x^{b+q})+{x^q}(1+...+x^{b-1})(1+...+x^{-q-1})\\
		-&{x^p}(1+...+x^{-p-2})(1+...+x^{b+q})-{x^q}(1+...+x^{b-1})(1+...+x^{-q-1})\\
		&= 		x^{p-1}(1+\cd+x^{b+q}).
	\end{align*}
\end{proof}

\begin{theorem} 
		\label{ioproof}
	The Poincar\'e series for $\dim\ker f_{-a',p,b,q}$ in the Mixed Cone of Type II corresponding to $\io_2$ is given by:\\
	(1) \textbf{Case} $-p,-q\geq b+1$:
	\begin{align*}
		(x-1)^{-2}x^{-b - 1} (x^b - 1) (x^{b + 1} - 1).
	\end{align*}
	The index condition is that of  case (1) in Proposition \ref{caseMC2}.
	
	(2) \textbf{Case} $-q\leq b$:
	\begin{align*}
		(x-1)^{-2}\left( x^p(x^{-p-1} - 1) (x^{b +q+ 1} - 1) + x^q(x^b-1)(x^{-q}-1)\right).
	\end{align*}
	The index condition is that of the case (2) in Proposition \ref{caseMC2}.
\end{theorem}
\begin{proof}
	Let $T_{a,p,b,q}$ be the set of $\ta_1\ta_3$-generators in degree $(a,p,b,q);$ i.e. elements of the form
	\begin{align*}
		\fc{\ta_1\ta_3}{x_1^{n_1}y_1^{m_1}x_3^{n_3}y_3^{m_3}}x_2^{n_2}y_2^{m_2}.
	\end{align*}
	Recall that these elements give  not necessary nonzero cohomology classes  
		\begin{align*}
		\fc{\ta_1}{x_1^{n_1}y_1^{m_1}}\fc{\ta_3}{x_3^{n_3}y_3^{m_3}}x_2^{n_2}y_2^{m_2}.
	\end{align*}
	
	Similarly, let $I_{a,p,b,q}$ be the set of $\io_2$-generators in degree $(a,p,b,q);$ i.e. elements of the form
	\begin{align*}
		\fc{\io_2}{x_1^{n_1}y_1^{m_1}x_3^{n_3}y_3^{m_3}}x_2^{n_2}y_2^{m_2}.
	\end{align*}
	From Theorem \ref{m2cone} the homology of the complex
	\begin{align*}
			0\to \ub{\frac{\Z/2[x_1,y_1,x_2,y_2,x_3,y_3]}{(x_1^\infty,y_1^\infty,x_3^\infty,y_3^\infty)}\{\io_2\}}_{I} \xrightarrow{f} \ub{\frac{\Z/2[x_1,y_1,x_2,y_2,x_3,y_3]}{(x_1^\infty,y_1^\infty,x_3^\infty, y_3^\infty)}\{\theta_1\theta_3\}}_T \to 0.
		\end{align*}
		gives the Mixed Cone of Type II corresponding to $\io_2$. In other words, this  cohomology is $\ker f \op T/fT.$ Write
		\begin{align*}
			f_{a,p,b,q} : I_{a,p,b,q} \to T_{a+1,p,b,q}
		\end{align*}
		for the restriction of $f$ to the graded piece of degree $(a,p,b,q)$, a $\Z/2$ vector space.
	Recall that we denote $a'=-a$. We first compute
		$$\dim T_{a,p,b,q}-\dim I_{a-1,p,b,q}=\dim T_{-a',p,b,q}-\dim I_{-(a'+1),p,b,q}.$$
	For $T_{a,p,b,q},$ we have the equations
	\begin{align*}
		a &= m_1+m_3-m_2+4\\
		p &= -2-n_1-m_1\\
		b &= n_2+m_2\\
		q &= -2-n_3-m_3.
	\end{align*}
	Solving these equations and applying the condition that $n_i,m_i \geq 0,$ we have the inequalities
	\begin{align*}
		m_3 &= -2-n_3-q \geq 0 \implies -2-q \geq n_3,\\
		m_2 &= b-n_2 \geq 0 \implies n_2 \leq b,\\
		m_1 &= a+m_2-m_3-4\\
		&= a+b-n_2+2+n_3+q-4 \geq 0 \implies n_2-n_3 \leq a+b+q-2, \txa\\
		n_1 &= -2-p-a-b+n_2-n_3-q+2 \geq 0 \implies n_2-n_3 \geq a+p+b+q.
	\end{align*}
	Then
	\begin{equation*}
		\resizebox{\textwidth}{!}{%
			$\begin{aligned}
				\dim T_{a,p,b,q} = \left|\left\{(n_2,n_3) \in \Z_{\geq0}^2 : n_3 \leq -q-2 \xx{ and } n_2\leq b \xx{ and } a+p+b+q \leq n_2-n_3 \leq a+b+q-2 \right\}\right|.
			\end{aligned}$
		}
	\end{equation*}
	Similarly, for $I_{a,p,b,q}$ we have the equations
	\begin{align*}
		a &= m_1+m_3-m_2+1\\
		p &= -1-n_1-m_1\\
		b &= n_2+m_2+1\\
		q &= -1-n_3-m_3.
	\end{align*}
	
	Solving these equations and applying the condition that $n_i,m_i \geq 0,$ we have the inequalities
	\begin{align*}
		m_3 &= -1-n_3-q \geq 0 \implies -1-q \geq n_3,\\
		m_2 &= b-n_2 -1\geq 0 \implies n_2 \leq b-1,\\
		m_1 &= a+m_2-m_3-1\\
		&= a+b-n_2-1+1+n_3+q-1 \geq 0 \implies n_2-n_3 \leq a+b+q-1, \txa\\
		n_1 &= -1-p-a-b+n_2-n_3-q+1 \geq 0 \implies n_2-n_3 \geq a+p+b+q.
	\end{align*}
	Then
	\begin{equation*}
		\resizebox{\textwidth}{!}{%
			$\begin{aligned}
				&\dim I_{a,p,b,q} =\left|\left\{(n_2,n_3) \in \Z_{\geq0}^2 : n_3 \leq -q-1 \xx{ and } n_2\leq b-1\xx{ and } a+p+b+q \leq n_2-n_3 \leq a+b+q-1 \right\}\right|.
			\end{aligned}$
		}
	\end{equation*}
	Then we have that $\dim T_{-a',p,b,q}-\dim I_{-(a'+1),p,b,q}$ is given by
	\begin{align*}
		&\sum_{j=0}^{-q-2}\sum_{i=0}^{b}\ic_{[-a'+p+b+q,-a'+b+q-2]}(i-j)-\sum_{j=0}^{-q-1}\sum_{i=0}^{b-1}\ic_{[-a'-1+p+b+q,-a'+b+q-2]}(i-j)\\
		=& \ob{\sum_{j=0}^{-q-2}\sum_{i=0}^{b-1}\ic_{[-a'+p+b+q,-a'+b+q-2]}(i-j)-\ic_{[-a'-1+p+b+q,-a'+b+q-2]}(i-j)}^{L:=}\\
		&+ \ob{\sum_{j=0}^{-q-2}\ic_{[-a'+p+b+q,-a'+b+q-2]}(b-j)}^{M_1:=}+ \ob{\sum_{i=0}^{b-1}-\ic_{[-a'-1+p+b+q,-a'+b+q-2]}(i+q+1)}^{M_2:=}.
	\end{align*}
	Then we have that
	\begin{align*}
		L=& -\left|\left\{(i,j):i-j = -a'-1+p+b+q\right\}\cap\left\{(i,j): i \in [0,b-1], j\in[0,-q-2]\right\}\right|,\\
		M_1 =& \left|[-a'+p+b+q,-a'+b+q-2]\cap[b+q+2,b]\right|, \txa\\
		M_2 =& \left|[-a'-1+p+b+q,-a'+b+q-2]\cap[q+1,b+q]\right|.
	\end{align*}
	Observe that, geometrically, as $a'$ varies, the magnitude of $L$ is the cardinality of the intersection of a moving line with a fixed rectangular integer lattice. As the line varies, we see that the cardinality is described precisely by a Poincar\'e series of the form appearing in Lemma \ref{coeff}. Hence, the Poincar\'e series for $L$ is given by
	\begin{align*}
		-\fc{x^{p+q}}{(x-1)^2} (x^b-1)(x^{-q-1}-1).
	\end{align*}
	Similarly, as $a'$ varies, the magnitude of $M_i$ can be described as the cardinality of the intersection of a moving interval with a fixed interval,
	which is again a Poincar\'e series of the form appearing in Lemma \ref{coeff}. Then the Poincar\'e series for $M_1$ and $M_2$ are given by
	\begin{align*}
		\fc{x^{p+q}}{(x-1)^2} (x^{-p-1}-1)(x^{-q-1}-1)\txa-\fc{x^{p-1}}{(x-1)^2}(x^b-1)(x^{-p}-1).
	\end{align*}
	respectively. Then the Poincar\'e series for $\dim T_{-a',p,b,q}-\dim I_{-(a'+1),p,b,q} = L+M_1+M_2$ is given by
	\begin{align*}
		\tag{T1}
		(x-1)^{-2}\left(x^{b + p + q} - x^{b - 1} - x^{p - 1} - x^{q - 1} + x^{-2} + x^{-1}\right).
	\end{align*}
	We have that
	\begin{align*}
		&\pi_{a'}((\Si^{p\si+b\ep+q\so}H\ul{\Z/2})^{\mathcal{K}}) =\dim H_{Br,K}^{-a'+p\sigma+b\epsilon+q\sigma\otimes\epsilon}(pt)\\
		&= \dim\ker f_{-a',p,b,q} + \dim T_{-a',p,b,q}-\dim\im f_{-a'-1,p,b,q}\\
		&= \dim\ker f_{-a',p,b,q}+ \dim\ker f_{-a'-1,p,b,q} + \dim T_{-a',p,b,q}-\dim I_{-a'-1,p,b,q}.
	\end{align*}
	By Theorem \ref{caseMC2}, the full Poincar{\'e} series is
	\begin{align*}
		(x-1)^{-2}\left(x^{p+b+q}-x^{p-1}-x^{q-1}+x^{-b-2}+x^{-b-1}+x^b-x^{-1}-1\right),
	\end{align*}
	if $-p,-q \geq b+1,$ and
	\begin{align*}
		(x-1)^{-2}\left(-x^{p+q+b+1}+x^p+x^q+x^b-x^{-1}-1\right),
	\end{align*}
	if $b \geq -q.$ Then the Poincar{\'e} series for $\dim\ker f_{-a',p,b,q}  + \dim\ker f_{-a'-1,p,b,q} $ is given by taking the full Poincar{\'e} series and subtracting (T1) which is
	\begin{align*}
		&(x-1)^{-2}\left(x^{b - 1} + x^{-b - 2} + x^{-b - 1} + x^b - x^{-2} - 2x^{-1} - 1\right)\\
		=&(x-1)^{-2}(x + 1) x^{-b - 2} (x^b - 1) (x^{b + 1} - 1),
	\end{align*}
	if $-p,-q \geq b+1,$ and
	\begin{equation*}
		\resizebox{\textwidth}{!}{%
			$\begin{aligned}
				&(x-1)^{-2}\left(-x^{b + p + q} - x^{b + p + q + 1} + x^{b - 1} + x^b + x^{p - 1}+ x^p + x^{q - 1} + x^q - x^{-2} - 2x^{-1} - 1\right)\\
				=&	(x-1)^{-2}(x+1)(-x^{b + p + q} + x^{b - 1} + x^{p - 1} + x^{q - 1} - x^{-2} - x^{-1}),
			\end{aligned}$
		}
	\end{equation*}
	if $b \geq -q.$ This implies that the Poincar{\'e} series for $\dim\ker f_{-a',p,b,q}$ is given by
	\begin{align*}
		(x-1)^{-2}x^{-b - 1} (x^b - 1) (x^{b + 1} - 1),
	\end{align*}
	if $-p,-q \geq b+1,$ and
	\begin{align*}
		&(x-1)^{-2}(-x^{p+q+b+1}+x^p+x^q+x^b-x^{-1}-1)\\
		=&(x-1)^{-2}\left( x^p(x^{-p-1} - 1) (x^{b +q+ 1} - 1) + x^q(x^b-1)(x^{-q}-1)\right) ,
	\end{align*}
	if $b \geq -q.$ Note that in this case, the kernel is the full Poincar{\'e} series, so every $\ta_1\ta_3$-generator is in the image of $f.$
\end{proof}
\begin{proposition}
	\label{k1all}
	The Poincar{\'e} series for the number of cohomology classes of the form 
	\begin{align*}
		\kappa_1^{-p}\frac{\theta_2}{x_2^{n_2}y_2^{m_2}}\frac{\theta_3}{x_3^{n_3}y_3^{m_3}}
	\end{align*}
	in $H_{Br,K}^{-a'+p\sigma+b\epsilon+q\sigma\otimes\epsilon}(pt)$ is given by
	\begin{align*}
		x^{p+b+q}(1+\cdots+x^{-p-b-2})(1+\cdots+x^{-p-q-2}).
	\end{align*}
more precisely by  the coefficient of $x^{a'}$ in the Poincar\'e series. Notice that in the case $p,b,q\leq -1$ then Proposition \ref{gen} says that all these cohomology classes are non-zero.
\end{proposition}

\begin{proof}
	We consider all elements of the form
	
	\begin{align*}
		\kappa_1^{-p}\frac{\theta_2}{x_2^{n_2}y_2^{m_2}}\frac{\theta_3}{x_3^{n_3}y_3^{m_3}}
	\end{align*}
	in $H_{Br,K}^{-a'+p\sigma+b\epsilon+q\sigma\otimes\epsilon}(pt).$ From the equations
	\begin{equation}
	\begin{aligned} \label{eq5}
		a' &= -p -m_2 - m_3 -4 \\
		b &= -p-n_2 -m_2 - 2\\
		q &= -p-n_3 -m_3 - 2
	\end{aligned}
	\end{equation}
	we obtain the equations
	\begin{equation}\label{eq6}
	\begin{aligned}
		m_3 &= -p-q-2-n_3\\
		m_2 &= q-a'-2+n_3\\
		n_2 &= -p-b-q+a'-n_3.
	\end{aligned}
	\end{equation}
	We then have the equivalences
	\begin{align*}
		n_2 \geq 0 & \iff n_3\leq -p-b-q+a'\\
		m_2 \geq 0  & \iff n_3 \geq a'+2-q\\
		m_3 \geq 0 & \iff n_3 \leq -p-q-2,
	\end{align*}
	from which we see that 
	\begin{align*}
		n_3 \in [ a'+2-q,-p-b-q+a'] \cap [0, -p-q-2].
	\end{align*}
	Then 
	\begin{align*}
		t_{a'}:=\# \text{products }
		&=  \left| [ a'+2-q,-p-b-q+a'] \cap [0, -p-q-2]\right|\\ &=\min(-p-b-q+a',-p-q-2)-\max(0,a'+2-q)+1.
	\end{align*}
	If the last expression is negative, then it should be interpreted as zero. Evaluating this expression, we obtain 
	\begin{align*}
				t_{a'}  =\begin{cases}
					a'-p-b-q+1& \text{if } p+b+q\leq a' \leq b-2,\\
					-p-q-1 & \text{if } b-2\leq a' \leq q-2,\\
				-p-1-a' & \text{if } q-2\leq a' \leq -p-4,\\
				\end{cases} 
			\end{align*}
			where we assumed $b \leq q$ without reducing generality because of the symmetry.
	We now see that by Lemma \ref{coeff} that $t_{a'}$ agrees with the coefficient of $x^{a'}$ in the Poincar{\'e} series 
	\begin{align*}
		x^{p+b+q}(1+\cdots+x^{-p-b-2})(1+\cdots+x^{-p-q-2}).
	\end{align*}
\end{proof}
\begin{proposition}
	\label{k1less}
	Let $p\leq -2, q,b\leq -1$. The Poincar{\'e} series for the number of cohomology classes of the form 
	\begin{align*}
		\kappa_1^{-p}\frac{\theta_2}{x_2^{n_2}y_2^{m_2}}\frac{\theta_3}{x_3^{n_3}y_3^{m_3}}
	\end{align*}
	in $H_{Br,K}^{-a'+p\sigma+b\epsilon+q\sigma\otimes\epsilon}(pt)$ with the property that $m_3,m_2 \leq -p-2$ is given by the coefficient $t_{a'}$ of $x^{a'}$	
	\begin{align*}
		x^{p-1}(x+\cdots+x^{-p-1})(1+\cdots+x^{-p-2}).
	\end{align*}
\end{proposition}
\begin{proof}
	We consider elements of the form
	
	\begin{align*}
		\kappa_1^{-p}\frac{\theta_2}{x_2^{n_2}y_2^{m_2}}\frac{\theta_3}{x_3^{n_3}y_3^{m_3}}
	\end{align*}
	in $H_{Br,K}^{a=-a'+p\sigma+b\epsilon+q\sigma\otimes\epsilon}(pt),$ where $m_3,m_2 \leq -p-2.$ From the equations \ref{eq5} and \ref{eq6} we have the equivalences
	\begin{align*}
		n_2 \geq 0 & \iff n_3\leq -p-b-q+a'\\
		m_2 \geq 0  & \iff n_3 \geq a'+2-q\\
		m_3 \geq 0 & \iff n_3 \leq -p-q-2\\
		m_2 \leq -p-2  & \iff n_3 \leq a'-p-q\\
		m_3 \leq -p-2 & \iff n_3 \geq -q,
	\end{align*}
	from which we see that 
	\begin{align*}
		m_2 \in [ a'+2-q,-p-q+a'] \cap [-q, -p-q-2].
	\end{align*}
	Then 
	\begin{align*}
		\# \text{products }
		&=  \left| [ a'+2-q,-p-q+a'] \cap [-q, -p-q-2]\right|\\ &=\min(-p-q+a',-p-q-2)-\max(-q,a'+2-q)+1.
	\end{align*}
	If the last expression is negative, then it should be interpreted as zero. Evaluating this expression, we obtain 
	
	\begin{align*}
				t_{a'}  =\begin{cases}
					a'-p+1& \text{if } p\leq a' \leq -2,\\
					-p-1 & \text{if } -2\leq a' \leq -2,\\
				-p-3-a' & \text{if } -2\leq a' \leq -p-4,\\
				\end{cases} 
			\end{align*}
	and we see that by Lemma \ref{coeff}, it agrees with the coefficient of $x^{a'}$ in the Poincar{\'e} series 
	\begin{align*}
		&x^p(1+\cdots+x^{-p-2})(1+\cdots+x^{-p-2})\\
		=&x^{p-1}(x+\cdots+x^{-p-1})(1+\cdots+x^{-p-2}).
	\end{align*}
\end{proof}
\bibliographystyle{plain}
 \bibliography{road5}
\end{document}